\documentclass[en,oneside,bbfont,article]{amsart} 
\usepackage[lmargin = 30mm, rmargin = 30mm, bmargin = 25mm, tmargin=25mm]{geometry}
\usepackage[utf8]{inputenc}

\usepackage{mathtools}
\DeclarePairedDelimiter\ceil{\lceil}{\rceil}

\providecommand{\examplename}{Example}

\newtheorem{theorem}{Theorem}[section]

\newtheorem{corollary}[theorem]{Corollary}
\newtheorem{lemma}[theorem]{Lemma}
\theoremstyle{remark}
\newtheorem{remarkx}[theorem]{Remark}
\newenvironment{remark}
    {\pushQED{\qed}\remarkx}
    {\popQED\endremarkx}
\theoremstyle{definition}
\newtheorem*{example*}{\protect\examplename}

\theoremstyle{definition}
\newtheorem{defx}[theorem]{Definition}
\newenvironment{definition}
    {\pushQED{\qed}\defx}
    {\popQED\enddefx}

\newtheorem*{assumption*}{Assumption}

\usepackage{mathrsfs}
\usepackage{amssymb}
\usepackage{amstext}
\usepackage{dsfont}
\usepackage{amsthm}
\usepackage{mleftright}
\usepackage{amsmath}
\mathtoolsset{showonlyrefs}
\usepackage{bm}
\usepackage{comment}
\usepackage{caption,enumerate,enumitem}
\usepackage[labelformat=simple]{subcaption}

\usepackage{scalefnt}
\usepackage{float,graphicx,verbatim}
\usepackage{tikz}
\usetikzlibrary{calc}
\usepackage{pgfplots}
\pgfplotsset{compat=1.16}
\RequirePackage[colorlinks=true,linkcolor=black,
	citecolor=blue,urlcolor=blue]{hyperref}

\usetikzlibrary{fadings,decorations.pathreplacing} 

\tikzfading[name=fade right,
  left color=transparent!0,  
  right color=transparent!100 
]

\usepackage[backend=biber,style=ieee,citestyle=numeric,sorting=nyt,maxbibnames=99,dashed=false,sortcites=true]{biblatex}
\renewbibmacro{in:}{}
\appto{\bibsetup}{\sloppy}
\usepackage[foot]{amsaddr}

\usepackage{todonotes}
\allowdisplaybreaks

\newcommand\N{\mathbb{N}}

\newcommand\calP{\mathcal{P}}
\newcommand\calT{\mathcal{T}}

\newcommand\R{\mathbb{R}}

\newcommand\E{\mathds{E}}

\newcommand\Ei{\mathrm{Ei}}

\newcommand\1{\mathds{1}}
\newcommand\Oh{\mathcal{O}}
\newcommand\oh{\mathrm{o}}

\newcommand\da{\downarrow}

\newcommand\eqd{\overset{d}{=}}

\newcommand\Var{\mathrm{Var}}

\newcommand\nf[1]{\normalfont{#1}}

\newcommand{\D}{\,\mathrm{d}}
\newcommand{\ov}[1]{\overline{#1}}

\newcommand{\wt}[1]{\widetilde{#1}}
\newcommand{\wh}[1]{\widehat{#1}}

\usepackage[foot]{amsaddr}

\title{Generalised Random Parking for Trapeziums on a Strip}
\author{David Kramer-Bang$^{\dag}$ \& Stjepan \v{S}ebek$^*$}

\address{$^\dag$Aarhus University, Department of Mathematics, DK}
\address{$^{*}$University of Zagreb Faculty of Electrical Engineering and Computing, HR}

\email{bang@math.au.dk}
\email{stjepan.sebek@fer.unizg.hr}

\begin{document}

\begin{abstract}
In this article, we study a generalisation of R\'{e}nyi’s car-parking problem in which isosceles trapeziums are sequentially deposited on a strip. We derive an explicit formula for the parking constant in terms of the lengths of the two bases, recovering the classical rectangular and triangular models as special cases. We further obtain quantitative finite-size asymptotics for both the expected number and the variance of deposited trapeziums, with convergence rates that depend explicitly on the geometry of the deposited particle. In particular, although the recursive construction involves two different substrate geometries, their variances have the same leading asymptotic density. Finally, we show that the parking constant depends non-monotonically on the ratio of the two base lengths and possesses a unique minimiser, so that the least efficient shape is a genuine trapezium rather than a triangle.

\end{abstract}


\maketitle

\section{Introduction}\label{sec:intro}
Many processes in physical chemistry and biology exhibit irreversible dynamics of the following type: particles of finite size arrive randomly in a bounded $d$-dimensional region and are adsorbed, provided they do not overlap with already placed particles. The procedure continues until no further particles can be accommodated, in which case we say that saturation or jamming has been achieved. This mechanism is known under several names, including random sequential adsorption~\cite{Evans_1993}, random sequential packing~\cite[\S 1.10]{Hall_1988}, random parking~\cite{Renyi_OG}, sequential inhibition~\cite[\S 5.6]{Diggle_2003}, and irreversible monomer filling~\cite{Wolf_Evans_Hoffman_1984}. The terminology depends on the context. The particles may be called atoms, cars, or monomers, and the region is a surface or a substrate. The models most thoroughly investigated in the literature are either lattice-based or continuum one-dimensional systems, where the substrate is an interval, and the particles are line segments of fixed length arriving uniformly at random. In higher dimensions, the particles are typically spheres, cubes, or related shapes~\cite{Viot_Tarjus_Ricci_Talbot_1992}.

Two classical models are the model of the attachment of pendant groups in a polymer chain by Paul Flory~\cite{Flory_1939}, and the car-parking model by Alfr\'{e}d R\'{e}nyi~\cite{Renyi_OG}. The essential difference is that Flory’s model is defined on a discrete substrate, whereas R\'{e}nyi’s is defined on a continuous substrate. Since our work is a generalisation of R\'{e}nyi’s model, we recall its formulation in more detail. Fix $x>0$, representing the length of the kerb, i.e.\ the interval $[0,x]$. Cars are unit-length intervals. At each step, a uniform random variable on $(0,x)$ is generated, representing the left endpoint of a new car. If this car overlaps with an already parked one (or exceeds the boundaries of the parking lot), it is discarded; otherwise, it parks. The process continues until every remaining gap is shorter than one, at which point jamming occurs.

One of the most important numerical characteristics of all the random sequential adsorption models is the so-called jamming limit (parking constant). The jamming limit is the expected occupied density of the substrate as the size of the substrate tends to infinity. It can be studied experimentally~\cite{marvel1938structure}, via simulations~\cite{nord1991irreversible, tory1983simulation, wang2000series, krizanc2009, krapivsky2023jamming}, or analytically~\cite{Flory_1939, Renyi_OG, krapivsky2010kinetic, fan-percus, baram-kutasov, dining_table, krizanc2009, DPSZ-ladder, triangles}. While analytic expressions are known for some special cases, most results in the literature rely on approximations from simulations. Explicit formulas are usually limited to one-dimensional models. In higher dimensions, analytic approaches are rarely feasible. Related to R\'enyi’s model, Pal\'asti conjectured~\cite{palasti1960some} that the coverage of aligned $d$-dimensional hypercubes satisfies $\theta_d=\theta_R^d$ (where $\theta_R$ is the parking constant from the one-dimensional R\'enyi's model; see~\eqref{eq:renyi_const}), which sparked debate, including an erroneous “proof” and subsequent critiques~\cite{palasti1960some,weiner1978sequential,weiner1980alternative,hori1979weiner,tory1979some,pickard1980critique}.

In this paper, we investigate an extension of the one-dimensional setting, where the substrate is a strip as in Figure~\ref{fig:strip} (see also Figure~\ref{fig:shapes_of_strips}). This direction was initiated in~\cite{triangles}. The motivation for such generalisation comes from discrete analogues. Flory’s original model~\cite{Flory_1939} is a one-dimensional lattice adsorption model, later reformulated in various equivalent ways, such as a discrete version of Rényi’s parking problem~\cite{GerinPRparkingHAL, Page}, model of dimer deposition on a linear substrate~\cite{krapivsky2010kinetic}, unfriendly seating arrangement~\cite{Rothman_sol}, or Rydberg atom excitation model~\cite{krapivsky2020large}. Generalisations of Flory’s model have then been studied on ladder graphs (two-row square ladders) in~\cite{baram-kutasov, fan-percus, dining_table, DPSZ-ladder, krizanc2009}, and those works inspired the analogous generalisation in the continuous setting.

\begin{figure}
    \centering
    \begin{tikzpicture}[scale = 0.65]
	\draw (0, 0) -- (10, 0);
	\draw (0, 1) -- (10, 1);
        \draw (0, 0) -- (0, 1);
        \draw (10, 0) -- (10, 1);
        \draw[decorate,decoration={brace,amplitude=5pt,raise=0pt,mirror},xshift=0pt, yshift = -3pt] (0, 0) -- (10, 0) node [midway,yshift=-13pt,xshift=0pt]{$x$};
        \draw[decorate,decoration={brace,amplitude=5pt,raise=0pt,mirror},xshift=3pt, yshift = 0pt] (10, 0) -- (10, 1) node [midway,yshift=0pt,xshift=13pt]{$1$};
    \end{tikzpicture}
    \caption{Strip of length $x$ and height $1$.}
    \label{fig:strip}
\end{figure}
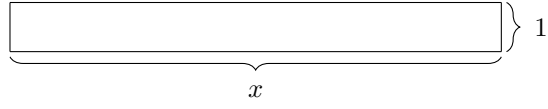

In~\cite{triangles}, the substrate is precisely the strip shown in Figure~\ref{fig:strip}. The deposited particles are isosceles triangles (see Figure~\ref{fig:deposited_triangles}), and the jamming limit is computed explicitly.
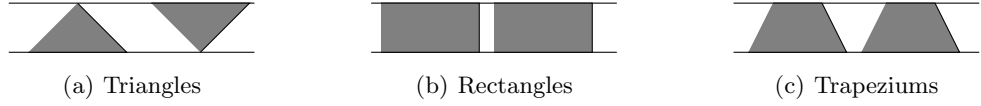
\begin{figure}
\centering
\begin{subfigure}{0.3\textwidth}
    \centering
    \begin{tikzpicture}[scale = 0.65]
	    \draw (0, 0) -- (5, 0);
	    \draw (0, 1) -- (5, 1);
        \filldraw[fill=gray] (0.4, 0) -- (2.4,0) -- (1.4,1);
        \filldraw[fill=gray] (2.9, 1) -- (4.9,1) -- (3.9,0);
    \end{tikzpicture}
    \caption{Triangles}
    \label{fig:deposited_triangles}
\end{subfigure}
\begin{subfigure}{0.3\textwidth}
    \centering
    \begin{tikzpicture}[scale = 0.65]
	    \draw (0, 0) -- (5, 0);
        \draw (0, 1) -- (5, 1);
        \filldraw[fill=gray] (0.2, 0) -- (2.2,0) -- (2.2,1) -- (0.2, 1);
        \filldraw[fill=gray] (2.5, 0) -- (4.5,0) -- (4.5,1) -- (2.5, 1);
    \end{tikzpicture}
    \caption{Rectangles}
    \label{fig:deposited_rectangles}
\end{subfigure}
\begin{subfigure}{0.3\textwidth}
    \centering
    \begin{tikzpicture}[scale = 0.65]
	   \draw (0, 0) -- (5, 0);
        \draw (0, 1) -- (5, 1);
        \filldraw[fill=gray] (0.3, 0) -- (2.3,0) -- (1.8,1) -- (0.8, 1);
        \filldraw[fill=gray] (2.6, 0) -- (4.6,0) -- (4.1,1) -- (3.1, 1);
    \end{tikzpicture}
    \caption{Trapeziums}
    \label{fig:deposited_trapeziums}
\end{subfigure}
\caption{Different shapes of deposited objects.}
\end{figure}
Note that if rectangles were deposited instead (see Figure~\ref{fig:deposited_rectangles}), the problem would reduce to Rényi’s classical one-dimensional model. Hence, the precise shape of the jamming limit (in this setting where objects are deposited on a substrate in the shape of a strip) is known in the extreme cases of triangles and rectangles. Specifically, Rényi~\cite{Renyi_OG} showed that in the one-dimensional car-parking problem, the expected coverage is
\begin{equation}\label{eq:renyi_const}
	\theta_R = \int_0^{\infty} \exp \mleft( -2 \int_0^x \frac{1 - e^{-y}}{y}dy \mright)dx \approx 0.7475979,
\end{equation}
and for the isosceles triangles, the jamming limit from~\cite{triangles} is given by
\begin{equation}\label{eq:triangle_const}
    \eta_\triangle = \frac{1}{2} \int_0^{\infty} \mleft( 3 - e^{-2t} \mright) \exp \mleft( -\int_0^t \frac{2 - e^{-u} - e^{-2u}}{u} du \mright) dt \approx 0.662939. 
\end{equation}
In the present paper, we interpolate between these two cases by considering isosceles trapeziums (see Figure~\ref{fig:deposited_trapeziums}). As already mentioned, our substrate is a strip. Once we start depositing trapeziums, the relevant substrates are instead the two strip geometries introduced below (see Figure~\ref{fig:shapes_of_strips}). Trapeziums can arrive on the strip in two orientations -- with the longer base on the upper side of the strip, and with the longer base on the lower side of the strip. The precise position and orientation of the newly arrived trapezium is fully determined with its left-most point (i.e.\ with the left-most point of its longer base). Hence, we only need to sample one point, either from the upper side, or from the lower side of the strip, and we know exactly how to park a trapezium at that point. Let the sum of the lengths of the upper side and the lower side of the substrate be $2x$, where the upper one is of length $y \in (0, 2x)$, and the lower one has length $2x-y$. At each step a uniform random variable on $(0, 2x)$ is generated, representing this left endpoint of a new trapezium, where values sampled between $0$ and $y$ imply that we try to park the new trapezium so that its longer base is on the upper side of the substrate, and values sampled between $y$ and $2x$ imply that we try to park the new trapezium so that its longer base is on the lower side of the substrate. If this trapezium overlaps with an already parked one (or exceeds the boundaries of the parking lot), it is discarded; otherwise, it is parked at that position. The process continues until every remaining gap can no longer fit an additional trapezium, and we then say that we have reached the jammed state. This procedure is described in full detail in Definition~\ref{defn:UDS_} below. As our main result, we determine the jamming limit for the whole family of isosceles trapeziums, recovering both known results,~\eqref{eq:renyi_const} and~\eqref{eq:triangle_const}, as special cases.

Beyond just determining the jamming limit for trapeziums, we also strongly refine the asymptotic analysis from~\cite{triangles}. In our analysis, we follow closely the method by Dvoretzky and Robbins~\cite{Dvoretzky_Robbins}. Let us denote by $\mu(x)$ the expected number of parked cars in the jammed state of R\'enyi's model, when the parking has length $x$. Dvoretzky and Robbins~\cite{Dvoretzky_Robbins} proved, 
\begin{equation}
	\mu(x) = \theta_R x + (\theta_R - 1) + \Oh \mleft( \mleft( \frac{2e}{x} \mright)^{x - 3/2} \mright), \quad \text{as }x \to \infty.
\end{equation}
Earlier, R\'{e}nyi~\cite{Renyi_OG} had established a weaker version with an error term of order $\Oh(x^{-n})$ for any $n$, and a similar result was obtained by Ney~\cite{ney-thesis}. As R\'{e}nyi reports~\cite{Renyi_OG}, N.\ G.\ de Bruijn suggested that sharper asymptotics could be derived by adapting his method for certain differential-difference equations in number theory~\cite{debruijn}. This approach is closely related to the method later used by Dvoretzky and Robbins. An even sharper estimate was obtained in the thesis of Beenakker~\cite[Ch.~IX]{beenakker}, under de Bruijn’s supervision, using a more involved technique, where the author shows, 
\begin{equation*}
	\mu(x) = \theta_R x + (\theta_R - 1) + \Oh \mleft( \mleft( \frac{2e}{x\log x} \mright)^{x - 3/2} \mright), \quad \text{as }x \to \infty.
\end{equation*}

Dvoretzky and Robbins~\cite{Dvoretzky_Robbins} also obtained a corresponding second-order asymptotic result for the fluctuations of the parking process. More precisely, writing $\sigma^2(x)\coloneqq \Var(N_x)$, where $N_x$ denotes the number of cars in the jammed state, they proved that there exists a constant $\lambda_2>0$ such that, 
\begin{equation}
\sigma^2(x)
=
\lambda_2 x+\lambda_2
+
\Oh\mleft(
\mleft(\frac{4e}{x}\mright)^{x-4}
\mright), \quad \text{as }x \to \infty.
\end{equation}
Thus, in addition to identifying the linear growth of the variance, their result gives a super-exponentially small remainder after the constant-order correction. Below, we establish an analogue of this second-order estimate for the trapezium parking model. In particular, we show that the variance grows linearly for both of the substrate types arising from the recursive decomposition, with the same asymptotic variance density, and obtain explicit rates which depend on the geometry of the deposited trapezium.

In our setting, the classical result of Dvoretzky and Robbins~\cite{Dvoretzky_Robbins} is recovered when the deposited trapezium (with base lengths $a$ and $b$, see Section~\ref{sec:prem_park_prob}) reduces to a square. More generally, when $a=b$, horizontal rescaling reduces the model to the classical one-dimensional parking problem, and our estimates make explicit how the finite-size error depends on the common width $a$. For genuine trapeziums, $a>b$, a different asymptotic regime appears. In particular, the distinction between $a=b$ and $a>b$ changes the polynomial correction in the Dvoretzky--Robbins-type remainder, while the length $a$ of the longer base determines the natural scale $x/a$ entering its exponent. Thus, the quantitative bounds exhibit a geometry-dependent polynomial correction, and not merely as a multiplicative constant in front of an otherwise universal error term. In this sense, our asymptotic results go beyond a direct extension of~\cite{Dvoretzky_Robbins} from unit intervals to a larger class of particles.

We also obtain a corresponding second-order theory. The recursive geometry of the parking procedure naturally gives rise to two different auxiliary substrate types, a parallelogram and a trapezium, denoted below by $\calP_x$ and $\calT_x$ respectively (see Definition~\ref{defn_tarp_parll} and Figure~\ref{fig:shapes_of_strips}). These are intrinsic to the model, since, after the first accepted deposition, the remaining region splits into two smaller substrates which are again of these two types. Throughout, we denote the number of trapeziums deposited on a parallelogram $\calP_x$ by $P_x$ and the number trapeziums of deposited on a trapezium $\calT_x$ by $T_x$, see~\eqref{eq:defn:P_x,T_x} below for a rigorous definition of the random variables $P_x$ and $T_x$. Although the corresponding local recursions are different, we prove that their fluctuations have the same leading-order behaviour. More precisely, there exists a constant $\lambda_2(a,b)>0$ such that both variances have the common asymptotic expansion
\[
\Var(P_x)
=
\lambda_2(a,b)\mleft(x+\frac{a+b}{2}\mright)
+\oh(1),
\qquad
\Var(T_x)
=
\lambda_2(a,b)\mleft(x+\frac{a+b}{2}\mright)
+\oh(1),
\]
and we in fact obtain explicit Dvoretzky--Robbins-type bounds for the two remainder terms. In the square case $a=b=1$, this recovers the second-order asymptotic estimate of~\cite{Dvoretzky_Robbins}. For genuine trapeziums, the rate again changes in a shape-dependent way. The fact that the two substrate geometries have different finite-volume recursions but nevertheless share the same asymptotic variance density is, in particular, a useful indication that their boundary geometry disappears at the level of the leading bulk fluctuations.

General stabilisation theory predicts linear-order fluctuations and Gaussian limits for broad classes of random sequential adsorption (RSA) models~\cite{MR2291806,MR1890065}. These results provide important general context for the fluctuation behaviour of RSA processes, but their hypotheses and geometric framework do not directly yield a central limit theorem (CLT) for the quasi-one-dimensional trapezium parking model considered here. By contrast, the recursive structure of the present model permits much finer Dvoretzky–Robbins-type finite-size estimates for the mean and variance, including explicit super-exponential error bounds. In the classical case $a=b$, corresponding to the deposition of squares, a CLT is already known from Dvoretzky and Robbins~\cite{Dvoretzky_Robbins}. For genuine trapeziums $0\le b<a$, however, establishing asymptotic normality of $P_x$ and $T_x$ remains open. Proving such a result would require either an adaptation of existing stabilisation methods to the present setting or a different argument exploiting the recursive decomposition developed in this paper. We therefore leave the CLT, as well as quantitative normal-approximation bounds, for future work; see Section~\ref{sec:open_problems}.

Finally, the explicit formula for the parking constant reveals a geometric phenomenon that is invisible from the two previously known endpoint cases. By scaling, the coverage density $\xi(a,b)$ depends only on the ratio $b/a$, so the family may at first appear to provide a simple interpolation between triangles, corresponding to $b/a=0$, and rectangles, corresponding to $b/a=1$. This intuition is false. We prove that $r\mapsto \xi(1,r)$ for $r \in [0,1]$, is not monotone and, moreover, possesses a unique minimiser. Consequently, the least efficient particle in this family is a genuine trapezium rather than a triangle, despite the rectangle having the largest parking density among the two endpoint geometries. This non-monotonicity arises from the possibility, absent in the triangular case, of placing neighbouring trapeziums with their longer bases on the same side of the strip, thereby creating relatively large unusable gaps. Thus, rather than merely interpolating between the classical rectangular and triangular models, the trapezium family exhibits a genuinely new shape-dependent packing phenomenon.

Taken together, our results provide an exact first-order jamming law for the entire family of isosceles trapeziums, quantitative finite-size asymptotics, a common linear variance law for the two recursive substrate geometries, and a non-trivial optimisation of the parking density over the shape of the deposited particle.

\subsection{Preliminaries of Parking Problem}\label{sec:prem_park_prob}
Throughout this paper, we will always consider \emph{isosceles trapeziums}, so from this point onward, we will, for the sake of brevity, write \emph{trapezium}. Moreover, we will always let $0 \le b \le a<\infty$, such that $a>0$, throughout the paper, where $a$ and $b$ are the lengths of the bases of the trapeziums we are depositing in our model.

Next, we define the parallelogram $\calP_x$ and trapezium $\calT_x$ as shown in Figure~\ref{fig:shapes_of_strips}.

\begin{definition}\label{defn_tarp_parll}
    Define $\calT_x$ as the unique trapezium of height $1$, and bases of length $x+(a-b)/2$ and $x-(a-b)/2$, where the longer base is on the lower side. Similarly, we define $\calP_x$ as the unique parallelogram with two (parallel) sides of length $x$, height $1$ between those two sides, and acute angle $\vartheta=\arctan(2/(a-b))$ whenever $a>b$, and $\vartheta=\pi/2$ if $a=b$, where the left-most point is on the upper line.
\end{definition}

Note that $\calP_x$ and $\calT_x$ both have area $x$, and that the lengths of the upper and lower sides of $\calP_x$ (resp. $\calT_x$) sum to $2x$. Next, we will rigorously define the \emph{uniform depositing algorithm} ($\mathcal{UDA}$), used to deposit the trapezium $\calT_{(a+b)/2}$ on the strips $\calP_x$ and $\calT_x$. In the introduction, we presented short and intuitive descriptions of algorithms in the background of R\'enyi's car parking problem, and our parking model. However, to be able to find the exact value of the jamming limit, one has to be much more careful and precise. We follow the idea employed in~\cite{Renyi_OG} and~\cite{konheim1962parking} that was used to find R\'enyi's car parking constant by first developing a delay-differential relation for the expected value of the number of parked cars on a parking lot of a fixed size. The same idea was used in~\cite{triangles}. Even though we wrote that in R\'enyi's car parking problem one generates a uniform random variable on $(0, x)$, we clearly cannot park a car if we sample a point in $(x-1, x)$. This plays an important role, and has to be taken into account once one starts proving things. Clearly, in our model, we also won't be able to park trapeziums regardless of where we sample our point in the interval $(0, 2x)$. Notice that the situation is much more involved in our case since we have substrates of different shape, and our point can be chosen on two different sides of the substrate. We now carefully describe our uniform depositing algorithm.

\begin{figure}[ht]
\centering
\begin{subfigure}{0.49\textwidth}
    \centering
    \begin{tikzpicture}[scale = 0.65]
        \draw (1, 0) -- (10, 0);
	\draw (0, 1) -- (9, 1);
        \draw (1, 0) -- (0, 1);
        \draw (10, 0) -- (9, 1);
        \draw[decorate,decoration={brace,amplitude=5pt,raise=0pt,mirror},xshift=0pt, yshift = -3pt] (1, 0) -- (10, 0) node [midway,yshift=-13pt,xshift=0pt]{$x$};
        \draw[decorate,decoration={brace,amplitude=5pt,raise=0pt,mirror},xshift=0pt, yshift = 3pt] (9, 1) -- (0, 1) node [midway,yshift=13pt,xshift=0pt]{$x$};

        \draw (9.5, 0.5) arc[start angle=135, end angle=180, radius=0.7];
        \node at (9.1, 0.3) {$\vartheta$};
    \end{tikzpicture}
    \caption{$\calP_x$}
    \label{fig:strip_parallelogram}
\end{subfigure}
\begin{subfigure}{0.49\textwidth}
    \centering
    \begin{tikzpicture}[scale = 0.65]
        \draw (0, 0) -- (10, 0);
	\draw (1, 1) -- (9, 1);
        \draw (0, 0) -- (1, 1);
        \draw (10, 0) -- (9, 1);
        \draw[decorate,decoration={brace,amplitude=5pt,raise=0pt,mirror},xshift=0pt, yshift = -3pt] (0, 0) -- (10, 0) node [midway,yshift=-13pt,xshift=0pt]{$x+\frac{a-b}{2}$};
        \draw[decorate,decoration={brace,amplitude=5pt,raise=0pt,mirror},xshift=0pt, yshift = 3pt] (9, 1) -- (1, 1) node [midway,yshift=13pt,xshift=0pt]{$x-\frac{a-b}{2}$};
        \draw (9.5, 0.5) arc[start angle=135, end angle=180, radius=0.7];
        \node at (9.1, 0.3) {$\vartheta$};
    \end{tikzpicture}
    \caption{$\calT_x$}
    \label{fig:strip_trapezium}
\end{subfigure}
\caption{Illustrations of the parallelogram $\calP_x$ and trapezium $\calT_x$, which appear in the scheme of depositing trapeziums
.}
\label{fig:shapes_of_strips}
\end{figure}
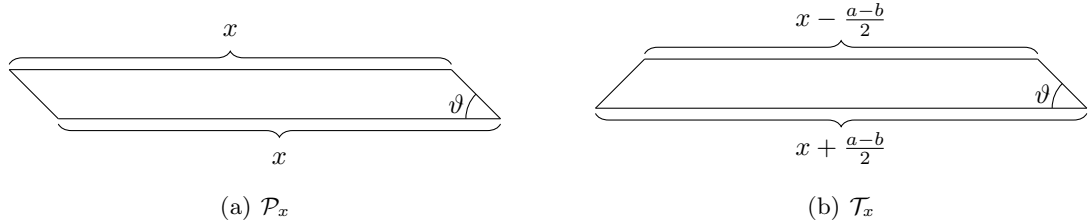

\begin{definition}[Definition of $\mathcal{UDA}$]\label{defn:UDS_}
\textbf{Depositing on $\calP_x$:} The following scheme defines how we deposit the trapezium $\calT_{(a+b)/2}$ on $\calP_x$. Let $x>a$, and let $\tau \sim \mathcal{U}(0,2x-2a).$ 
If $\tau \in (0,x-a)$, we deposit the left-most point of $\calT_{(a+b)/2}$, $\tau$ into the upper baseline of $\calP_x$. If $\tau \in [x-a,2x-2a)$, we deposit the left-most point of $\calT_{(a+b)/2}$, $y=\tau+a-x$ into the lower baseline of $\calP_x$. (Illustrated in Figure~\ref{fig:trapezium_arriving_possible_strips}(a) \&~(b).)

\textbf{Depositing on $\calT_x$:} The following scheme defines how we deposit the trapezium $\calT_{(a+b)/2}$ on $\calT_x$. To do so, we will split into two cases: (i) if $(a+b)/2<
x< (3a-b)/2$, and (ii) $x \ge (3a-b)/2$. 
\begin{itemize}[leftmargin=3em]
    \item[{\textbf{(i)}}] Assume that $(a+b)/2
< x< (3a-b)/2$, in which case we can only deposit on the lower baseline of $\calT_x$. Hence, let $\tau \sim \mathcal{U}(0,x-(a+b)/2)$, and deposit $\calT_{(a+b)/2}$, $\tau$ into the lower baseline of $\calT_x$. (Illustrated in Figure~\ref{fig:trapezium_arriving_possible_strips}(e).)
\item[{\textbf{(ii)}}] Let $x\ge (3a-b)/2$, and $\tau \sim \mathcal{U}(0,2x-2a)$. If $\tau \in (0,x-(3a-b)/2)$, we deposit the left-most point of $\calT_{(a+b)/2}$, $\tau$ into the upper baseline of $\calT_x$. If $\tau \in [x-(3a-b)/2,2x-2a)$, we deposit the left-most point of $\calT_{(a+b)/2}$, $y=\tau-(x-(3a-b)/2)$ into the lower baseline of $\calT_x$. (Illustrated in Figure~\ref{fig:trapezium_arriving_possible_strips}(c) \&~(d).)
\end{itemize}

\textbf{Algorithm:} When a trapezium $\calT_{(a+b)/2}$ has been deposited on either $\calP_x$ or $\calT_x$, there will be two leftover strips. Conditional on the first deposition, the parking procedures on these two strips are defined to be independent copies of the $\mathcal{UDA}$, with the corresponding strip types and sizes. 
This procedure is continued recursively until there are no strips remaining which can fit $\calT_{(a+b)/2}$.
\end{definition}
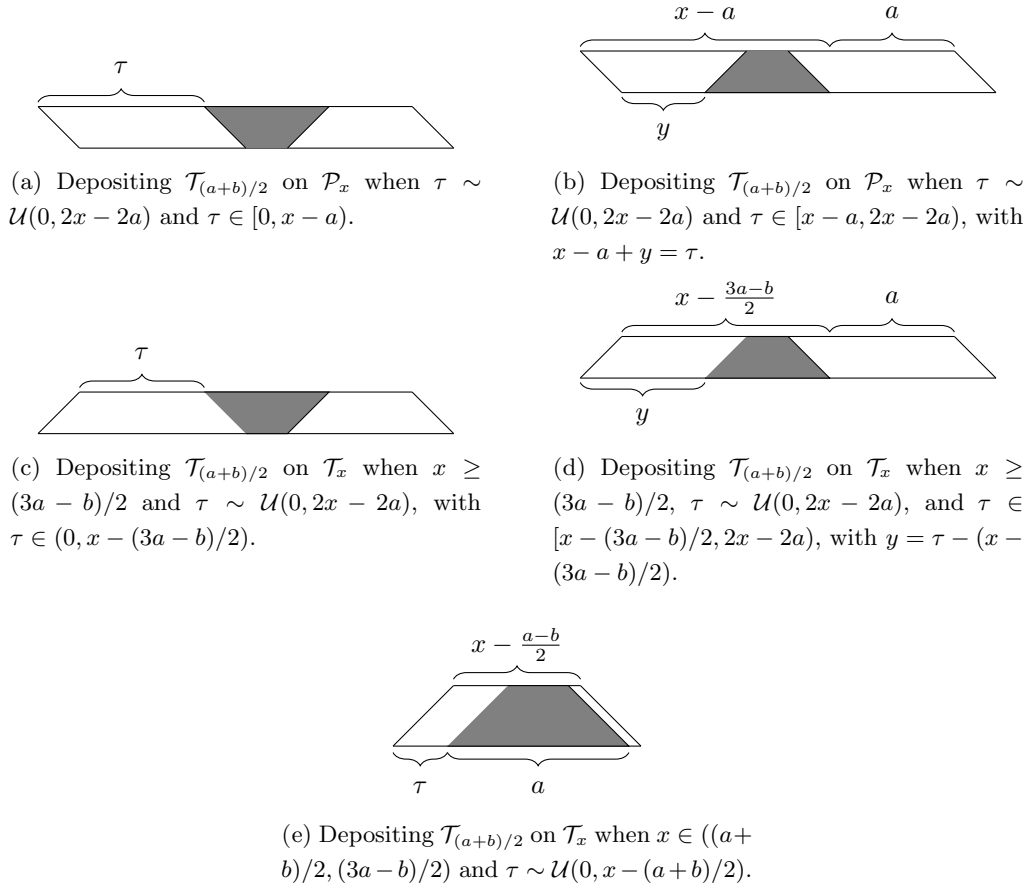
\begin{figure}[!ht]
\begin{subfigure}[t]{0.4\textwidth}
    \centering
    \begin{tikzpicture}[scale = 0.55]
        \draw (1, 0) -- (10, 0);
	\draw (0, 1) -- (9, 1);
        \draw (1, 0) -- (0, 1);
        \draw (10, 0) -- (9, 1);
	\filldraw[fill=gray] (6, 0) -- (7,1) -- (4,1) -- (5, 0);
        \draw[decorate,decoration={brace,amplitude=5pt,raise=0pt,mirror},xshift=0pt, yshift = 3pt] (4, 1) -- (0, 1) node [midway,yshift=13pt,xshift=0pt]{$\tau$};
    \end{tikzpicture}
    \caption{Depositing $\calT_{(a+b)/2}$ on $\calP_x$ when $\tau\sim \mathcal{U}(0,2x-2a)$ and $\tau \in [0,x-a)$.}
    \label{fig:parallelogram_up_defn}
\end{subfigure} \hspace{2em}
\begin{subfigure}[t]{0.4\textwidth}
    \centering
    \begin{tikzpicture}[scale = 0.55]
        \draw (1, 0) -- (10, 0);
	\draw (0, 1) -- (9, 1);
        \draw (1, 0) -- (0, 1);
        \draw (10, 0) -- (9, 1);
	\filldraw[fill=gray] (5, 1) -- (6,0) -- (3,0) -- (4, 1);
        \draw[decorate,decoration={brace,amplitude=5pt,raise=0pt,mirror},xshift=0pt, yshift = -3pt] (1, 0) -- (3, 0) node [midway,yshift=-13pt,xshift=0pt]{$y$};
        \draw[decorate,decoration={brace,amplitude=5pt,raise=0pt,mirror},xshift=0pt, yshift = 3pt] (9, 1) -- (6, 1) node [midway,yshift=13pt,xshift=0pt]{$a$};
        \draw[decorate,decoration={brace,amplitude=5pt,raise=0pt,mirror},xshift=0pt, yshift = 3pt] (6, 1) -- (0, 1) node [midway,yshift=13pt,xshift=0pt]{$x-a$};
    \end{tikzpicture}
    \caption{Depositing $\calT_{(a+b)/2}$ on $\calP_x$ when $\tau\sim \mathcal{U}(0,2x-2a)$ and $\tau \in [x-a,2x-2a)$, with $x-a+y=\tau$.}
    \label{fig:parallelogram_down_defn}
\end{subfigure}
\vspace{1em}
\begin{subfigure}[t]{0.4\textwidth}
    \centering
    \begin{tikzpicture}[scale = 0.55]
        \draw (0, 0) -- (10, 0);
	\draw (1, 1) -- (9, 1);
        \draw (0, 0) -- (1, 1);
        \draw (10, 0) -- (9, 1);
	\filldraw[fill=gray] (5, 0) -- (6, 0) -- (7,1) -- (4,1);
        \draw[decorate,decoration={brace,amplitude=5pt,raise=0pt,mirror},xshift=0pt, yshift = 3pt] (4, 1) -- (1, 1) node [midway,yshift=13pt,xshift=0pt]{$\tau$};
    \end{tikzpicture}
    \caption{Depositing $\calT_{(a+b)/2}$ on $\calT_x$ when $x \ge (3a-b)/2$ and $\tau\sim \mathcal{U}(0,2x-2a)$, with $\tau \in (0,x-(3a-b)/2)$.}
    \label{fig:trapeze_up_defn}
\end{subfigure} \hspace{2em}
\begin{subfigure}[t]{0.4\textwidth}
    \centering
    \begin{tikzpicture}[scale = 0.55]
        \draw (0, 0) -- (10, 0);
	\draw (1, 1) -- (9, 1);
        \draw (0, 0) -- (1, 1);
        \draw (10, 0) -- (9, 1);
	\filldraw[fill=gray] (4, 1) -- (5,1) -- (6,0) -- (3, 0);
        \draw[decorate,decoration={brace,amplitude=5pt,raise=0pt,mirror},xshift=0pt, yshift = -3pt] (0, 0) -- (3, 0) node [midway,yshift=-13pt,xshift=0pt]{$y$};
        \draw[decorate,decoration={brace,amplitude=5pt,raise=0pt,mirror},xshift=0pt, yshift = 3pt] (6, 1) -- (1, 1) node [midway,yshift=13pt,xshift=0pt]{$x-\frac{3a-b}{2}$};
        \draw[decorate,decoration={brace,amplitude=5pt,raise=0pt,mirror},xshift=0pt, yshift = 3pt] (9, 1) -- (6, 1) node [midway,yshift=13pt,xshift=0pt]{$a$};
    \end{tikzpicture}
    \caption{Depositing $\calT_{(a+b)/2}$ on $\calT_x$ when $x \ge (3a-b)/2$, $\tau\sim \mathcal{U}(0,2x-2a)$, and $\tau \in [x-(3a-b)/2,2x-2a)$, with $y=\tau -(x-(3a-b)/2)$.}
    \label{fig:trapeze_down_defn}
\end{subfigure}
\vspace{1em}
\begin{subfigure}[t]{0.4\textwidth}
    \centering
    \begin{tikzpicture}[scale = 0.8]
        \draw (3.6, 0) -- (7.7, 0);
	\draw (4.6, 1) -- (6.7, 1);
        \draw (3.6, 0) -- (4.6, 1);
        \draw (7.7, 0) -- (6.7, 1);
	\filldraw[fill=gray] (5.5, 1) -- (6.5,1) -- (7.5,0) -- (4.5, 0);
        \draw[decorate,decoration={brace,amplitude=5pt,raise=0pt,mirror},xshift=0pt, yshift = -3pt] (3.6, 0) -- (4.5, 0) node [midway,yshift=-13pt,xshift=0pt]{$\tau$};
        \draw[decorate,decoration={brace,amplitude=5pt,raise=0pt,mirror},xshift=0pt, yshift = -3pt] (4.5, 0) -- (7.5, 0) node [midway,yshift=-13pt,xshift=0pt]{$a$};
        \draw[decorate,decoration={brace,amplitude=5pt,raise=0pt,mirror},xshift=0pt, yshift = 3pt] (6.7, 1) -- (4.6, 1) node [midway,yshift=13pt,xshift=0pt]{$x-\frac{a-b}{2}$};
    \end{tikzpicture}
    \caption{Depositing $\calT_{(a+b)/2}$ on $\calT_x$ when $x \in ((a+b)/2,(3a-b)/2)$ and $\tau \sim \mathcal{U}(0,x-(a+b)/2)$.}
    \label{fig:trapeze_down_defn__}
\end{subfigure}
\caption{Illustration of all the possible scenarios of the trapezium $\calT_{(a+b)/2}$ arriving on strips $\calP_x$ and $\calT_x$, defined by the scheme $\mathcal{UDA}$ in Definition~\ref{defn:UDS_}.}
\label{fig:trapezium_arriving_possible_strips}
\end{figure}


\begin{remark}
\noindent{\nf{\bf{(i)}}} Note that $x \ge (3a-b)/2$ is necessary to be able to deposit anything on the shorter, upper line in the case of trapezium $\calT_x$. 

\smallskip

\noindent{\nf{\bf{(ii)}}} We note that after depositing a trapezium $\calT_{(a+b)/2}$ on the lower baseline of a parallelogram $\calP_x$, we end up with two strips, where one is a trapezium $\calT_y'$ (for some appropriate $y$), where $\calT_y'$ is flipped upside down so the longer baseline is on top. For such objects, we note that the law of depositing a new trapezium $\calT_{(a+b)/2}$ on $\calT_y'$ is the same as depositing on $\calT_y$, i.e.\ as if the trapezium is flipped. Hence, we will do this without further note throughout the paper for simplicity. A similar case can happen for a flipped $\calP_x$, for which the same arguments hold. 

\smallskip

\noindent{\nf{\bf{(iii)}}}
By definition of the recursive $\mathcal{UDA}$, whenever the first deposited trapezium splits the current substrate into two child strips, the subsequent parking procedures on the two strips are carried out using independent copies of the depositing randomness. Consequently, conditional on the position and orientation of the first deposited trapezium, the eventual numbers of trapeziums deposited in the two child strips are independent, with distributions determined solely by the types and sizes of the respective child strips. We will use this conditional independence in the second-moment and variance recursions
below.

\smallskip

\noindent{\nf{\bf{(iv)}}}
A useful structural feature of the model is that the recursive decomposition closes within only two substrate types. As illustrated in Figure~\ref{fig:trapezium_arriving}, after the first accepted deposition on either $\calP_x$ or $\calT_x$, each of the two remaining components is, up to the reflection convention in~{\nf{\bf{(ii)}}}, again either a parallelogram of the form $\calP_y$ or a trapezium of the form $\calT_y$, for an appropriate $y\ge 0$. Thus, no additional substrate geometries are generated by the recursion.

This closure property is the basic reason that the model remains analytically tractable. It yields a closed system of recursive equations for the first moments, and the same two-type branching structure persists at the level of second moments. In particular, Figure~\ref{fig:trapezium_arriving} should be viewed not only as a geometric illustration of the possible depositions, but also as the recursive structure underlying the expectation and variance calculations developed below.
\end{remark}

For all $x>0$ we define the corresponding random variables $P_x$ and $T_x$ as
\begin{align}\label{eq:defn:P_x,T_x}
\begin{aligned}
     P_x &\coloneqq \text{the number of deposited trapeziums $\calT_{(a+b)/2}$ on the parallelogram $\calP_x$, under $\mathcal{UDA}$}, \\
    T_x &\coloneqq \text{the number of deposited trapeziums $\calT_{(a+b)/2}$ on the trapezium $\calT_x$, under $\mathcal{UDA}$}.
\end{aligned}
\end{align}
By convention, if there can be no trapeziums deposited, then we write $P_x,T_x\coloneqq 0$. E.g.\ if $x \in [0,a)$, then $P_x=0$, and if $x \in [0,(a+b)/2)$ then $T_x=0$. Likewise, we set $P_a = 1$ and $T_{(a+b)/2} = 1$. Moreover, define
\begin{equation}\label{eq:defn_eta}
    \mu(x)\coloneqq \E[P_x], \quad \nu(x) \coloneqq \E[T_x], \quad \text{ and }\quad \eta(x) \coloneqq \frac{\mu(x)+\nu(x)}{2}, \quad \text{ for all }x>0,
\end{equation} with the convention that $\eta(0)\coloneqq 0$. We emphasise that $\eta(x)$ itself does not have a direct physical interpretation as an expected parking number. Rather, this particular average is introduced because it leads to a convenient closed recursion. This distinction is immaterial in the large-substrate limit. As shown later in Lemma~\ref{lem:typewise_mean_error}, both $\mu(x)$ and $\nu(x)$ have the same asymptotic linear growth, and hence give the same limiting parking density. Throughout the paper, we will also use the notation $X_\calP(x)=P_x$ and $X_\calT(x)=T_x$.

\begin{remark}\label{rem:det_bound_eta}
    Note that we have the deterministic bounds
    \begin{equation}
        P_x, T_x \le \frac{2x}{a+b}, \quad \text{ for all $x >0$},
    \end{equation} since a rough upper bound is the total area $x$ of $\calP_x$ (resp. $\calT_x$) divided by the area of each deposited trapezium $\calT_{(a+b)/2}$. This gives us an upper bound for the growth of $\eta$, namely $\eta(x) \le 2x/(a+b)$.
\end{remark}

Throughout the paper, we use the notation $f(x) = \Oh(g(x))$ as $x \to \infty$ for positive functions $f,g$, if there exists a constant $C$ and $x_0$ such that $f(x) \le C g(x)$ for all $x \ge x_0$. Likewise, we write $f(x)=\oh(g(x))$ if $f(x)/g(x) \to 0$ as $x \to \infty$. Throughout the proofs, we will use $C$ as a generic constant, which can change from line to line.

\subsection{Main Results}
Define $\lambda(a,b)$ as the limit of the ratio between the expected number of deposited trapeziums $\calT_{(a+b)/2}$ in the jammed state, and the size $x$ of the substrate (as $x \to \infty$), i.e.\ $\lambda(a,b)\coloneqq \lim_{x \to \infty} \eta(x)/x$. Moreover, we define $\xi(a,b)$ as the limit (as $x\to \infty$) of the average substrate coverage density of deposited trapeziums $\calT_{(a+b)/2}$ in the jammed state, i.e.\ $\xi(a,b) \coloneqq \lambda(a,b)(a+b)/2$.

\begin{theorem}\label{thm:main_theorem}
Assume that $0 \le b \le a <\infty$ such that $a >0$, and let $\eta$ be defined as in~\eqref{eq:defn_eta}. Then, it follows that $\lim_{x \to \infty}\eta(x)/x = \lambda(a,b) \in (0,\infty)$, where
\begin{equation}\label{eq:parking_const_number_dep}
    \lambda(a,b)\coloneqq \frac{1}{2} \int_0^\infty (2+e^{-tb}-e^{-ta})\exp \bigg( \int_0^t \mleft(-2+e^{-s(a+b)/2}+e^{-as}\mright) \frac{\D s}{s}\bigg) \D t.
    \end{equation} 
    Moreover, as $x > 2a(e+1)$, it holds that
    \begin{equation}\label{eq:gen_roc_eta}
    \bigg|\eta(x) -\lambda(a,b) x +1-\lambda(a,b) \frac{a+b}{2}\bigg| 
    \le\begin{dcases}
      \mleft(\frac{2e}{x/a}\mright)^{x/a-3/2} \frac{e^{3/2} K_1}{\sqrt{\pi}}, &\text{ if }a=b,\\
      \mleft(\frac{2e}{x/a}\mright)^{x/a-5/2} \frac{3e^{5/2}(2K_1+K_2)}{\sqrt{\pi}}, &\text{ if }a>b.
    \end{dcases}
\end{equation} where $K_1 \coloneqq 1 + 2a(2/(a+b)+\lambda(a,b)) + \lambda(a,b)\cdot  (a+b)/2>0$ and $K_2 \coloneqq (2/(a+b)+\lambda(a,b))(a-b)/2\ge 0$. 
\end{theorem}

\begin{remark}
\textbf{(i)} The rate of convergence in~\eqref{eq:gen_roc_eta} depends on $(a,b)$ in a genuinely structural way, not merely through a multiplicative constant. The dependence on $b$ is essentially through whether $a=b$ or not: the case $a=b$ yields a polynomially faster rate. The parameter $a$ serves as a time scale. Indeed, after the rescaling $x\mapsto ax$, the bound of $|\eta(ax)-\lambda(a,b) ax+1-\lambda(a,b)(a+b)/2|$ depends on $a$ only through a multiplicative constant.

Moreover, when $a=b$ the Dvoretzky--Robbins bound~\cite{Dvoretzky_Robbins} is recovered from~\eqref{eq:gen_roc_eta} with explicit constants. Note, however, that if $a=b\neq 1$ the rate still depends on $a$ through a time-scale, and not just through a multiplicative $a$-dependent constant, a point not noted in~\cite{Dvoretzky_Robbins}.

\smallskip

\textbf{(ii)} From the proof of Lemma~\ref{lem:ROC_first_order} (which yields the rate in Theorem~\ref{thm:main_theorem}), define
\begin{equation}\label{eq_defn_I_x,S_x_prime}
I_x \coloneqq \inf_{x \le c \le x+a} \frac{\eta(c)+1}{c+(a+b)/2}, \quad \text{ and }
\quad
S_x \coloneqq \sup_{x \le c \le x+a} \frac{\eta(c)+1}{c+(a+b)/2},
\quad \text{ for all } x\ge 0.
\end{equation}
By~\eqref{eq:bounded_over_lambda_ixsx} and the proof of Theorem~\ref{thm:main_theorem}, we have
 $I_x \le \lambda(a,b) \le S_x$ for all $x\ge 0$, and $\lim_{x\to\infty} I_x = \lambda(a,b) = \lim_{x\to\infty} S_x$. Thus, $\lambda(a,b)$ can be arbitrarily well approximated by taking $x$ large. This is often preferable when a closed form for $\lambda(a,b)$ is intractable, since $I_x$ and $S_x$ depend only on the explicit form of $\eta$.
\end{remark}

Sometimes it is more convenient to work with $\xi(a,b)$ instead of $\lambda(a,b)$, since it has an important scaling property, as considered in the ensuing corollary.

\begin{corollary}\label{cor:parking_constant}
    For the parking constant $\xi(a,b)= \lambda(a,b)(a+b)/2$, it follows that
\begin{equation}
    \xi(a,b)
    =
    \frac{1+\frac{b}{a}}{4} \int_0^\infty \bigg(2+e^{-t\cdot \frac{b}{a}}-e^{-t}\bigg)\exp \bigg( \int_0^t \mleft(-2+e^{-s(1+\frac{b}{a})/2}+e^{-s}\mright) \frac{\D s}{s}\bigg) \D t
    =
    \xi\big(1,\tfrac{b}{a}\big).
\end{equation} Thus, we may without loss of generality assume that $a=1$ and $b \in [0,1]$ when working with $\xi(a,b)$.
\end{corollary}

Theorem~\ref{thm:main_theorem} describes the first-order behaviour of the parking process by identifying the asymptotic mean number of deposited trapeziums and giving a quantitative estimate of its finite-size correction. It is natural to ask next how strongly the number of deposited trapeziums fluctuates around this mean. Besides being of independent interest, the variance determines the natural scale of the fluctuations and is therefore a necessary first step towards a central limit theorem for the parking process.

The recursive structure of the model makes this question particularly interesting. As explained above, a deposition produces child substrates of two different types, $\calP_x$ and $\calT_x$, and the corresponding parking variables $P_x$ and $T_x$ satisfy different local recursions. There is therefore no a priori reason to expect their second-order asymptotics to coincide. Nevertheless, we show that the distinction between the two substrate geometries disappears at leading order: both variances grow linearly with $x$, with the same strictly positive asymptotic variance density $\lambda_2(a,b)$. Moreover, we obtain quantitative bounds on the corresponding remainder terms. As for the first moment, these bounds distinguish sharply between rectangles, $a=b$, and genuine trapeziums, $a>b$, and their rates depend explicitly on the geometry of the deposited particle. When $a=b=1$, the result recovers the second-order asymptotic estimate of Dvoretzky and Robbins~\cite{Dvoretzky_Robbins}.

\begin{theorem}\label{thm:variance_asymptotics}
There exist finite constants $\lambda_2(a,b)>0$ and $C_{a,b}>0$ such that, 
\begin{equation}
    \max_{s \in \{\calP,\calT\}}\mleft| \Var(X_s(x))-\lambda_2(a,b)\mleft(x+\frac{a+b}{2} \mright)\mright| \le C_{a,b} \mathcal R_2(x),
\end{equation} for all $x>(3a-b)/2$, where

\begin{equation*}
    \mathcal R_2(x)
\coloneqq
\begin{dcases}
    \mleft(\dfrac{4e}{x/a}\mright)^{x/a-4},
    &\text{if }a=b,\\
    \mleft(\dfrac{4e}{x/a}\mright)^{x/a-(5+(a+b)/(2a))},
    &\text{if }a>b.
\end{dcases}
\end{equation*}
\end{theorem}
From Theorem~\ref{thm:variance_asymptotics}, it equivalently holds that
\[
    \Var(P_x)
    =
    \lambda_2(a,b)x
    +
    \lambda_2(a,b)\frac{a+b}{2}
    +
    \Oh(\mathcal R_2(x)), \quad \text{as }x \to \infty,
\]
and the same asymptotic expansion holds for $\Var(T_x)$.

\begin{remark}
\noindent{\bf{(i)}} When $a=b=1$, the two types of strips coincide, and
Theorem~\ref{thm:variance_asymptotics} becomes
\[
    \Var(P_x)
    =
    \lambda_2(1,1)x+\lambda_2(1,1)
    +
    \Oh\mleft(
        \mleft(\frac{4e}{x}\mright)^{x-4}
    \mright),
\]
which is precisely the form of the second-order estimate of Dvoretzky and Robbins~\cite[Eq.~(1.4)]{Dvoretzky_Robbins}. For a general rectangle, that is, when $a=b$, the corresponding statement follows after the natural rescaling $x\mapsto x/a$.

\smallskip

\noindent{\bf{(ii)}} The remainder in Theorem~\ref{thm:variance_asymptotics} exhibits a non-trivial dependence on the geometry of the deposited trapezium. To make this explicit, write $y=x/a$ and $r=b/a \in [0,1]$. Then the remainder has the form
\[
\mathcal R_2(x)
=
\1_{\{r=1\}}\mleft(\frac{4e}{y}\mright)^{y-4}+ \1_{\{r \in [0,1)\}}
\mleft(\frac{4e}{y}\mright)^{y-(11+r)/2}.
\]
Thus, $a$ determines the natural length scale $x/a$, whereas, for genuine trapeziums, the ratio $b/a$ enters directly into the polynomial correction to the super-exponential rate. In particular, for $r<1$, comparison with the rectangular remainder gives $\mleft(4e/y\mright)^{y-(11+r)/2}
=
\mleft(y/(4e)\mright)^{(3+r)/2}
\mleft(4e/y\mright)^{y-4}$. Hence, although both cases have the same leading super-exponential behaviour, the bound for a genuine trapezium loses a polynomial factor of order $y^{(3+r)/2}$ relative to the rectangular case. Most notably, this loss persists as $r\uparrow1$: for trapeziums whose two bases are arbitrarily close in length, the loss approaches a factor of order $y^2$, whereas at the exact rectangular value $r=1$ the sharper exponent $y-4$ is recovered. The case $a=b$ is therefore not merely the limiting member of the trapezoidal family at the level of our error estimate; its additional symmetry produces an improved remainder rate.

\smallskip

\noindent{\nf{\bf{(iii)}}} Locally, let $P^{(a,b)}_x$ (resp. $T_x^{(a,b)}$) be the random variable $P_x$ (resp. $T_x$) associated with the lengths $a$ and $b$. Then, under horizontal scaling $a^{-1}$, it follows that $P_x^{(a,b)}\eqd P_{x/a}^{(1,b/a)}$ and $T_x^{(a,b)}\eqd T_{x/a}^{(1,b/a)}$. Hence, comparing the variance asymptotics, we get that $\lambda_2(a,b)=a^{-1}\lambda_2(1,b/a)$, therefore implying that $\max_{s \in \{\calP,\calT\}}| \Var(X_s(x))-\lambda_2(1,b/a)(x/a+(1+b/a)/2 )| \le C_{a,b} \mathcal R_2(x)$, where for $a>b$, we get $\mathcal{R}_2(x)=(4e/(x/a))^{x/a-(11+b/a)/2}$, making it clear that the relevant ratios are $x/a$ and $b/a$.
\end{remark}

\subsection{Worst trapezium}
A relevant question, having Corollary~\ref{cor:parking_constant} in mind, is whether there exists a ``worst'' (resp.\ ``best'') ratio $b/a$ for which the average coverage density $\xi(1,b/a)$ is smallest (resp.\ largest). As we see in Lemma~\ref{lem:worst_trap_exists_uniq}, there exists a unique worst ratio, minimising the average coverage density $x \mapsto \xi(1,x)$ for $x \in [0,1]$, and surprisingly, we find that $x \mapsto \xi(1,x)$ is not strictly increasing. Moreover, since $x \mapsto \xi(1,x)$ only has one extremum point, which is a minimum on $(0,1)$, the ``best'' ratio must therefore be at $0$ or $1$, values which are already known from~\cite{Renyi_OG} and~\cite{triangles}, concluding that the ``best'' case is $b/a=1$.
\begin{lemma}\label{lem:worst_trap_exists_uniq}
The function $x \mapsto \xi(1,x)$ defined on $[0,1]$ has a unique minimiser $x^\ast \in (0,1)$ such that $\xi(1,x^*)=\inf_{x \in [0,1]} \xi(1,x)$.
\end{lemma}
From available analytical tools, finding a closed form for the minimiser does not seem possible. However, numerical evaluation shows $x^\ast \approx 0.07991824$, as we can also see in Figure~\ref{fig:worst_trapezium}.

\begin{figure}[ht]
     \centering
     \begin{subfigure}[b]{0.48\textwidth}
         \centering
         \includegraphics[width=\textwidth]{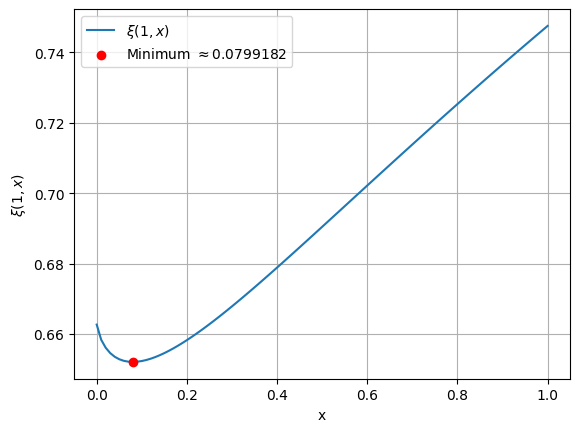}
         \caption{Plot of $x\mapsto \xi(1,x)$ for $x \in [0,1]$.}
         \label{fig:xi(1,x)}
     \end{subfigure}
     \hfill
     \begin{subfigure}[b]{0.49\textwidth}
         \centering
         \includegraphics[width=\textwidth]{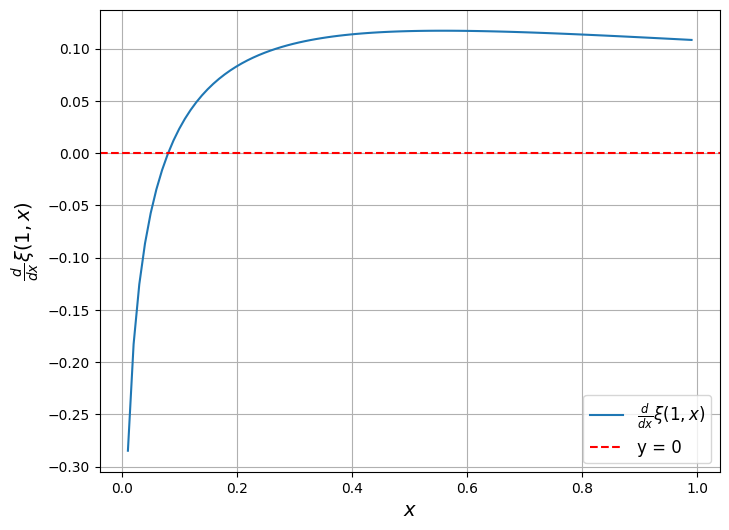}
         \caption{Plot of the derivative $x \mapsto \tfrac{\D}{\D x}\xi(1,x)$ for $x \in (0,1]$.}
         \label{fig:xi'(1,x)}
     \end{subfigure}
        \caption{Finding the worst trapezium.}
        \label{fig:worst_trapezium}
\end{figure}

At first, it seems surprising that $x \mapsto \xi(1, x)$ from Corollary~\ref{cor:parking_constant} is not strictly increasing. The cases of $x = 0$ and $x = 1$ were known, since those are precisely the jamming limits from~\cite{Renyi_OG} and~\cite{triangles}. Hence, one could be tempted to think that we should expect to see a strictly increasing function $x \mapsto \xi(1,x)$. However, after simulations (Figure~\ref{fig:worst_trap_sim}) and plotting the closed form of the function (Figure~\ref{fig:worst_trapezium}), it seems clearer why such a shape would occur. Indeed, if our deposited objects are triangles, we will almost surely alternate between triangles with their base on the upper line of the substrate, and triangles with their base on the lower line of the substrate (see Figure~\ref{fig:triangles-small_gaps}). However, as soon as the value $x$ is strictly larger than $0$, there is a positive probability of having two trapeziums next to each other, both having their (longer) base on the same line of the substrate (see Figure~\ref{fig:trapeziums-big_gaps}). This can potentially leave large unusable gaps, and this provides a geometric mechanism explaining the initial decrease for the jamming limit.

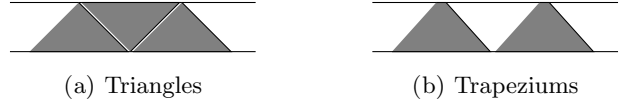
\begin{figure}
\centering
\begin{subfigure}{0.3\textwidth}
    \centering
    \begin{tikzpicture}[scale = 0.65]
	    \draw (0, 0) -- (5, 0);
	    \draw (0, 1) -- (5, 1);
        \filldraw[fill=gray] (0.4, 0) -- (2.4,0) -- (1.4,1);
        \filldraw[fill=gray] (1.45, 1) -- (3.45,1) -- (2.45,0);
        \filldraw[fill=gray] (2.5, 0) -- (4.5,0) -- (3.5,1);
    \end{tikzpicture}
    \caption{Triangles}
    \label{fig:triangles-small_gaps}
\end{subfigure}
\begin{subfigure}{0.3\textwidth}
    \centering
    \begin{tikzpicture}[scale = 0.65]
	    \draw (0, 0) -- (5, 0);
        \draw (0, 1) -- (5, 1);
        \filldraw[fill=gray] (0.4, 0) -- (2.4,0) -- (1.5,1) -- (1.3, 1);
        \filldraw[fill=gray] (2.5, 0) -- (4.5,0) -- (3.6,1) -- (3.4, 1);
    \end{tikzpicture}
    \caption{Trapeziums}
    \label{fig:trapeziums-big_gaps}
\end{subfigure}
\caption{Big gaps appearing once triangles are replaced with trapeziums.}
\end{figure}

\begin{figure}[ht]
     \centering
         \includegraphics[width=0.49\textwidth]{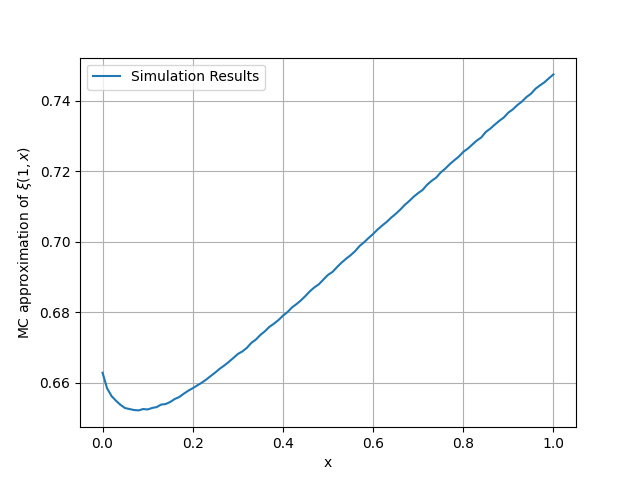}
         \caption{Simulation of $x\mapsto \xi(1,x)$ for $x \in [0,1]$.}
         \label{fig:worst_trap_sim}
\end{figure}

In addition to plotting the function $\xi(1, x)$ from Corollary~\ref{cor:parking_constant}, we ran a simulation of our model. The precise values of $x$ in the simulations went from $0$ to $1$ with a step size of $0.01$. For each particular value of $x$, we ran $500$ simulations, with the strip length set to $5000$. Simulation results are shown in Figure~\ref{fig:worst_trap_sim}. Clearly, the illustrative MC approximation aligns very well with the function shown in Figure~\ref{fig:xi(1,x)}.

\subsection{Open Problems}\label{sec:open_problems}
Below, we list some open problems related to those considered in this paper, which seem appropriate for future work. 
\begin{itemize}[leftmargin=2em]
    \item A CLT for the variables $T_x$ and $P_x$, whenever $a>b$, and a quantitative bound via a Stein's argument in the total-variation distance or the Wasserstein distance. 
    \item Finding whether $\lambda_2(a,b)$ admits an explicit integral representation analogous to $\lambda(a,b)$.
    \item While we have shown that there exists a unique minimiser for $x \mapsto \xi(1,x)$ on $[0,1]$, an open problem is finding a closed analytical form for the minimiser $x^*$. Due to the very complex form of $\xi(1,x)$, such an analytical closed form for $x^*$ seems out of reach, but it is an open problem.
\end{itemize}

\section{Proofs}

\subsection{Recursive equations for regularity}

\begin{lemma}\label{lem:eta_func}
For $\eta$ defined in~\eqref{eq:defn_eta}, it follows that $\eta(x)=0$ for $x \in [0,(a+b)/2)$, $\eta(x)=1/2$ for $[(a+b)/2,a)$, $\eta(a)=1$ and  
    \begin{equation}\label{eq:defn_eta_x}
        \eta(x) =
            1+\frac{1}{x-a} \bigg( \int_0^{x-a} \eta(t) \D t + \int_0^{x-\tfrac{a+b}{2}} \eta(t)\D t\bigg), \quad \text{ for all } x > a,
    \end{equation} where $\eta$ is a non-decreasing function on $(0,\infty)$, and continuous on $(a,\infty)$, specifically $\eta$ is locally absolutely continuous on $(a,\infty)$.
\end{lemma}

\begin{proof}[Proof of Lemma~\ref{lem:eta_func}]
    Recall the definition of $P_x$ and $T_x$ and how we deposit the trapeziums using the random sequential adsorption model. Under the $\mathcal{UDA}$, we let $\tau\sim \mathcal{U}(0,2x-2a)$, which yields a self-recursive relation which works in the case where $x \ge (3a-b)/2$. This case is covered in~\eqref{eq:probibalistic_vers_P_x_T_x} below. However, we will first consider the values of $x$ up to $(3a-b)/2$, since we can determine directly the exact number of possible deposited trapeziums.

    We start by finding the initial values of $\nu$ and $\mu$ for $x<(3a+b)/2$, by counting the number of possibilities for deposited trapeziums. Note that we need the initial values of $\mu$ and $\nu$ only for $x < (3a-b)/2$. However, in the procedure that follows, we build up those values with steps of particular sizes, and we do that until we are sure that we crossed over $(3a-b)/2$.  When $x<(a+b)/2$, it follows directly from definition that $\calT_{(a+b)/2}$ cannot fit on either $\calT_x$ or $\calP_x$, and hence $\eta(x)=0$ for $x \in (0,(a+b)/2)$. Next, for $x \in [(a+b)/2,a)$, we have that $P_x=0$ since $\calT_{(a+b)/2}$ cannot fit on $\calP_x$, and $T_x=1$ since we will always be able to fit one and only one $\calT_{(a+b)/2}$ on $\calT_x$. This yields that $\eta(x)=1/2$ for all $x \in [(a+b)/2,a)$. Consider now $x \in [a,a+b)$. In this case $T_x=1$ (resp. $P_x=1$) since we can fit one and only one $\calT_{(a+b)/2}$ on $\calT_x$ (resp. $\calP_x$). Hence, $\eta(x)=1$ for all $x \in [a,a+b)$. Notice that $a+b$ is not necessarily greater than $(3a-b)/2$, so we have to go further.

    Let now $x \in [a+b,(3a+b)/2)$. In this case, it follows that $T_x=1$ for all $x \in [a+b,(3a+b)/2)$, and hence it suffices to calculate $P_x$. Let $\tau \sim \mathcal{U}(0,2x-2a)$, implying by~\eqref{eq:defn:P_x,T_x} that
    \begin{equation}
        P_x=2 \1_{\{\tau \in (0,x-a-b)\}}+\1_{\{\tau \in (x-a-b,x-a)\}}+\1_{\{\tau \in (x-a,x-a+b)\}}+2\1_{\{\tau \in (x-a+b,2x-2a)\}} \, \text{ a.s}.
    \end{equation} Hence, by the law of total expectation $\E[X_T]=\E[\E[X_t|T=t]]$, it follows that 
    \begin{equation}
        \mu(x)= \E[P_x]=\frac{1}{2x-2a}\bigg( \int_0^{x-a-b}2 \D t +\int_{x-a-b}^{x-a+b}\D t + \int_{x-a+b}^{2x-2a}2 \D t \bigg)=1+\frac{x-a-b}{x-a}.
    \end{equation} Altogether, we therefore have that $\eta(x)=1+(x-a-b)/(2x-2a)$ for all $x \in [a+b,(3a+b)/2)$.

    If $x > (3a-b)/2$, it follows from Definition~\ref{defn:UDS_} of $\mathcal{UDA}$, that 
\begin{align}\label{eq:probibalistic_vers_P_x_T_x}
    \begin{aligned}
        P_x&
        = 1+\1_{\{\tau \in (0,x-a)\}}\bigg(P_{\tau}+T_{x-\tau-\tfrac{a+b}{2}}\bigg)+\1_{\{\tau \in (x-a,2x-2a)\}}\bigg(T_{t_\tau}+P_{x-t_{\tau}-\tfrac{a+b}{2}}\bigg) \,\text{ a.s. and},\\
        T_x
       & = 
        1+\1_{\big\{\tau \in \big(0,x-\tfrac{3a-b}{2}\big)\big\}}\bigg(T_{\tau + \tfrac{a-b}{2}}+T_{x-\tau-a}\bigg)+\1_{\big\{\tau \in \big(x-\tfrac{3a-b}{2},2x-2a\big)\big\}}\bigg(P_{t_\tau}+P_{x-t_{\tau}-\tfrac{a+b}{2}}\bigg) \, \text{ a.s.},
    \end{aligned}
    \end{align} where $t_\tau=\tau+(3a-b)/2-x$, as illustrated in Figure~\ref{fig:trapezium_arriving}. For all $x\ge (3a+b)/2>(3a-b)/2$ we can find $\eta(x)$ by the recursive relationship in~\eqref{eq:probibalistic_vers_P_x_T_x}. By the useful rule $\E[X_T]=\E[\E[X_t|T=t]]$ and change of variables, it follows from~\eqref{eq:probibalistic_vers_P_x_T_x} that
    \begin{align}\label{eq:resurc_mu}
    \begin{aligned}
        \mu(x)
        &=
        1+\frac{1}{2x-2a} \bigg( \int_0^{x-a} \mu(t)+\nu\big(x-t-\tfrac{a+b}{2}\big) \D t + \int_{\tfrac{a-b}{2}}^{x-\tfrac{a+b}{2}} \nu(t)+\mu\big(x-t-\tfrac{a+b}{2}\big)\D t\bigg)\\
        &=1+\frac{1}{x-a} \bigg(\int_0^{x-a} \mu(t)\D t + \int_{\tfrac{a-b}{2}}^{x-\tfrac{a+b}{2}}\nu(t) \D t\bigg).
    \end{aligned}
    \end{align} Similarly, it follows from~\eqref{eq:probibalistic_vers_P_x_T_x} that
    \begin{align}\label{eq:resurc_nu}
    \begin{aligned}
         \nu(x)
        &=
        1+\frac{1}{2x-2a}\bigg( \int_{0}^{x-\tfrac{3a-b}{2}} \nu\big(t+\tfrac{a-b}{2}\big)+\nu (x-t-a )\D t + \int_0^{x-\tfrac{a+b}{2}} \mu(t)+ \mu\big(x-t-\tfrac{a+b}{2}\big) \D t \bigg)\\
        &=1+\frac{1}{x-a} \bigg( \int_{\tfrac{a-b}{2}}^{x-a} \nu(t) \D t + \int_0^{x-\tfrac{a+b}{2}} \mu(t) \D t \bigg).
    \end{aligned}
    \end{align}
    Note that $\nu(x)=0$ for all $x \in (0,(a-b)/2)$, since we cannot fit $\calT_{(a+b)/2}$ inside $\calT_{(a-b)/2}$. Thus,
    \begin{equation}\label{eq:eta_xbig_3a+b/2}
        \eta(x)= 1+\frac{1}{x-a} \bigg( \int_0^{x-a} \eta(t) \D t+ \int_0^{x-\tfrac{a+b}{2}} \eta(t) \D t\bigg), \quad \text{ for all }x\ge \frac{3a+b}{2}.
    \end{equation}
    By using the values for $\eta(x)$ when $x \in (0,a+b)$, one can show that the form $\eta(x)=1+(x-a-b)/(2x-2a)$ for all $x \in [a+b,(3a+b)/2)$ found above, corresponds to plugging the initial $\eta$ values into~\eqref{eq:eta_xbig_3a+b/2}, concluding that~\eqref{eq:eta_xbig_3a+b/2} holds for all $x>a+b$. The similar approach holds for $x \in (a,a+b)$, since $\eta(x)=0$ for all $x <(a+b)/2$. Hence,~\eqref{eq:eta_xbig_3a+b/2} holds for all $x>a$.

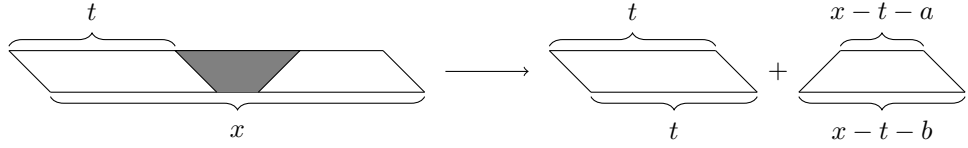
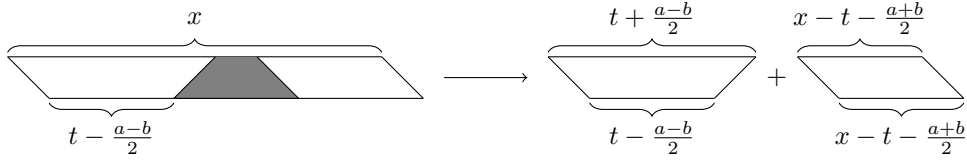
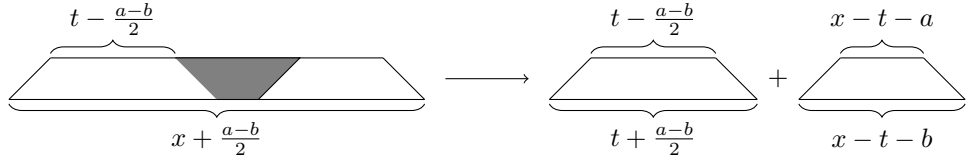
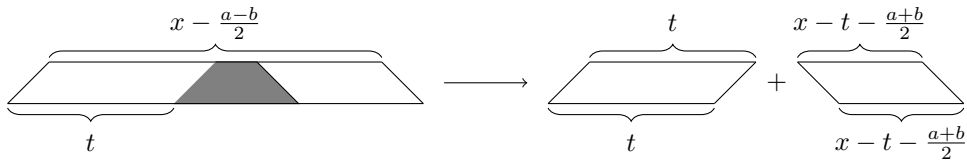
\begin{figure}[!ht]
\begin{subfigure}{1.0\textwidth}
    \centering
    \begin{tikzpicture}[scale = 0.55]
        \draw (1, 0) -- (10, 0);
	\draw (0, 1) -- (9, 1);
        \draw (1, 0) -- (0, 1);
        \draw (10, 0) -- (9, 1);
	\filldraw[fill=gray] (6, 0) -- (7,1) -- (4,1) -- (5, 0);
        \draw[decorate,decoration={brace,amplitude=5pt,raise=0pt,mirror},xshift=0pt, yshift = -3pt] (1, 0) -- (10, 0) node [midway,yshift=-13pt,xshift=0pt]{$x$};
        \draw[decorate,decoration={brace,amplitude=5pt,raise=0pt,mirror},xshift=0pt, yshift = 3pt] (4, 1) -- (0, 1) node [midway,yshift=13pt,xshift=0pt]{$t$};

        \draw[->] (10.5, 0.5) -- (12.5, 0.5);

        \draw (14, 0) -- (18, 0);
	\draw (13, 1) -- (17, 1);
        \draw (14, 0) -- (13, 1);
        \draw (18, 0) -- (17, 1);
        \draw[decorate,decoration={brace,amplitude=5pt,raise=0pt,mirror},xshift=0pt, yshift = -3pt] (14, 0) -- (18, 0) node [midway,yshift=-13pt,xshift=0pt]{$t$};
        \draw[decorate,decoration={brace,amplitude=5pt,raise=0pt,mirror},xshift=0pt, yshift = 3pt] (17, 1) -- (13, 1) node [midway,yshift=13pt,xshift=0pt]{$t$};
        
        \node[rectangle, draw=none, minimum size=1pt] at (18.5, 0.5) {$+$};

        \draw (19, 0) -- (23, 0);
	\draw (20, 1) -- (22, 1);
        \draw (19, 0) -- (20, 1);
        \draw (23, 0) -- (22, 1);
        \draw[decorate,decoration={brace,amplitude=5pt,raise=0pt,mirror},xshift=0pt, yshift = -3pt] (19, 0) -- (23, 0) node [midway,yshift=-13pt,xshift=0pt]{$x - t - b$};
        \draw[decorate,decoration={brace,amplitude=5pt,raise=0pt,mirror},xshift=0pt, yshift = 3pt] (22, 1) -- (20, 1) node [midway,yshift=13pt,xshift=0pt]{$x - t - a$};
    \end{tikzpicture}
    \caption{Trapezium $\calT_{(a+b)/2}$ arriving on $\calP_x$ with base up.}
    \label{fig:parallelogram_up}
\end{subfigure}

\bigskip

\begin{subfigure}{1.0\textwidth}
    \centering
    \begin{tikzpicture}[scale = 0.55]
        \draw (1, 0) -- (10, 0);
	\draw (0, 1) -- (9, 1);
        \draw (1, 0) -- (0, 1);
        \draw (10, 0) -- (9, 1);
	\filldraw[fill=gray] (6, 1) -- (7,0) -- (4,0) -- (5, 1);
        \draw[decorate,decoration={brace,amplitude=5pt,raise=0pt,mirror},xshift=0pt, yshift = -3pt] (1, 0) -- (4, 0) node [midway,yshift=-13pt,xshift=0pt]{$t-\frac{a - b}{2}$};
        \draw[decorate,decoration={brace,amplitude=5pt,raise=0pt,mirror},xshift=0pt, yshift = 3pt] (9, 1) -- (0, 1) node [midway,yshift=13pt,xshift=0pt]{$x$};

        \draw[->] (10.5, 0.5) -- (12.5, 0.5);

        \draw (14, 0) -- (17, 0);
	\draw (13, 1) -- (18, 1);
        \draw (14, 0) -- (13, 1);
        \draw (17, 0) -- (18, 1);
        \draw[decorate,decoration={brace,amplitude=5pt,raise=0pt,mirror},xshift=0pt, yshift = -3pt] (14, 0) -- (17, 0) node [midway,yshift=-13pt,xshift=0pt]{$t-\frac{a-b}{2}$};
        \draw[decorate,decoration={brace,amplitude=5pt,raise=0pt,mirror},xshift=0pt, yshift = 3pt] (18, 1) -- (13, 1) node [midway,yshift=13pt,xshift=0pt]{$t+\frac{a-b}{2}$};
        
        \node[rectangle, draw=none, minimum size=1pt] at (18.5, 0.5) {$+$};

        \draw (20, 0) -- (23, 0);
	\draw (19, 1) -- (22, 1);
        \draw (20, 0) -- (19, 1);
        \draw (23, 0) -- (22, 1);
        \draw[decorate,decoration={brace,amplitude=5pt,raise=0pt,mirror},xshift=0pt, yshift = -3pt] (20, 0) -- (23, 0) node [midway,yshift=-13pt,xshift=0pt]{$x - t - \frac{a+b}{2}$};
        \draw[decorate,decoration={brace,amplitude=5pt,raise=0pt,mirror},xshift=0pt, yshift = 3pt] (22, 1) -- (19, 1) node [midway,yshift=13pt,xshift=0pt]{$x - t - \frac{a+b}{2}$};
    \end{tikzpicture}
    \caption{Trapezium $\calT_{(a+b)/2}$ arriving on $\calP_x$ with base down.}
    \label{fig:parallelogram_down}
\end{subfigure}

\bigskip

\begin{subfigure}{1.0\textwidth}
    \centering
    \begin{tikzpicture}[scale = 0.55]
        \draw (0, 0) -- (10, 0);
	\draw (1, 1) -- (9, 1);
        \draw (0, 0) -- (1, 1);
        \draw (10, 0) -- (9, 1);
	\filldraw[fill=gray] (5, 0) -- (6, 0) -- (7,1) -- (4,1);
        \draw[decorate,decoration={brace,amplitude=5pt,raise=0pt,mirror},xshift=0pt, yshift = -3pt] (0, 0) -- (10, 0) node [midway,yshift=-13pt,xshift=0pt]{$x+\frac{a-b}{2}$};
        \draw[decorate,decoration={brace,amplitude=5pt,raise=0pt,mirror},xshift=0pt, yshift = 3pt] (4, 1) -- (1, 1) node [midway,yshift=13pt,xshift=0pt]{$t-\frac{a-b}{2}$};

        \draw[->] (10.5, 0.5) -- (12.5, 0.5);

        \draw (13, 0) -- (18, 0);
	\draw (14, 1) -- (17, 1);
        \draw (13, 0) -- (14, 1);
        \draw (18, 0) -- (17, 1);
        \draw[decorate,decoration={brace,amplitude=5pt,raise=0pt,mirror},xshift=0pt, yshift = -3pt] (13, 0) -- (18, 0) node [midway,yshift=-13pt,xshift=0pt]{$t+\frac{a-b}{2}$};
        \draw[decorate,decoration={brace,amplitude=5pt,raise=0pt,mirror},xshift=0pt, yshift = 3pt] (17, 1) -- (14, 1) node [midway,yshift=13pt,xshift=0pt]{$t-\frac{a-b}{2}$};
        
        \node[rectangle, draw=none, minimum size=1pt] at (18.5, 0.5) {$+$};

        \draw (19, 0) -- (23, 0);
	\draw (20, 1) -- (22, 1);
        \draw (19, 0) -- (20, 1);
        \draw (23, 0) -- (22, 1);
        \draw[decorate,decoration={brace,amplitude=5pt,raise=0pt,mirror},xshift=0pt, yshift = -3pt] (19, 0) -- (23, 0) node [midway,yshift=-13pt,xshift=0pt]{$x - t - b$};
        \draw[decorate,decoration={brace,amplitude=5pt,raise=0pt,mirror},xshift=0pt, yshift = 3pt] (22, 1) -- (20, 1) node [midway,yshift=13pt,xshift=0pt]{$x - t - a$};
    \end{tikzpicture}
    \caption{Trapezium $\calT_{(a+b)/2}$ arriving on $\calT_x$ with base up.}
    \label{fig:trapeze_up}
\end{subfigure}

\bigskip

\begin{subfigure}{1.0\textwidth}
    \centering
    \begin{tikzpicture}[scale = 0.55]
        \draw (0, 0) -- (10, 0);
	\draw (1, 1) -- (9, 1);
        \draw (0, 0) -- (1, 1);
        \draw (10, 0) -- (9, 1);
	\filldraw[fill=gray] (5, 1) -- (6,1) -- (7,0) -- (4, 0);
        \draw[decorate,decoration={brace,amplitude=5pt,raise=0pt,mirror},xshift=0pt, yshift = -3pt] (0, 0) -- (4, 0) node [midway,yshift=-13pt,xshift=0pt]{$t$};
        \draw[decorate,decoration={brace,amplitude=5pt,raise=0pt,mirror},xshift=0pt, yshift = 3pt] (9, 1) -- (1, 1) node [midway,yshift=13pt,xshift=0pt]{$x-\frac{a-b}{2}$};

        \draw[->] (10.5, 0.5) -- (12.5, 0.5);

        \draw (13, 0) -- (17, 0);
	\draw (14, 1) -- (18, 1);
        \draw (13, 0) -- (14, 1);
        \draw (17, 0) -- (18, 1);
        \draw[decorate,decoration={brace,amplitude=5pt,raise=0pt,mirror},xshift=0pt, yshift = -3pt] (13, 0) -- (17, 0) node [midway,yshift=-13pt,xshift=0pt]{$t$};
        \draw[decorate,decoration={brace,amplitude=5pt,raise=0pt,mirror},xshift=0pt, yshift = 3pt] (18, 1) -- (14, 1) node [midway,yshift=13pt,xshift=0pt]{$t$};
        
        \node[rectangle, draw=none, minimum size=1pt] at (18.5, 0.5) {$+$};

        \draw (20, 0) -- (23, 0);
	\draw (19, 1) -- (22, 1);
        \draw (20, 0) -- (19, 1);
        \draw (23, 0) -- (22, 1);
        \draw[decorate,decoration={brace,amplitude=5pt,raise=0pt,mirror},xshift=0pt, yshift = -3pt] (20, 0) -- (23, 0) node [midway,yshift=-13pt,xshift=0pt]{$x - t - \frac{a+b}{2}$};
        \draw[decorate,decoration={brace,amplitude=5pt,raise=0pt,mirror},xshift=0pt, yshift = 3pt] (22, 1) -- (19, 1) node [midway,yshift=13pt,xshift=0pt]{$x - t - \frac{a+b}{2}$};
    \end{tikzpicture}
    \caption{Trapezium $\calT_{(a+b)/2}$ arriving on $\calT_x$ with base down.}
    \label{fig:trapeze_down}
\end{subfigure}
\caption{Illustration of all the possible scenarios of the trapezium $\calT_{(a+b)/2}$ arriving on strips $\calP_x$ and $\calT_x$ (introduced in Figure~\ref{fig:shapes_of_strips}). After a trapezium arrives, the original strip decomposes into two smaller strips that are again of the two types introduced in Figure~\ref{fig:shapes_of_strips}.}
\label{fig:trapezium_arriving}
\end{figure}


We can now show that $\eta$ is continuous on $(a,\infty)$. Define the indefinite integral $J(z)\coloneqq \int_{0}^{z}\eta(t)\D t$, for $z \ge 0$. Since $\eta$ is (piecewise) bounded on bounded intervals, $J$ is locally Lipschitz-continuous, and hence continuous, on $[0,\infty)$. For $x>a$ we may rewrite~\eqref{eq:defn_eta_x} as $\eta(x)=1+(J(x-a)+J(x-(a+b)/2))/(x-a)$. Here $x\mapsto x-a$ and $x\mapsto x-(a+b)/2$ are continuous on $(a,\infty)$, $J$ is continuous on $[0,\infty)$, and the denominator $x-a$ never vanishes on $(a,\infty)$. Therefore, $\eta$ is continuous on $(a,\infty)$.

Next, we prove that $x\mapsto \eta(x)$ is non-decreasing on $(0,\infty)$. Since $\eta$ is locally bounded, $J$ is locally absolutely continuous on $[0,\infty)$. Moreover, for all $x>a$, we have $\eta(x)=1+(J(x-a)+J(x-(a+b)/2))/(x-a)$, as discussed above. Therefore, it follows that $\eta$ is locally absolutely continuous on $(a,\infty)$. In particular, since $J'(z)=\eta(z)$ for almost every $z>0$, differentiation of the preceding identity yields, for almost every $x>a$,
\begin{equation}\label{eq:star_3_pictres}
    \frac{\D}{\D x}\eta(x)
    =
    -\frac{1}{(x-a)^2}
    \bigg(
        \int_0^{x-a}\eta(t)\D t
        +
        \int_0^{x-\tfrac{a+b}{2}}\eta(t)\D t
    \bigg)
    +
    \frac{1}{x-a}
    \bigg(
        \eta(x-a)
        +
        \eta\bigl(x-\tfrac{a+b}{2}\bigr)
    \bigg).
\end{equation}
Furthermore,
\[
    \frac{\eta(x-a)}{x-a}
    =
    \frac{1}{(x-a)^2}
    \int_0^{x-a}\eta(x-a)\D t, \quad \text{and}\quad
    \frac{\eta\bigl(x-\tfrac{a+b}{2}\bigr)}{x-a}
    =
    \frac{1}{(x-a)^2}
    \int_{\tfrac{a-b}{2}}^{x-\tfrac{a+b}{2}}
    \eta\bigl(x-\tfrac{a+b}{2}\bigr)\D t.
\]
Since $\eta(t)=0$ for $t\in(0,(a-b)/2)$, we therefore obtain, for almost every $x>a$,
\begin{equation}\label{eq:derivative_eta}
    \frac{\D}{\D x}\eta(x)
    =
    \frac{1}{(x-a)^2}
    \bigg(
        \int_0^{x-a}
        \bigl(\eta(x-a)-\eta(t)\bigr)\D t
        +
        \int_{\tfrac{a-b}{2}}^{x-\tfrac{a+b}{2}}
        \bigl(
            \eta\bigl(x-\tfrac{a+b}{2}\bigr)-\eta(t)
        \bigr)\D t
    \bigg).
\end{equation}

By the initial values established above, $x\mapsto\eta(x)$ is non-decreasing on $(0,a]$. Moreover, $\eta\ge 0$, and hence~\eqref{eq:defn_eta_x} shows that $\eta(x)\ge 1=\eta(a)$ for every $x>a$. Thus, in order to extend the monotonicity beyond $a$, it suffices, in view of~\eqref{eq:derivative_eta}, to verify that
\begin{equation}\label{eq:suff_cond_non-decreasing}
    \eta(x-a)-\eta(t)\ge 0
    \quad\text{for all }t\in(0,x-a), \quad \text{and}\quad 
    \eta\bigl(x-\tfrac{a+b}{2}\bigr)-\eta(t)\ge 0
    \quad\text{for all }
    t\in
    \bigl(
        \tfrac{a-b}{2},
        x-\tfrac{a+b}{2}
    \bigr).
\end{equation}

Since $\eta$ is non-decreasing on $(0,a]$, both equations in~\eqref{eq:suff_cond_non-decreasing} hold whenever $x-(a+b)/2<a$, that is, whenever $ a<x<a+(a+b)/2 = (3a+b)/2$. Indeed, since $(a+b)/2\le a$, we also have $x-a\le x-(a+b)/2<a$. Hence $\eta'(x)\ge 0$ for almost every $x\in(a,(3a+b)/2)$. Since $\eta$ is absolutely continuous on every compact subinterval of $(a,\infty)$, it follows that $\eta$ is non-decreasing on $(a,(3a+b)/2)$. Together with the initial values and the fact that $\eta(x)\ge \eta(a)$ for $x>a$, this shows that $\eta$ is non-decreasing on $(0,(3a+b)/2)$, i.e. on $(0,a+(a+b)/2)$.

We now proceed by induction. For $n\ge 1$, let the $n$-th induction hypothesis be
\begin{equation}\label{eq:induction_proof}
    x\mapsto\eta(x)
    \text{ is non-decreasing on }
    \biggl(
        0,
        a+n\frac{a+b}{2}
    \biggr).
\end{equation}
The preceding argument proves the case $n=1$. Assume that~\eqref{eq:induction_proof} holds for some $n\ge 1$. Let $a+n(a+b)/2<x<a+(n+1)(a+b)/2$. Then, $x-(a+b)/2<a+n(a+b)/2$. Moreover, since $(a+b)/2\le a$, it follows that $ x-a \le x-(a+b)/2< a+n(a+b)/2$. Thus, by the induction hypothesis,~\eqref{eq:suff_cond_non-decreasing} holds for this $x$. It follows from~\eqref{eq:derivative_eta} that $\eta'(x)\ge 0$ for almost every $x\in (a+n(a+b)/2, a+(n+1)(a+b)/2)$. Since $\eta$ is locally absolutely continuous on $(a,\infty)$, for any $a+n(a+b)/2<u<v< a+(n+1)(a+b)/2$, we have $\eta(v)-\eta(u) = \int_u^v \eta'(x)\D x \ge 0$. Hence $\eta$ is non-decreasing on this interval. Since $\eta$ is continuous on $(a,\infty)$, this monotonicity joins continuously with the induction hypothesis at $a+n(a+b)/2$, and therefore~\eqref{eq:induction_proof} holds with $n$ replaced by $n+1$.

By induction, $x\mapsto\eta(x)$ is non-decreasing on $(0,\infty)$, completing the proof.
\end{proof}

\subsection{Exact parking constant}

\begin{lemma}\label{lem:parking_constant_a,b}
    For all $0 \le b \le a <\infty$ such that $a >0$, and $\eta$ defined in~\eqref{eq:defn_eta}, it follows that
\begin{equation}\label{eq:parking_const_number_dep_2}
        \lim_{x \to \infty}\frac{\eta(x)}{x} = \lambda(a,b)= \frac{1}{2} \int_0^\infty (2+e^{-tb}-e^{-ta})\exp \bigg( \int_0^t \frac{-2+e^{-s(a+b)/2}+e^{-as}}{s} \D s\bigg) \D t \in (0,\infty).
    \end{equation} 
\end{lemma}

For the proof of Lemma~\ref{lem:parking_constant_a,b}, we define the  ``$a$-truncated'' Laplace transform of $\eta$, by
\begin{equation}\label{eq:truncated_laplace_transform}
    F(x)\coloneqq \int_a^\infty \eta(s)e^{-s x} \D s,
\end{equation} which is well-defined, since $s \mapsto \eta(s)e^{-s x}\in \mathcal{L}^1((a,\infty))$ for all $x>0$.

\begin{proof}[Proof of Lemma~\ref{lem:parking_constant_a,b}]
    From Lemma~\ref{lem:eta_func} we have that $\eta$ is locally absolutely continuous on $(a,\infty)$, and hence by~\eqref{eq:star_3_pictres}, it follows for almost every $x>a$, that \begin{equation}\label{eq:star_3_pictres_alt}
        \eta(x)+(x-a)\eta'(x)=1+\eta(x-a)+\eta\big(x-\tfrac{a+b}{2}\big).
    \end{equation} 
    Note that $s \mapsto (\eta(s)+(s-a)\eta'(s))e^{-s x} \in \mathcal{L}^1((a,\infty))$ for all $x >0$. Indeed, by~\eqref{eq:star_3_pictres_alt} it is equivalent to show that $s \mapsto (1+\eta(s-a)+\eta(s-(a+b)/2)) e^{-s x} \in \mathcal{L}^1((a,\infty))$ for all $x >0$, which follows directly from the linear growth of $\eta$ as shown in Remark~\ref{rem:det_bound_eta}. Thus, $(s-a)\eta'(s)e^{-s x} \in \mathcal{L}^1((a,\infty))$ by Remark~\ref{rem:det_bound_eta} and since $\eta(s)$ is positive. Hence, it follows from~\eqref{eq:star_3_pictres_alt} and Remark~\ref{rem:det_bound_eta} that
\begin{equation}\label{eq:integral_form_Laplace}
    F(x)+\int_a^\infty (s-a)\eta'(s)e^{-sx} \D s 
    = 
    \int_a^\infty e^{-s x} \D s
    +
    \int_a^\infty \eta(s-a) e^{-s x} \D s
    +
    \int_a^\infty \eta(s-\tfrac{a+b}{2}) e^{-s x} \D s,
\end{equation} for all $x>0$.
Each integral in~\eqref{eq:integral_form_Laplace} is finite and well-defined, as argued above. By integration by parts, each integral in~\eqref{eq:integral_form_Laplace} can be written in terms of $F(x)$. Since $\eta$ is locally absolutely continuous on $(a,\infty)$, we may apply integration by parts on $[a+\varepsilon,R]$, for every $\varepsilon>0$ and $R>a+\varepsilon$. Letting $\varepsilon\downarrow0$ and $R\to\infty$, the boundary terms vanish by the linear growth bound from Remark~\ref{rem:det_bound_eta}, and hence
\begin{align}
    I_1(x)\coloneqq \int_a^\infty (s-a)\eta'(s)e^{-sx} \D s
    =
    x \int_a^\infty s\eta(s)e^{-sx}\D s-(1+ax)F(x).
\end{align} Note, that $s \mapsto \eta(s)e^{-sx} \in \mathcal{L}^1((a,\infty))$ for all $x>0$ and $|\tfrac{\D }{\D x}\eta(s)e^{-sx}|=|s\eta(s)e^{-sx}|\le s\eta(s)e^{-s}$ for all $x \ge 1$, where $s \mapsto s\eta(s)e^{-s} \in \mathcal{L}^1((a,\infty))$. If we instead consider $x \in (0,1)$, then we may fix $x_0>0$ and consider all $x \in [x_0,1)$, for which $|\tfrac{\D }{\D x}\eta(s)e^{-sx}| \le s\eta(s)e^{-s x_0/2}$, where $s \mapsto s\eta(s)e^{-s x_0 /2} \in \mathcal{L}^1((a,\infty))$. Since $x_0>0$ is arbitrary, we have that $|\tfrac{\D }{\D x}\eta(s)e^{-sx}|$ is in $\mathcal{L}^1((a,\infty))$ for all $x>0$. Hence, when calculating $F'(x)$, we can change the order of integration and differentiation, implying that $F'(x)=- \int_a^\infty s\eta(s)e^{-sx} \D s$ for all $x >0$. Thus, 
\begin{equation}
    I_1(x)=-xF'(x)-(1+ax)F(x), \quad \text{ for all }x >0.
\end{equation} Next, we get directly that $I_2(x)\coloneqq \int_a^\infty e^{-s x} \D s = e^{-ax}/x$ for all $x>0$. By change of variables, it follows that
\begin{align}
    I_3(x)&\coloneqq \int_a^\infty \eta(s-a) e^{-s x} \D s 
    =
    e^{-ax}\int_0^\infty \eta(s) e^{-sx} \D s \\
    &=
    e^{-ax}\bigg( \frac{1}{2}\int_{\tfrac{a+b}{2}}^a e^{-sx} \D s+F(x)\bigg)
    =e^{-a x}\bigg( \frac{1}{2x} \bigg( e^{-x\big(\tfrac{a+b}{2}\big)}-e^{-x a}\bigg)+F(x)\bigg).
\end{align} Similarly, it follows that
\begin{equation}
    I_4(x)\coloneqq \int_a^\infty \eta\big(s-\tfrac{a+b}{2}\big) e^{-s x} \D s 
    =
    e^{-\big( \tfrac{a+b}{2}\big) x}\bigg( \frac{1}{2x} \bigg( e^{-x\big(\tfrac{a+b}{2}\big)}-e^{-x a}\bigg)+F(x)\bigg).
\end{equation} Using~\eqref{eq:integral_form_Laplace}, together with $I_1(x)$, $I_2(x)$, $I_3(x)$ and $I_4(x)$, it follows that
\begin{equation}\label{eq:star_5_pictures_cambridge}
    F'(x)+F(x)\mleft( a+\frac{e^{-x\big(\tfrac{a+b}{2} \big)}}{x}+\frac{e^{-a x }}{x} \mright)= \frac{1}{2x^2}\bigg( e^{-2xa}-2e^{-xa}-e^{-x(a+b)}\bigg), \quad \text{ for all }x >0.
\end{equation} Next, define $\mathcal{I}$ as follows:
\begin{equation}
    \mathcal{I}(x) \coloneqq \exp \mleft\{ \int_1^x \bigg( as+e^{-s \big( \tfrac{a+b}{2}\big)}+e^{-s a} \bigg)\frac{\D s}{s}\mright\}, \quad \text{ for all }x>0.
\end{equation} By definition of $\mathcal{I}$ and~\eqref{eq:star_5_pictures_cambridge}, we see that
\begin{equation}\label{eq:integrable_product_f_I}
    \frac{\D}{\D x} F(x)\mathcal{I}(x)
    =
    F'(x)\mathcal{I}(x)+\frac{F(x)\mathcal{I}(x)}{x}\mleft( ax+e^{-x\big(\tfrac{a+b}{2} \big)}+e^{-a x } \mright)
    = 
    \frac{\mathcal{I}(x)}{2x^2}\mleft( e^{-2xa}-2e^{-xa}-e^{-x(a+b)}\mright),
\end{equation} for all $x>0$. Next, we show that the right-hand side of~\eqref{eq:integrable_product_f_I} belongs to $\mathcal L^1((0,\infty))$. We consider separately the regions $(0,1]$ and $[1,\infty)$. First, let $x\in(0,1]$. Recall the definition of $\mathcal{I}$. Using the elementary inequality $e^{-u}\ge  1-u$, we obtain $e^{-s(a+b)/2}+e^{-as}
\ge 
2-((a+b)/2+a)s$ for all $s>0$. Since $x\le 1$, reversing the limits of integration yields
\begin{align*}
\int_1^x \frac{ e^{-s(a+b)/2}+e^{-as}}{s} \D s 
&\le
\int_1^x \mleft(\frac{2}{s}-\frac{a+b}{2}-a\mright)\D s 
=
2\log x - \mleft( \frac{a+b}{2}+a \mright)(x-1).
\end{align*} 
Consequently,
\begin{align*}
\log\mathcal I(x)
&\le
a(x-1)+2\log x - \mleft( \frac{a+b}{2}+a \mright)(x-1) 
=
2\log x+\frac{a+b}{2}(1-x)
\le
2\log x+\frac{a+b}{2}.
\end{align*}
Hence, $\mathcal I(x)
    \le
    e^{(a+b)/2}x^2$ for all $x \in (0,1]$. It follows for all $x \in (0,1]$, that
\begin{align*}
&\mleft| \frac{\mathcal I(x)}{2x^2} \mleft( e^{-2ax}-2e^{-ax}-e^{-(a+b)x} \mright) \mright| 
\le
\frac{e^{(a+b)/2}}{2} \mleft( e^{-2ax}+2e^{-ax}+e^{-(a+b)x} \mright)
\le
2e^{(a+b)/2},
\end{align*}
where we used $a,b>0$. Therefore,
\begin{equation}\label{eq:integrability_near_zero}
\int_0^1 \mleft| \frac{\mathcal I(x)}{2x^2} \mleft( e^{-2ax}-2e^{-ax}-e^{-(a+b)x} \mright) \mright| \D x
<\infty.
\end{equation}

It remains to consider $x\ge 1$. Since $x\mapsto \int_1^x (e^{-s(a+b)/2}+e^{-as})s^{-1} \D s$ is an increasing function and $\int_1^\infty e^{-cs} s^{-1} \D s = -\Ei(-c)$ for all $c>0$, we have
\begin{equation}\label{eq:bound_I_(x)}
    \mathcal I(x)
    \le
    C_{a,b}e^{ax}, \quad \text{where}\quad C_{a,b}
\coloneqq
\exp \mleft\{
-a-\Ei \mleft(-\frac{a+b}{2}\mright)-\Ei(-a)
\mright\}
<\infty,
    \quad\text{for all } x\ge 1.
\end{equation}
Thus,
\begin{align*}
&\mleft| \frac{\mathcal I(x)}{2x^2} \mleft( e^{-2ax}-2e^{-ax}-e^{-(a+b)x} \mright) \mright| 
\le
\frac{C_{a,b}}{2x^2} \mleft( e^{-ax}+2+e^{-bx} \mright)
\le
\frac{2C_{a,b}}{x^2}, \qquad \text{for all } x\ge 1.
\end{align*}
Since $x\mapsto x^{-2} \in \mathcal{L}^1((1,\infty))$, we conclude that
\begin{equation}\label{eq:integrability_infinity}
\int_1^\infty \mleft| \frac{\mathcal I(x)}{2x^2} \mleft( e^{-2ax}-2e^{-ax}-e^{-(a+b)x} \mright) \mright| \D x
<\infty.
\end{equation}
Combining~\eqref{eq:integrability_near_zero}
and~\eqref{eq:integrability_infinity}, we obtain $x \mapsto \mathcal I(x)(2x^2)^{-1}( e^{-2ax}-2e^{-ax}-e^{-(a+b)x} ) \in \mathcal L^1((0,\infty))$. In particular,
\[
\lim_{t\downarrow0} \int_t^\infty \mleft| \frac{\mathcal I(x)}{2x^2} \mleft( e^{-2ax}-2e^{-ax}-e^{-(a+b)x} \mright) \mright| \D x
=
\int_0^\infty \mleft| \frac{\mathcal I(x)}{2x^2} \mleft( e^{-2ax}-2e^{-ax}-e^{-(a+b)x} \mright) \mright| \D x
<\infty.
\]

Note that $\lim_{x \to \infty } F(x)\mathcal{I}(x)=0$. Indeed, by Remark~\ref{rem:det_bound_eta} together with~\eqref{eq:bound_I_(x)}, it suffices to show that $e^{ax}\int_a^\infty t e^{-tx} \D t \to 0$ as $x \to \infty$, which follows directly by calculating the integral:
\begin{equation}
    e^{ax}\int_a^\infty te^{-tx} \D t= \frac{1 + ax}{x^2} \to 0, \quad \text{ as }x \to \infty.
\end{equation} Hence, since $\lim_{x \to \infty } F(x)\mathcal{I}(x)=0$, integrating on both sides of~\eqref{eq:integrable_product_f_I} implies that
\begin{equation}
    F(x)\mathcal{I}(x)= \int_x^\infty \frac{\mathcal{I}(t)}{2t^2 }e^{-ta}\mleft( 2+e^{-tb}-e^{-ta}\mright) \D t, \quad \text{ for all }x>0.
\end{equation} Dividing through with $\mathcal{I}(x)$ and using the definition of $\mathcal{I}$ implies, for all $x \ge 1$, that
\begin{align*}
    F(x)
    &=
    \frac{1}{2}\int_x^\infty \frac{e^{-ta}}{t^2}(2+e^{-bt}-e^{-at})\exp \mleft\{ \int_x^t \bigg( as+e^{-s \big( \tfrac{a+b}{2}\big)}+e^{-s a} \bigg)\frac{\D s}{s}\mright\} \D t\\
    &=
    \frac{e^{-xa}}{2x^2}\int_x^\infty \frac{x^2}{t^2}(2+e^{-bt}-e^{-at})\exp \mleft\{ \int_x^t \bigg(e^{-s \big( \tfrac{a+b}{2}\big)}+e^{-s a} \bigg)\frac{\D s}{s}\mright\} \D t.
\end{align*} Next, using that $x^2/t^2=\exp \big(-\int_x^t 2s^{-1} \D s\big)$, it follows that
\begin{equation}\label{eq:reduced_F(x)_form}
    F(x)
    =
    \frac{e^{-xa}}{2x^2}\int_x^\infty (2+e^{-bt}-e^{-at})\exp \mleft\{ \int_x^t \bigg(-2+e^{-s \big(\tfrac{a+b}{2}\big)}+e^{-s a} \bigg)\frac{\D s}{s}\mright\} \D t.
\end{equation} Consider the Laplace--Stieltjes transform of $\eta$, defined by $\wh \eta(x)\coloneqq \int_{[0,\infty)} e^{-sx}\D\eta(s)=x\int_0^\infty \eta(s)e^{-sx}\D s$ for $x>0$. Using the definition of $F$,~\eqref{eq:reduced_F(x)_form}, and Lemma~\ref{lem:eta_func}, it follows that
\begin{equation*}
    \wh\eta (x) 
    =
    xF(x)+\frac{x}{2}\int_{\tfrac{a+b}{2}}^a e^{-sx }\D s 
    =
    \frac{1}{2}\mleft( e^{-x(a+b)/2}-e^{-x a} \mright) +xF(x) \sim \frac{\lambda(a,b)}{x}, \quad \text{ as } x \da 0,
\end{equation*} where $\lambda(a,b)$ is defined in~\eqref{eq:parking_const_number_dep}. We assume for now that $\lambda(a,b)\in (0,\infty)$, but we will show this as the final part of the proof. Using the Tauberian theorem for Laplace--Stieltjes transforms~\cite[Thm~1.7.1]{MR1015093} implies that $\eta(x) \sim \lambda(a,b)x$ as $x \to \infty$, and hence $\lim_{x \to \infty } \eta(x)/x=\lambda(a,b)$. 

Finally, it only remains to prove that $\lambda(a,b)\in (0,\infty)$. It follows directly from the definition of $\lambda(a,b)$ in~\eqref{eq:parking_const_number_dep} that $\lambda(a,b) > 0$. Hence, it remains to prove that $\lambda(a,b)<\infty$. To show this, we note by~\cite[Eq.~(6.2.1),~(6.2.3) \&~(6.2.4)]{NIST:DLMF} together with $\Gamma(0,x)=\int_x^\infty e^{-u}u^{-1} \D u$ for all $x>0$, that
\begin{equation}\label{eq:Walpha_equal}
    \int_0^t \frac{1-e^{-cu}}{u} \D u = \log(ct)+\Gamma(0,ct)+\gamma, \quad \text{ for all }t>0 \text{ and }c>0,
\end{equation} where $\gamma=-\int_0^\infty e^{-u}\log(u)\D u$ is the Euler-Mascheroni constant. Note that the right-hand side of~\eqref{eq:Walpha_equal} is finite for all $t>0$. 

We can now use~\eqref{eq:Walpha_equal} to show that $\lambda(a,b)<\infty$. Note at first that
\begin{equation}
    \int_0^1 (2+e^{-tb}-e^{-ta})\exp \bigg( \int_0^t \frac{-2+e^{-s(a+b)/2}+e^{-as}}{s} \D s\bigg) \D t
    \le 
    3,
\end{equation} since $e^{-tb}-e^{-ta}\le 1$ for all $t >0$, and since $\int_0^t (-2+e^{-s(a+b)/2}+e^{-as})s^{-1} \D s \le 0$ for all $t\in (0,1)$. Thus, it follows that
\begin{equation*}
    \frac{2\lambda(a,b) }{3}-1 \le \int_1^\infty \exp \bigg( \int_0^t \frac{-2+e^{-s(a+b)/2}+e^{-as}}{s} \D s\bigg) \D t \eqqcolon \wt \lambda(a,b),
\end{equation*} and it therefore suffices to show that $\wt\lambda(a,b)<\infty$. By~\eqref{eq:Walpha_equal}, it follows that
\begin{align*}
    \wt\lambda(a,b) 
    &=
    \int_1^\infty \exp \mleft( -\int_0^t \frac{1-e^{-s(a+b)/2}}{s} \D s - \int_0^t \frac{1-e^{-sa}}{s} \D s\mright) \D t\\
    &=
    e^{-2\gamma-\log\big(\tfrac{a^2+ab}{2}\big)}\int_1^\infty e^{-\Gamma\big(0,\tfrac{a+b}{2}t\big)-\Gamma(0,at)} \frac{\D t}{t^2}
    \le 
    e^{-2\gamma-\log\big(\tfrac{a^2+ab}{2}\big)}<\infty.
\end{align*} In the last inequality, we used that $e^{-\Gamma(0,t(a+b)/2)-\Gamma(0,at)} \le 1$ for all $t\ge 1$, and that $\int_1^\infty t^{-2}\D t=1$.
\end{proof} 

\subsection{Quantitative convergence}

\begin{lemma}\label{lem:ROC_first_order}
    There exists some $\ov\lambda(a,b) \in \R$, such that, for all $x > 2a(e+1)$, it holds that
    \begin{equation}\label{eq:gen_roc_eta_tech}
    \bigg|\eta(x) -\ov\lambda(a,b) x +1-\ov\lambda(a,b) \frac{a+b}{2}\bigg|
    \le \begin{dcases}
      \mleft(\frac{2e}{x/a}\mright)^{x/a-3/2} \frac{e^{3/2} K_1}{\sqrt{\pi}}, &\text{ if }a=b,\\
      \mleft(\frac{2e}{x/a}\mright)^{x/a-5/2} \frac{3e^{5/2}(2K_1+K_2)}{\sqrt{\pi}}, &\text{ if }a>b.
    \end{dcases}
\end{equation} where $K_1 \coloneqq 1 + 2a(2/(a+b)+\ov\lambda(a,b)) + \ov\lambda(a,b)\cdot  (a+b)/2>0$ and $K_2 \coloneqq (2/(a+b)+\ov\lambda(a,b))(a-b)/2\ge 0$.
\end{lemma}
For the proof of Lemma~\ref{lem:ROC_first_order}, we follow the idea of the proof from~\cite{Dvoretzky_Robbins}; however, due to the complexity of our setup, many additional technical additions and variations are needed.
\begin{proof}[Proof of Lemma~\ref{lem:ROC_first_order}]
Recall from~\eqref{eq:defn_eta_x} from Lemma~\ref{lem:eta_func}, that,
    \begin{equation*}
        \eta(x + a) = 1 + \frac{1}{x} \mleft( \int_0^x \eta(t)\D t + \int_{\frac{a-b}{2}}^{x + \frac{a-b}{2}} \eta(t)\D t \mright), \quad \text{ for all }x>0.
    \end{equation*}
Now, for $f(x) = \eta(x) + 1$, for all $x>0$, we can apply~\eqref{eq:defn_eta_x} with $\eta(x+a)$, and see that
    \begin{equation}\label{eq:f(x+a)}
        \begin{aligned}
            f(x+a)
            & = \eta(x+a) + 1 = 2 + \frac{1}{x} \mleft( \int_0^x (f(t) - 1)\D t + \int_{\frac{a-b}{2}}^{x + \frac{a-b}{2}} (f(t) - 1)\D t \mright) \\
            & = \frac{1}{x} \mleft( \int_0^x f(t)\D t + \int_{\frac{a-b}{2}}^{x + \frac{a-b}{2}} f(t)\D t \mright), \quad \text{ for all }x>0.
        \end{aligned}
    \end{equation}
Next, applying~\eqref{eq:f(x+a)}, we note that for $0 < x \le y$, it holds that
    \begin{equation}\label{eq:f(y+a)-(x/y)f(x+a)}
        f(y+a) - \frac{x}{y}f(x+a) = \frac{1}{y} \mleft( \int_x^y f(t)\D t + \int_{x+\frac{a-b}{2}}^{y + \frac{a-b}{2}} f(t)\D t \mright).
    \end{equation}
It is now straightforward to check that $f(x) = x + (a+b)/2$ satisfies the equation
    \begin{equation*}
        f(x+a) = \frac{1}{x} \mleft( \int_0^x f(t)\D t + \int_{\frac{a-b}{2}}^{x + \frac{a-b}{2}} f(t)\D t \mright).
    \end{equation*}
Hence, we may plug $f(x) = x + (a+b)/2$ into~\eqref{eq:f(y+a)-(x/y)f(x+a)}, which yields
    \begin{equation}\label{eq:f(y+a)-(x/y)f(x+a)-precise_f}
        y + \frac{3a+b}{2} = \frac{x}{y} \mleft( x + \frac{3a+b}{2} \mright) + \frac{1}{y} \mleft( \int_x^y \mleft( t + \frac{a+b}{2} \mright)\D t + \int_{x+\frac{a-b}{2}}^{y + \frac{a-b}{2}} \mleft( t + \frac{a+b}{2} \mright)\D t \mright).
    \end{equation}
Subtracting~\eqref{eq:f(y+a)-(x/y)f(x+a)-precise_f} multiplied with some $A \in \R$ from equation~\eqref{eq:f(y+a)-(x/y)f(x+a)}, then implies that
    \begin{equation}\label{eq:star_9}
       \begin{aligned}
            f(y+a)
        & - A\mleft( y + \frac{3a+b}{2} \mright) = \frac{x}{y} \mleft( f(x+a) - A\mleft( x + \frac{3a+b}{2} \mright) \mright) \\
        & \quad + \frac{1}{y} \mleft\{ \int_x^y \mleft( f(t) - A\mleft( t + \frac{a+b}{2} \mright) \mright)\D t + \int_{x+\frac{a-b}{2}}^{y + \frac{a-b}{2}} \mleft( f(t) - A\mleft( t + \frac{a+b}{2} \mright) \mright)\D t \mright\}.
       \end{aligned}
    \end{equation}
Next, we define 
\begin{equation}\label{eq_defn_I_x,S_x}
    I_x \coloneqq \inf_{x \le c \le x+a} \frac{f(c)}{c+(a+b)/2}, \quad \text{ and } \quad S_x \coloneqq \sup_{x \le c \le x+a} \frac{f(c)}{c+(a+b)/2}, \quad \text{ for all }x\ge0.
\end{equation}
Note by definition of $I_x$, that $f(y)-I_x\cdot (y+(a+b)/2)\ge 0$ for all $y \in [x,x+a]$. Hence, using~\eqref{eq:star_9} with $A=I_x$, we see that $f(y+a)-I_x\cdot (y+(3a+b)/2)\ge 0$ for all $y \in [x,x+(a+b)/2]$, which implies that $f(y)-I_x\cdot (y+(a+b)/2)\ge 0$ for all $y \in [x+a,x+(3a+b)/2]$. Thus, it follows altogether that 
\begin{equation}\label{eq:ineq_i_x_i_y}
    \frac{f(y)}{(y+(a+b)/2)}\ge I_x, \quad \text{ for all } y \in \mleft[ x,x+\frac{3a+b}{2}\mright].
\end{equation} We can now use~\eqref{eq:ineq_i_x_i_y}, to prove that $I_x \le I_z$ for all $z \in [x,x+(a+b)/2]$. Indeed, fix some $z \in [x,x+(a+b)/2]$, in which case $[z,z+a] \subset [x,x+(3a+b)/2]$, since $a+b\le (3a+b)/2$. Then, taking infimum over all $y \in [z,z+a]$ in~\eqref{eq:ineq_i_x_i_y}, implies that $I_x \le I_z$. Moreover, this local monotonicity implies that $I_x \le I_y$ for arbitrary $0\le x\le y<\infty$. Indeed, let $N \coloneqq \lceil 2(y-x)/(a+b)\rceil $, and, if $x<y$, set $x_k\coloneqq x+ k(y-x)/N$ for all $k=0,\ldots, N$. Then, $x_{k+1}-x_k\le (a+b)/2$, and therefore the preceding argument gives $I_{x_k}\le I_{x_{k+1}}$ for every $k=0,\ldots,N-1$. Consequently, $I_x=I_{x_0}\le I_{x_1}\le\cdots\le I_{x_N}=I_y$. Thus, $I_x\le I_y$ for all $0\le x\le y<\infty$.

Following the exact same steps for $S_x$, again using~\eqref{eq:star_9} and definition~\eqref{eq_defn_I_x,S_x}, it follows that $S_x \ge S_y$ for all $0\le x \le y<\infty$. We now show that $I_\infty \coloneqq \lim_{x \to \infty} I_x$ and $S_\infty \coloneqq \lim_{x \to \infty} S_x$ exist and that $-\infty < I_\infty \le S_\infty <\infty$. First, since $I_x \le I_y$ for all $0 \le x \le y<\infty$, it follows that $I_x \le \liminf_{y \to \infty} I_y$ for all $x \ge 0$. Hence, it follows that
\begin{equation*}
    \liminf_{x \to \infty} I_x \le \limsup_{x \to \infty} I_x \le \liminf_{y \to \infty} I_y,
\end{equation*} which therefore yields that $\lim_{x \to \infty} I_x =I_\infty$ exists in $\bar{\R}$. Moreover, by definition of $\eta$, the quantity $I_0$ can be calculated exactly, and it follows that, for all $x \ge 0$, $ I_\infty \ge I_x \ge I_0>-\infty$. Hence, $I_\infty$ exists and $I_\infty >-\infty$. By an analogous argument it follows that $S_\infty$ exists and that $S_\infty <\infty$, and hence 
\begin{equation}\label{eq:bounded_over_lambda_ixsx}
    -\infty < I_0 \le I_x \le I_\infty \le S_\infty \le S_y \le S_0 < \infty, \quad \text{ for all }x,y\ge 0.
\end{equation}

We will now show that $I_\infty =S_\infty$. For all $x,y >0$, we have from~\eqref{eq:f(y+a)-(x/y)f(x+a)}, that
\begin{equation*}
    f(y+a)-f(x+a)=\frac{x-y}{y} f(x+a) + \frac{1}{y} \mleft\{  \int_{x+\frac{a-b}{2}}^{y+\frac{a-b}{2}} f(t)\D t + \int_{x}^{y} f(t)\D t \mright\}.
\end{equation*} From Remark~\ref{rem:det_bound_eta}, we recall that $\eta(x) \le 2x/(a+b)$ for all $x \ge 0$, implying that for $c \coloneqq \max\{2/(a+b), 1\}$ we have $f(x) \le c(x + 1)$ for all $x \ge 0$. Using this, it follows from the display above that
\begin{align*}
    &\sup_{x \le y \le x+a } |f(y+a)-f(x+a)| \\ &\qquad \le  c \sup_{x \le y \le x+a }  \mleft( \frac{|x-y| (x+a+1)}{y}+\frac{1}{y} \mleft\{  \int_{x+\frac{a-b}{2}}^{y+\frac{a-b}{2}} (t+1)\D t + \int_{x}^{y} (t+1) \D t \mright\} \mright)\\
    &\qquad = c \sup_{x \le y \le x+a } \mleft( \frac{|y-x|(x+a+1)}{y}+\frac{ (y-x)(y+x+a-b)+(y-x)(y+x)}{2y}  + \frac{2(y-x)}{y}\mright).
\end{align*}
Next, under the supremum, we note that $y-x \le a$, $y+x \le 2x+a$, and $1/y \le 1/x$ for all $x>0$, which implies that
\begin{equation}\label{eq:sup_norm_bound_f}
     \sup_{x \le y \le x+a } |f(y+a)-f(x+a)| \le c \mleft(  3a + \frac{5a^2 + 6a -ab}{2x} \mright) \eqqcolon \wt c + \frac{\wh c}{x}, \quad \text{ for all }x>0.
\end{equation} Recall from Lemma~\ref{lem:eta_func}, that $\eta$ is a non-decreasing function, implying the same for $f$. Hence, $\sup_{x \le y \le x+a } |f(y+a)-f(x+a)|=\sup_{x \le y \le x+a } f(y+a)-f(x+a)$. By definition of $S_x$ from~\eqref{eq_defn_I_x,S_x} and by~\eqref{eq:sup_norm_bound_f}, it follows for all $x>0$, that
\begin{align*}
    S_{x+a}-\frac{f(x+a)}{x+a+(a+b)/2}&= \sup_{x \le y \le x+a} \frac{f(y+a)}{y+a+(a+b)/2}-\frac{f(x+a)}{x+a+(a+b)/2}\\
    &\le  \frac{\sup_{x \le y \le x+a} f(y+a)-f(x+a)}{x+a+(a+b)/2} \le \frac{\wt c+\wh c /x}{x+a+(a+b)/2}.
\end{align*} Thus, for all $x>a$, we may shift from $x+a$ to $x$, in the equation above, implying that
\begin{equation*}
    S_x- \frac{f(x)}{x+(a+b)/2} \le \frac{\wt c + \wh c /(x-a)}{x+(a+b)/2}, \quad \text{ for all }x>a.
\end{equation*} Thus, taking supremum, we see, that
\begin{align*}
    0 \le S_x-I_x &\le \sup_{x \le t \le x+a} S_t - \inf_{x \le t \le x+a} \frac{f(t)}{t+(a+b)/2} \le \sup_{x \le t \le x+a} \mleft( S_t-\frac{f(t)}{t+(a+b)/2}\mright)\\
    &\le \sup_{x \le t \le x+a} \frac{\wt c + \wh c /(t-a)}{t+(a+b)/2} \le \frac{\wt c + \wh c /(x-a)}{x+(a+b)/2}, \quad \text{ for all }x>a.
\end{align*} Hence, taking $x \to \infty$, it follows that $S_\infty = I_\infty$, as we wanted to show, and we will now call this quantity $\ov{\lambda}(a,b)=\ov \lambda$ where we suppress the dependence on $a$ and $b$.

Since $\eta$ is continuous for all $x>a$, we have that $f$ is a continuous function for all $x>a$, and will therefore attain its infimum and supremum on any compact interval. Hence, since $I_x \le \ov \lambda \le S_x$, it follows for all $x >a$ that there exists some $x' \in [x,x+a]$, such that 
\begin{equation*}
    \mleft| \frac{f(x')}{x'+(a+b)/2}-\ov \lambda\mright|=0, \qquad \text{ i.e.\ }  \qquad f^*(x')\coloneqq f(x')-\ov \lambda \cdot (x'+(a+b)/2)=0.
\end{equation*} Now, consider some $\varpi>a$, and the interval $[\varpi,\varpi+a]$, where $x_\varpi \in [\varpi,\varpi+a]$ is given such that $f^*(x_\varpi)=0$. Let $y+a \in [x_\varpi,\varpi+a]$, then~\eqref{eq:star_9} with $A=\ov \lambda$, implies that
\begin{align*}
    |f^*(y+a)|
    &=
    \mleft|\frac{x_\varpi-a}{y} f^*(x_\varpi)+\frac{1}{y} \mleft\{ \int_{x_\varpi-a}^{y} f^*(t)\D t+\int_{x_\varpi-a+\frac{a-b}{2}}^{y+\frac{a-b}{2}} f^*(t) \D t \mright\} \mright|\\
    &\le 
    \frac{y-x_\varpi+a}{y} \mleft( \sup_{x_\varpi-a \le t \le y} |f^*(t)|+ \sup_{x_\varpi-a+\frac{a-b}{2}\le t \le y+\frac{a-b}{2}}|f^*(t)|\mright).
\end{align*} Next, since $y-x_\varpi+a \le \varpi+a-x_\varpi \le a$, it follows that
\begin{equation*}
    |f^*(y+a)| \le \frac{a}{\varpi-a} \mleft( \sup_{\varpi-a \le t \le \varpi} |f^*(t)|+ \sup_{\varpi-a+\frac{a-b}{2}\le t \le \varpi +\frac{a-b}{2}}|f^*(t)| \mright).
\end{equation*} Thus, for $\mathfrak{D}(z) \coloneqq \sup_{z \le t \le z+a}|f^*(t)|$, it holds that
\begin{equation}\label{eq:part_1_f^*_bound}
    \sup_{y \in [x_\varpi,\varpi + a]}|f^*(y)|=\sup_{y+a \in [x_\varpi, \varpi+a]}|f^*(y+a)| \le \frac{a}{\varpi-a} (\mathfrak{D}(\varpi-a)+\mathfrak{D}(\varpi-a+(a-b)/2)).
\end{equation}
Similarly, if $x_\varpi \in [\varpi, \varpi+a]$ and $x+a \in [\varpi,x_\varpi]$, then~\eqref{eq:star_9} with $A=\ov \lambda$ and $y=x_\varpi-a$, implies that 
\begin{equation*}
    0=f^*(x_\varpi) = \frac{x}{x_\varpi-a} f^*(x+a)+ \frac{1}{x_\varpi-a} \mleft\{ \int_x^{x_\varpi-a} f^*(t) \D t+ \int_{x+\frac{a-b}{2}}^{x_\varpi-a+\frac{a-b}{2}} f^*(t)\D t\mright\}.
\end{equation*} Since,  $x_\varpi \in [\varpi,\varpi+a]$ and $x+a \in [\varpi,x_\varpi]$, we have that $x_\varpi -a-x \le a$ and $1/x \le 1/(\varpi-a)$, giving us
\begin{align*}
    |f^*(x+a)| &\le \frac{x_\varpi-a-x}{x} \mleft( \sup_{ x \le t \le x_\varpi-a} |f^*(t)|+\sup_{x+\frac{a-b}{2}\le t \le x_\varpi-a+\frac{a-b}{2}}|f^*(t)|\mright)\\
    &\le \frac{a}{\varpi-a}  \mleft( \sup_{ \varpi-a \le t \le \varpi} |f^*(t)|+\sup_{\varpi-a+\frac{a-b}{2}\le t \le \varpi+\frac{a-b}{2}}|f^*(t)|\mright).
\end{align*} Altogether, we see that
\begin{equation}\label{eq:part_2_f^*_bound}
    \sup_{x \in [\varpi,x_\varpi]}|f^*(x)|=\sup_{x+a \in [\varpi,x_\varpi]}|f^*(x+a)| \le \frac{a}{ \varpi-a} (\mathfrak{D}( \varpi-a)+\mathfrak{D}(\varpi-a+(a-b)/2)).
\end{equation}

From~\eqref{eq:part_1_f^*_bound} \&~\eqref{eq:part_2_f^*_bound}, we can now conclude that
\begin{equation}\label{eq:gen_bound_th_D_xi}
    \mathfrak{D}(\varpi)=\sup_{\varpi \le t \le \varpi+a} |f^*(t)| \le \frac{a}{\varpi-a} \mleft(\mathfrak{D}(\varpi-a)+\mathfrak{D}(\varpi-(a+b)/2) \mright), \quad \text{ for all }\varpi>a.
\end{equation}

The idea is now to iterate the bound in~\eqref{eq:gen_bound_th_D_xi}, so that we get a concrete bound on $\mathfrak{D}(\varpi)$ for all $\varpi>a$. Let $\varpi>a$, and define $k_\varpi \coloneqq \inf\{n \in \N: \varpi-an \le a\}$. From~\eqref{eq:gen_bound_th_D_xi} together with Lemma~\ref{lem:technical_D_lemma} below, we note that
\begin{equation}\label{eq:D_bound_general}
    \mathfrak{D}(\varpi) \le \frac{a^{k_\varpi}}{(\varpi-a)(\varpi-2a)\cdots (\varpi-k_\varpi a)} \sum_{r=0}^{k_\varpi} \binom{k_\varpi}{r} \mathfrak{D}\mleft(\varpi -ra - (k_\varpi-r)\mleft(\frac{a+b}{2}\mright)\mright).
\end{equation} Let now $r \in \{0 , \ldots , k_\varpi\}$, and note that $\varpi-k_\varpi a \le a$, implies that 
\begin{equation}\label{eq:ci_int_bound}
    \varpi -r a + a - (k_\varpi-r)\mleft(\frac{a+b}{2}\mright) 
    = (\varpi -k_\varpi a) + a + (k_\varpi-r)\mleft(\frac{a-b}{2}\mright) \le 2a + (k_\varpi-r) \mleft( \frac{a-b}{2}\mright).
\end{equation}
Recall from Remark~\ref{rem:det_bound_eta}, that $\eta(x) \le cx$ for $c=2/(a+b)$. Hence, from this,~\eqref{eq:ci_int_bound} and the definition of $\mathfrak{D}$, it holds that
\begin{equation}\label{eq:D_bound_k_1,k_2}
\begin{aligned}
     \mathfrak{D}\mleft(\varpi -ra - (k_\varpi-r)\mleft(\frac{a+b}{2}\mright)\mright) 
    \le \sup_{0 \le t \le 2a + (k_\varpi-r)\mleft(\frac{a-b}{2}\mright)} \mleft(c t+1+\ov\lambda \mleft(t+\frac{a+b}{2}\mright)\mright) \\
    \le  1 + 2a(c+\ov\lambda) + \ov\lambda \mleft(\frac{a+b}{2}\mright) + (k_\varpi-r)(c+\ov\lambda)\mleft(\frac{a-b}{2}\mright) = K_1 + K_2 (k_\varpi -r),
\end{aligned}
\end{equation} where $K_1 \coloneqq 1 + 2a(c+\ov\lambda) + \ov\lambda (a+b)/2>0$ and $K_2 \coloneqq (c+\ov\lambda)(a-b)/2\ge 0$. Note that $K_2=0$ if $a=b$. Using $\sum_{i=0}^n \binom{n}{i}=2^n$ and $\sum_{i=0}^n \binom{n}{i}(n-i)=n2^{n-1}$, together with~\eqref{eq:D_bound_general} and~\eqref{eq:D_bound_k_1,k_2}, we then get
\begin{equation}
     \mathfrak{D}(\varpi) 
     \le \frac{a^{k_\varpi}\sum_{r=0}^{k_\varpi} \binom{k_\varpi}{r} ( K_1 + K_2 (k_\varpi -r))}{(\varpi-a)(\varpi-2a)\cdots (\varpi-k_\varpi a)} 
     =\frac{(2a)^{k_\varpi}\mleft( K_1 +K_2 k_\varpi/2\mright)}{(\varpi-a)(\varpi-2a)\cdots (\varpi-k_\varpi a)}.
\end{equation}

We note now, that $k_\varpi = \ceil{\varpi/a}-1$, and let $x = \varpi/a>1$, yielding $k_\varpi=\ceil{x}-1$. Next, applying the property $\Gamma(x+1)=x\Gamma(x)$ for all $x>0$ of the Gamma-function, implies that
\begin{equation*}
    (\varpi-a)(\varpi-2a)\cdots (\varpi-k_\varpi a) = a^{k_\varpi} \prod_{j=1}^{k_\varpi} (x-j) = a^{k_\varpi} \frac{\Gamma(x)}{\Gamma(\rho_x)},
\end{equation*} where $\rho_x= x-k_\varpi =x-\ceil{x}+1 \in (0,1]$. Hence, it follows that
\begin{equation*}
    \mathfrak{D}(\varpi) 
     \le 
     \frac{2^{x-\rho_x}\mleft( K_1 +K_2 (x-\rho_x)/2\mright) \Gamma(\rho_x)}{\Gamma(x)}.
\end{equation*} Next, by~\cite[Eq.~(3.9)]{ArtinGamma1964}, and the standard rule $\Gamma(x+1)=x\Gamma(x)$, Stirling's inequality for the Gamma function follows, i.e.\ $\sqrt{2\pi x}(x/e)^x<\Gamma(x+1)< \sqrt{2\pi x}(x/e)^x e^{1/12}$ for all $x \ge 1$. Hence, for all $x\ge 2$ (i.e.\ $\varpi \ge 2a$, which we will see later is sufficient), it follows that
\begin{equation}\label{eq:D_bound_stirling}
\begin{aligned}
     \mathfrak{D}(\varpi) 
     &\le  
     2^{-\rho_x+1} \Gamma(\rho_x) \mleft(\frac{2e}{x-1}\mright)^{x-1} \frac{K_1+K_2 (x-\rho_x)/2}{\sqrt{2\pi (x-1)}} \\
     &\le 2 \Gamma(\rho_x)
     \mleft(\frac{2e}{x-1}\mright)^{x-1} \frac{K_1+K_2 x/2}{\sqrt{2\pi (x-1)}} 
     =
     2\Gamma(\rho_x) \mleft(\frac{2e}{\varpi/a-1}\mright)^{\varpi/a-1} \frac{K_1+K_2 \varpi/(2a)}{\sqrt{2\pi (\varpi/a-1)}}.
\end{aligned}
\end{equation} 

For any $t>2a(e+1)$, there exists some $n \in \N$ such that $na \le t \le na+a$. In this case, for $\varpi=an$, we note that $x=\varpi/a=n \in \N$, implying that $\rho_x\equiv 1$ and $\Gamma(\rho_x)=1$. Hence, by~\eqref{eq:D_bound_stirling}, it follows for $an \le t \le a(n+1)
$, that
\begin{align}
    |\eta(t) -\ov\lambda t +1-\ov\lambda (a+b)/2| 
    &\le 
    \mathfrak{D}(an) 
    \le 
    2\mleft(\frac{2e}{n-1}\mright)^{n-1} \frac{K_1+K_2 n/2}{\sqrt{2\pi (n-1)}} \nonumber\\
    & \le 
    \mleft(\frac{2e}{t/a-2}\mright)^{t/a-2} \frac{2K_1+K_2 t/a}{\sqrt{2\pi (t/a-2)}}\nonumber \\
    &
    \le \begin{dcases}
      \mleft(\frac{2e}{t/a-2}\mright)^{t/a-3/2} \frac{K_1}{\sqrt{\pi e}}, &\text{ if }a=b,\\
      \mleft(\frac{2e}{t/a-2}\mright)^{t/a-5/2} \frac{(6K_1+3K_2)\sqrt{e}}{\sqrt{\pi}}, &\text{ if }a>b.
    \end{dcases}\label{eq:t/abound_almost}
\end{align} In the case $a>b$, we used the inequalities that $t>3a$ (i.e.\ $t/a>1$ and $t/a-2>1$), and hence
\begin{equation*}
    \frac{K_1+K_2 t/(2a)}{\sqrt{2\pi (t/a-2)}} 
    \le 
    \frac{(K_1+K_2/2) t/a}{\sqrt{2\pi (t/a-2)}} \le \frac{(K_1+K_2/2) (t/a-2+2)}{\sqrt{2\pi (t/a-2)}} \le  \frac{3(K_1+K_2/2) (t/a-2)}{\sqrt{2\pi (t/a-2)}}.
\end{equation*}

Next, from~\eqref{eq:t/abound_almost}, it follows we can note, for $m \in \{3/2,5/2\}$, that
\begin{equation}\label{eq:bound_exp_rate}
    \mleft(\frac{2e}{t/a-2}\mright)^{t/a-m}
    = 
    \mleft(\frac{2e}{t/a}\mright)^{t/a-m} \mleft(\frac{t/a}{t/a-2}\mright)^{t/a-m} 
    \le
    e^2\mleft(\frac{2e}{t/a}\mright)^{t/a-m} .
\end{equation} To see that the final inequality holds, we note that it is equivalent to proving that $(1-2/y)^{m-y}<e^2$ for all $y>3$, which holds if $G(y)=(m-y)\log(1-2/y)<2$ for all $y>3$. Note that $G'(y)=2(m-y)/(y^2-2y)-\log(1-2/y) \ge 2(m-y)/(y^2-2y)+ 2/(y-1)$ for all $y >3$, which follows from the bound $-\log(1-z) \ge 2z/(2-z)$ for all $z \in (0,1)$. Hence, since 
\begin{equation*}
    G'(y)\ge  \frac{2(m-y)}{y^2-2y}+ \frac{2}{y-1}=2\frac{y(m-1)-m}{y(y-2)(y-1)}>0, \quad \text{ for } m \in \{3/2,5/2\} \text{ and all }y>3,
\end{equation*} it follows that $G$ is a strictly increasing function. We can therefore conclude $G(y) < \lim_{y\to \infty} G(y)=2$ for all $y>3$, finishing the argument.

Thus, by~\eqref{eq:t/abound_almost} \&~\eqref{eq:bound_exp_rate}, it follows for all $t>2a(e+1)$, that 
\begin{equation*}
    |\eta(t) -\ov\lambda t +1-\ov\lambda (a+b)/2| \le \begin{dcases}
      \mleft(\frac{2e}{t/a}\mright)^{t/a-3/2} \frac{e^{3/2} K_1}{\sqrt{\pi}}, &\text{ if }a=b,\\
      \mleft(\frac{2e}{t/a}\mright)^{t/a-5/2} \frac{3e^{5/2}(2K_1+K_2)}{\sqrt{\pi}}, &\text{ if }a>b.
    \end{dcases}
    \qedhere
\end{equation*}
\end{proof}
\begin{lemma}\label{lem:technical_D_lemma}
Let $a>0$ and $b\in [0,a]$, and set $c\coloneqq (a+b)/2\le a$. Suppose $\mathfrak{D}:(0,\infty)\to[0,\infty)$ satisfies, for all $\varpi>a$,
\begin{equation}\label{eq:one-step}
\mathfrak{D}(\varpi)\le \frac{a}{\varpi-a}\Big(\mathfrak{D}(\varpi-a)+\mathfrak{D}(\varpi-c)\Big).
\end{equation}
For $\varpi>a$ define $k_\varpi\coloneqq \inf\{n\in\mathbb{N}: \varpi-an\le a \}$. Then
\begin{equation}\label{eq:iterated}
\mathfrak{D}(\varpi)\le \frac{a^{k_\varpi}}{(\varpi-a)(\varpi-2a)\cdots(\varpi-k_\varpi a)} \sum_{r=0}^{k_\varpi}\binom{k_\varpi}{r} \mathfrak{D}\big(\varpi-ra-(k_\varpi-r)c\big).
\end{equation}
\end{lemma}

\begin{proof}
We first prove, by induction on $k\in\mathbb{N}$, that for every $\varpi>ka$,
\begin{align}
\mathfrak{D}(\varpi)
&\le 
\frac{a^{k}}{\prod_{j=1}^{k}(\varpi-ja)} \sum_{r=0}^{k}\binom{k}{r}  \mathfrak{D}\big(\varpi-ra-(k-r)c\big)= \mathfrak{P}_k(\varpi) \mathfrak{S}_k(\varpi), \quad \text{where}\label{eq:Ik}\\
    \mathfrak{P}_k(\varpi) 
    &\coloneqq 
    \frac{a^{k}}{\prod_{j=1}^{k}(\varpi-ja)}, \quad \text{ and }\quad \mathfrak{S}_k(\varpi) \coloneqq \sum_{r=0}^{k}\binom{k}{r} \mathfrak{D}\big(\varpi-ra-(k-r)c\big).\nonumber
\end{align}

The initial step $k=1$, follows directly from~\eqref{eq:one-step}, since we by definition of $\mathfrak{P}_1$ and $\mathfrak{S}_1$, have that $\mathfrak{P}_1(\varpi)=a/(\varpi-a)$ and $\mathfrak{S}_1(\varpi)=\mathfrak{D}(\varpi-a)+\mathfrak{D}(\varpi-c)$. Assume now that~\eqref{eq:Ik} holds for some $k\ge1$ and all $\zeta>ka$. Fix $\varpi>(k+1)a$. Using~\eqref{eq:one-step} at $\varpi$ and then the induction hypothesis at $\varpi-a$ and $\varpi-c$ (note that $\varpi-a>ka$ and $\varpi-c>ka$), we get
\[
\mathfrak{D}(\varpi)
\le 
\frac{a}{\varpi-a}\Big(\mathfrak{D}(\varpi-a)+\mathfrak{D}(\varpi-c)\Big)\le \frac{a}{\varpi-a}\Big(\mathfrak{P}_k(\varpi-a) \mathfrak{S}_k(\varpi-a)+\mathfrak{P}_k(\varpi-c) \mathfrak{S}_k(\varpi-c)\Big).
\]
Because $c\le a$, for each $1\le j\le k$ we have $\varpi-c-ja\ge\varpi-a-ja=\varpi-(j+1)a$, hence
\[
\mathfrak{P}_k(\varpi-c)=\frac{a^k}{\prod_{j=1}^k(\varpi-c-ja)}
\le
\frac{a^k}{\prod_{j=2}^{k+1}(\varpi-ja)} 
= 
\mathfrak{P}_k(\varpi-a).
\]
Therefore, $\mathfrak{D}(\varpi)\le a(\varpi-a)^{-1} \mathfrak{P}_k(\varpi-a) (\mathfrak{S}_k(\varpi-a)+\mathfrak{S}_k(\varpi-c))$. A direct reindexing in the sum, and an application of Pascal's identity, yield
\[
\begin{aligned}
\mathfrak{S}_k(\varpi-a)+\mathfrak{S}_k(\varpi-c)
&=
\sum_{r=0}^k\binom{k}{r}
\Big(\mathfrak{D}(\varpi-(r+1)a-(k-r)c)+\mathfrak{D}(\varpi-ra-(k-r+1)c)\Big)\\
&=
\sum_{s=0}^{k+1}\mleft(\binom{k}{s-1}+\binom{k}{s}\mright)
\mathfrak{D}\big(\varpi-sa-(k+1-s)c\big)\\
&=
\sum_{s=0}^{k+1}\binom{k+1}{s} \mathfrak{D}\big(\varpi-sa-(k+1-s)c\big)
=
\mathfrak{S}_{k+1}(\varpi).
\end{aligned}
\]
Finally, from the definition of $\mathfrak{P}_k$, we see that
\[
\frac{a}{\varpi-a} \mathfrak{P}_k(\varpi-a)
=
\frac{a}{\varpi-a}\cdot\frac{a^k}{\prod_{j=2}^{k+1}(\varpi-ja)}
=
\frac{a^{k+1}}{\prod_{j=1}^{k+1}(\varpi-ja)}
=
\mathfrak{P}_{k+1}(\varpi).
\]
Hence, $\mathfrak{D}(\varpi)\le \mathfrak{P}_{k+1}(\varpi) \mathfrak{S}_{k+1}(\varpi)$, completing the induction step.

Now take $k=k_\varpi$. By definition of $k_\varpi$ we have $\varpi>k_\varpi a$, so~\eqref{eq:Ik} applies and yields exactly~\eqref{eq:iterated}.
\end{proof}

\begin{proof}[Proof of Theorem~\ref{thm:main_theorem}]
The first conclusion of the theorem follows directly from Lemma~\ref{lem:parking_constant_a,b}. 
    
For the second conclusion of the theorem, we note that by the proof of Lemma~\ref{lem:ROC_first_order}, it follows that \begin{equation*}
    \lim_{x \to \infty }\inf_{x \le c \le x+a}\frac{\eta(c)+1}{c+(a+b)/2} =I_\infty = \ov\lambda(a,b)=S_\infty=\lim_{x \to \infty }\sup_{x \le c \le x+a}\frac{\eta(c)+1}{c+(a+b)/2},
    \end{equation*} and since $\lambda(a,b)=\lim_{x \to \infty} \eta(x)/x$ exists and is finite by Lemma~\ref{lem:parking_constant_a,b}, it follows by uniqueness of limits, that $\ov\lambda(a,b)=\lambda(a,b)$ for all $0 \le b \le a <\infty$ with $a>0$. Hence, the second conclusion on the rate of convergence follows directly from Lemma~\ref{lem:ROC_first_order}.
\end{proof}

\subsection{Worst trapezium}
\begin{proof}[Proof of Lemma~\ref{lem:worst_trap_exists_uniq}]
For $x\in(0,1]$, set $\alpha(x)\coloneqq (1+x)/2$ and $\Phi(x,t) \coloneqq \int_0^t (-2+e^{-\alpha(x)s}+e^{-s})s^{-1}\D s$ for all $t >0$. Furthermore, write $W_x(t)\coloneqq e^{\Phi(x,t)}$ and $A_x(t)\coloneqq 2+e^{-xt}-e^{-t}$. By Corollary~\ref{cor:parking_constant}, it follows that $ \xi(1,x) = 4^{-1}(1+x) \int_0^\infty A_x(t)W_x(t)\D t$. On every compact subinterval of $(0,1]$, differentiation under the integral sign is justified by dominated convergence. Consequently,
\begin{equation}\label{eq:defn_Xi_correct}
    \Xi(x)
    \coloneqq
    \frac{\D}{\D x}\xi(1,x)
    =
    \frac{1}{4} \int_0^\infty J_x(t)W_x(t)\D t, \quad \text{where}\quad 
    J_x(t)
    \coloneqq
    A_x(t)e^{-\alpha(x)t} - 2\alpha(x)t e^{-xt}.
\end{equation} We will prove that $\Xi$ has exactly one zero in $(0,1)$. The proof is
divided into four steps.

\textbf{Step 1:} Define
\begin{equation}\label{eq:defn_H_worst}
    H_x(t)
    \coloneqq
    A_x(t)(1-e^{-t}) - t(e^{-xt}+e^{-t}).
\end{equation}
We claim that
\begin{equation}\label{eq:Xi_H_representation}
    \Xi(x)
    =
    \frac14
    \int_0^\infty
    H_x(t)W_x(t)\D t.
\end{equation}
Indeed, $\partial_t\Phi(x,t)
    = (-2+e^{-\alpha(x)t}+e^{-t})t^{-1}$ and $\partial_t A_x(t)=A_x'(t) = -xe^{-xt}+e^{-t}$. Applying these equalities, applying the product rule thrice, and recalling the definition of $W_x(t)$, yields
\begin{equation}\label{eq:derivative_W_A_J_H}
\begin{aligned}
    \frac{1}{W_x(t)} \frac{\D}{\D t} \mleft( tA_x(t)W_x(t) \mright)
&=
A_x(t) + tA_x'(t) + tA_x(t)\partial_t\Phi(x,t)\\
&=
A_x(t) \big( e^{-\alpha(x)t}+e^{-t}-1 \big) + t(e^{-t}-xe^{-xt}) 
=
J_x(t)-H_x(t).
\end{aligned}
\end{equation}
Moreover, $ \lim_{t\downarrow0} tA_x(t)W_x(t)=0$. Next, by~\eqref{eq:Walpha_equal}, it follows that $\Phi(x,t) \le -2\log(t)- \log(\alpha(x))-2\gamma$, and hence $W_x(t) \le t^{-2}e^{-2\gamma}\alpha(x)^{-1} \le 2e^{-2\gamma}t^{-2}$ for all $t>0$ and for all $x \in (0,1]$, since $\alpha(x)\ge 1/2$ for all $x \in (0,1]$. Moreover, since $A_x(t)\le3$, it holds that $tA_x(t)W_x(t) \le 6 e^{-2\gamma} t^{-1}$ for all $t>0$ and $x \in (0,1]$, so that $ \lim_{t\to\infty} tA_x(t)W_x(t)=0$. Integrating the derivative identity in~\eqref{eq:derivative_W_A_J_H} over $(0,\infty)$ thereby yields $\int_0^\infty J_x(t)W_x(t)\D t = \int_0^\infty H_x(t)W_x(t)\D t$, which concludes~\eqref{eq:Xi_H_representation}.

\textbf{Step 2:} We first show that
\begin{equation}\label{eq:Xi_minus_infty_new}
    \lim_{x\downarrow0}\Xi(x)=-\infty.
\end{equation}
For $c>0$ and $t>0$, we have by~\eqref{eq:Walpha_equal}, that $ \int_0^t (e^{-cs}-1)s^{-1}\D s = -\gamma - \log(ct) - \Gamma(0,ct)$. Hence, $W_x(t) = e^{-2\gamma} \alpha(x)^{-1}t^{-2} e^{ -\Gamma(0,\alpha(x)t) - \Gamma(0,t)}$. Since, $1/2\le\alpha(x)\le1$ and $\Gamma(0,\cdot)$ is positive and decreasing, there exist constants $c_0,C_0>0$, independent of $x\in(0,1]$, such that
\begin{equation}\label{eq:W_two_sided}
    \frac{c_0}{t^2} \le W_x(t) \le \frac{C_0}{t^2}, \qquad \text{ for all } t\ge1.
\end{equation}

Fix $T\coloneqq 6e$. For $0<x\le T^{-1}$ and $T\le t\le x^{-1}$, we have $ e^{-xt}\ge e^{-1}$. Since $A_x(t)\le3$, it follows from~\eqref{eq:defn_H_worst} that $H_x(t) \le 3-te^{-xt} \le 3-t/e \le -t/(2e)$. Consequently, by~\eqref{eq:W_two_sided}, we have
\begin{equation}\label{eq:middle_negative_log}
    \int_T^{1/x} H_x(t)W_x(t)\D t \le -\frac{c_0}{2e} \int_T^{1/x}\frac{\D t}{t} = -\frac{c_0}{2e} \log\mleft(\frac{1}{Tx}\mright).
\end{equation}
On the other hand, $\sup_{0<x\le1} \int_0^T H_x(t)W_x(t)\D t \le 3T < \infty$, since the $W_x(t) \le 1 $ and $H_x(t) \le 3$ for all $x \in (0,1]$ and $t \in [0,T]$. Furthermore, $H_x(t)\le3$ for $t\ge x^{-1}$, and hence, by~\eqref{eq:W_two_sided}, we have $\int_{1/x}^\infty H_x(t)W_x(t)\D t \le 3C_0 \int_{1/x}^\infty t^{-2} \D t = 3C_0x$. Combining these estimates with~\eqref{eq:Xi_H_representation} proves~\eqref{eq:Xi_minus_infty_new}.

We next show that
\begin{equation}\label{eq:Xi_one_positive_new}
    \Xi(1)>0.
\end{equation}
When $x=1$, we have $\alpha(1)=1$, $J_1(t) = 2e^{-t}(1-t)$, and $t\tfrac{\D}{\D t}W_1(t) = -2(1-e^{-t})W_1(t)$. Hence, it follows that
\[
\begin{aligned}
\frac{\D}{\D t} \mleft( te^{-t}W_1(t) \mright)
&= 
e^{-t}(1-t)W_1(t) - 2e^{-t}(1-e^{-t})W_1(t).
\end{aligned}
\]
We have $\int_0^\infty \frac{\D}{\D t} \mleft( te^{-t}W_1(t) \mright) \D t=0$, since the boundary terms vanish at $0$ and $\infty$, i.e.\ $\lim_{t \da 0} te^{-t}W_1(t)=0$ and $\lim_{t \to \infty} te^{-t}W_1(t)=0$, and therefore $ \int_0^\infty e^{-t}(1-t)W_1(t)\D t = 2 \int_0^\infty e^{-t}(1-e^{-t})W_1(t)\D t$. Hence,
\begin{equation}\label{eq:Xi_one_positive_explicit}
    \Xi(1) 
    =
    \frac12 \int_0^\infty e^{-t}(1-t)W_1(t)\D t 
    =
    \int_0^\infty e^{-t}(1-e^{-t})W_1(t)\D t > 0.
\end{equation} By~\eqref{eq:Xi_minus_infty_new},~\eqref{eq:Xi_one_positive_new}, and continuity of $\Xi$ on $(0,1]$, the function $\Xi$ has at least one zero in $(0,1)$.

\textbf{Step 3:} We now prove the crucial fact that
\begin{equation}\label{eq:zero_upward}
    \Xi(x)=0 \quad\Longrightarrow\quad \Xi'(x)>0.
\end{equation} For this purpose, let $\Pi$ be a Poisson point process on $(0,1]$ with intensity measure $u^{-1} \D u$, and define $D\coloneqq\sum_{u\in\Pi}u$, where $\E[D]=1<\infty$, implying that the random variable $D$ is finite almost surely. The Laplace transform of $D$ is
\begin{equation}\label{eq:Dickman_LT}
    \varphi(t)
    \coloneqq
    \E[e^{-tD}] = \exp\mleft( \int_0^1 \frac{e^{-tu}-1}{u}\D u \mright)= \exp\mleft( \int_0^t \frac{e^{-v}-1}{v}\D v \mright).
\end{equation}
Let $D_1,D_2$ be independent copies of $D$, and put $S_x \coloneqq D_1+\alpha(x)D_2$. It follows from~\eqref{eq:Dickman_LT} and the independence between $D_1$ and $D_2$, that
\begin{equation}\label{eq:Wx_LT}
    \E[e^{-tS_x}] = \E[e^{-tD_1}]\E[e^{-t\alpha(x)D_2}] = W_x(t).
\end{equation}

Expanding~\eqref{eq:defn_H_worst}, we obtain $H_x(t)= 2+e^{-xt}-3e^{-t}-e^{-(x+1)t}+e^{-2t} - te^{-xt}-te^{-t}$. Using~\eqref{eq:Wx_LT} and Fubini's theorem, therefore gives
\begin{equation}\label{eq:Xi_h_expectation}
    4\Xi(x) = \E[h_x(S_x)],
\end{equation}
where, for $s>0$,
\begin{equation}\label{eq:defn_hx_long}
\begin{aligned}
    h_x(s) &\coloneqq \frac{2}{s} + \frac{1}{s+x} - \frac{3}{s+1} - \frac{1}{s+x+1} + \frac{1}{s+2} - \frac{1}{(s+x)^2} - \frac{1}{(s+1)^2}\\
&= \frac{3s+4}{s(s+1)^2(s+2)} - \frac{1}{(s+x)^2(s+x+1)}.
\end{aligned}
\end{equation}

We shall also use the following elementary identity. For every suitable measurable function $f$,
\begin{equation}\label{eq:macke_equality}
    \E[D f(D)] = \int_0^1 \E[f(D+u)]\D u.
\end{equation}
Indeed, by the Mecke identity~\cite[Sec.~3.1.1, Prop.~1]{PeccatiReitzner} for the Poisson point process $\Pi$, it follows that
\[
\E[D f(D)]= \E\mleft[ \sum_{u\in\Pi} u f(D) \mright] = \int_0^1 u\E[f(D+u)] \frac{\D u}{u}=\int_0^1 \E[f(D+u)]\D u.
\]

Recalling that $4\Xi(x)=\E[h_x(S_x)]$ for $S_x=D_1+\alpha(x)D_2$, and that $\alpha'(x)=1/2$, the chain rule yields $\tfrac{\D}{\D x}h_x(S_x) = \partial_x h_x(S_x) + (1/2) D_2\partial_s h_x(S_x)$. It remains to justify passing the derivative through the expectation. Fix $x_0\in(0,1]$, and choose $\delta\in(0,x_0)$. We work with $x\in K\coloneqq[\delta,1]$. This localisation is sufficient, since differentiability at $x_0$ only requires domination in a neighbourhood of $x_0$. Moreover, it is useful here because the derivatives of $h_x$ contain negative powers of $s+x$, which are not uniformly bounded as $x\downarrow0$ and $s\downarrow0$. From the explicit rational expression for $h_x$, there exists a constant $C_\delta<\infty$ such that, uniformly in $x\in K$ and $s>0$, we have $|\partial_x h_x(s)| \le C_\delta$ and $|\partial_s h_x(s)| \le C_\delta (1+s^{-2} )$. Since $\alpha$ is increasing, we have $ S_x  = D_1+\alpha(x)D_2 \ge  D_1+\alpha(\delta)D_2 \eqqcolon S_\delta $ for every $x\in K$. Furthermore, $S_x\ge \alpha(x)D_2$, and therefore, $ D_2/S_x^2 \le (\alpha(x)S_x)^{-1} \le (\alpha(\delta)S_\delta)^{-1}$. It follows that, uniformly in $x\in K$,
\[
\begin{aligned}
    \mleft| \frac{\D}{\D x}h_x(S_x) \mright|
    &\le 
    |\partial_x h_x(S_x)| + \frac{1}{2}D_2|\partial_s h_x(S_x)|\\
    &\le 
    C_\delta + \frac{C_\delta}{2} \mleft( D_2+\frac{D_2}{S_x^2} \mright)
    \le 
    C_\delta + \frac{C_\delta}{2} \mleft( D_2+ \frac{1}{\alpha(\delta)S_\delta} \mright).
\end{aligned}
\]
The random variable on the right-hand side is integrable. Indeed, $\E[D_2]=1$, while, since $S_\delta>0$ almost surely, Tonelli's theorem and~\eqref{eq:Wx_LT} give $\E[S_\delta^{-1}] = \int_0^\infty \E[e^{-tS_\delta}] \D t = \int_0^\infty W_\delta(t)\D t$. Now $W_\delta(t)\le 1$ for $0<t\le 1$, whereas for $t\ge 1$, $W_\delta(t) \le e^{-2\gamma}\alpha(\delta)^{-1}t^{-2}$. Consequently,
\[
\E[S_\delta^{-1}] 
\le 
1+ \frac{e^{-2\gamma}}{\alpha(\delta)} \int_1^\infty t^{-2}\D t
=
1+\frac{e^{-2\gamma}}{\alpha(\delta)}<\infty.
\]
Thus, $\frac{\D}{\D x}h_x(S_x)$ is dominated, uniformly for $x\in K$, by an integrable random variable which does not depend on $x$. Dominated convergence therefore allows differentiation under the expectation. Since $x_0\in(0,1]$ was arbitrary, we obtain
\begin{equation}\label{eq:Xi_prime_expectation_1}
    4\Xi'(x) = \E\mleft[ \partial_xh_x(S_x) \mright] + \frac12 \E\mleft[ D_2\partial_sh_x(S_x) \mright], \qquad \text{for all } x\in(0,1].
\end{equation}

To apply~\eqref{eq:macke_equality}, we first condition on $D_1$. For fixed $x\in(0,1]$ and $d_1\ge 0$, define $f_{x,d_1}(d) \coloneqq \partial_s h_x (d_1+\alpha(x)d )$ for all $d \ge 0$. Since $D_2$ is independent of $D_1$ and has the same distribution as $D$, we may apply~\eqref{eq:macke_equality} to the random variable $D_2$ and the function $f_{x,d_1}$. Thus, for almost every realisation of $D_1$, we have that
\[
\begin{aligned}
\E\mleft[ D_2\partial_s h_x\bigl(D_1+\alpha(x)D_2\bigr) \,\vert\,D_1 \mright]
&=
\E\mleft[ D_2 f_{x,D_1}(D_2) \,\vert\,D_1 \mright] = \int_0^1 \E\mleft[ f_{x,D_1}(D_2+u) \,\vert\,D_1 \mright]\D u\\
&=
\int_0^1 \E\mleft[ \partial_s h_x\bigl(D_1+\alpha(x)(D_2+u)\bigr) \,\vert|\,D_1 \mright]\D u.
\end{aligned}
\]
Recalling that $S_x=D_1+\alpha(x)D_2$, and taking expectations with respect to $D_1$, we obtain
\[
\E\mleft[ D_2\partial_s h_x(S_x) \mright]
=
\int_0^1 \E\mleft[ \partial_s h_x\bigl(S_x+\alpha(x)u\bigr) \mright]\D u.
\]
Finally, for every fixed value of $S_x$, the change of variables $v=S_x+\alpha(x)u$ gives
\[
\begin{aligned}
\int_0^1 \partial_s h_x\bigl(S_x+\alpha(x)u\bigr)\D u
&=
\frac{1}{\alpha(x)} \int_{S_x}^{S_x+\alpha(x)} \partial_s h_x(v)\D v = \frac{1}{\alpha(x)} \mleft( h_x\bigl(S_x+\alpha(x)\bigr)-h_x(S_x) \mright).
\end{aligned}
\]
Consequently,
\[
\E\mleft[ D_2\partial_s h_x(S_x) \mright]
=
\frac{1}{\alpha(x)} \E\mleft[ h_x\bigl(S_x+\alpha(x)\bigr)-h_x(S_x) \mright].
\] Since $2\alpha(x)=1+x$, equation~\eqref{eq:Xi_prime_expectation_1} becomes
\begin{equation}\label{eq:Xi_prime_expectation_2}
\begin{aligned}
4\Xi'(x) &= \E\mleft[ \partial_xh_x(S_x) + \frac{ h_x(S_x+\alpha(x))-h_x(S_x)}{1+x} \mright].
\end{aligned}
\end{equation}

Suppose now that $\Xi(x)=0$. By~\eqref{eq:Xi_h_expectation}, $\E[h_x(S_x)]=0$, and therefore~\eqref{eq:Xi_prime_expectation_2} yields
\begin{equation}\label{eq:Xi_prime_at_zero}
    4\Xi'(x) = \E[p_x(S_x)], \quad \text{where}\quad p_x(s) \coloneqq \partial_xh_x(s) + \frac{1}{1+x} h_x(s+\alpha(x)).
\end{equation} We claim that
\begin{equation}\label{eq:px_positive}
    p_x(s)>0, \qquad \text{for all }  x\in(0,1], \, s>0.
\end{equation} To prove~\eqref{eq:px_positive}, define $ k(t) \coloneqq t-1+e^{-t}$ and $a(t) \coloneqq 2-3e^{-t}+e^{-2t}-te^{-t}$. A direct Laplace-transform calculation from~\eqref{eq:defn_hx_long} gives
\begin{equation}\label{eq:hx_Laplace}
    h_x(s) = \int_0^\infty e^{-st} \mleft( a(t)-e^{-xt}k(t) \mright)\D t,
\end{equation}
since $ \int_0^\infty e^{-zt}k(t)\D t = z^{-2}(z+1)^{-1}$ for all $z >0$. It follows from~\eqref{eq:hx_Laplace} that
\begin{align}
    p_x(s) &= \frac{1}{1+x} \int_0^\infty e^{-st}R_x(t)\D t, \quad \text{ where } \label{eq:px_Rx}\\
    R_x(t) &\coloneqq e^{-\alpha(x)t}a(t) + k(t) \mleft( (1+x)t e^{-xt} - e^{-(\alpha(x)+x)t} \mright).\label{eq:defn_Rx}
\end{align}
Set $B\coloneqq e^{-xt}$, $C\coloneqq e^{-t}$, and $E\coloneqq e^{-\alpha(x)t}$. Since $a(t) = 2(1-C)^2-Ck(t)$, we can rewrite~\eqref{eq:defn_Rx} as
\begin{equation}\label{eq:Rx_Qx}
    R_x(t) = 2E(1-C)^2 + k(t)Q_x(t), \quad \text{where}\quad Q_x(t) \coloneqq (1+x)tB-E(B+C).
\end{equation} Since $k(t)=t-(1-C)>0$ for every $t>0$, it is immediate from~\eqref{eq:Rx_Qx} that $R_x(t)>0$ whenever $Q_x(t)\ge0$. Assume therefore that $Q_x(t)<0$. We first note that this implies $t<1$. Indeed, $Q_x(t)<0$ gives $(1+x)t < e^{-\alpha(x)t} ( 1+e^{-(1-x)t} )$. For $t\ge1$, we have that $e^{-\alpha(x)t} ( 1+e^{-(1-x)t} ) \le e^{-(1+x)/2} ( 1+e^{-(1-x)})$. The function $x\mapsto e^{-(1+x)/2} ( 1+e^{-(1-x)} )$ is non-increasing on $[0,1]$, and its value at $x=0$ equals $e^{-1/2}(1+e^{-1})<1$. Thus, $e^{-\alpha(x)t} ( 1+e^{-(1-x)t} )<1$ for $t\ge1$, whereas $(1+x)t\ge1$, which is a contradiction. Hence $t<1$.

For $0<t<1$, we have
\begin{equation}\label{eq:k_less_square}
    k(t) < (1-e^{-t})^2.
\end{equation}
Indeed, put $g(t) \coloneqq (1-e^{-t})^2-(t-1+e^{-t})$. Then, $ g'(t) = -(1-e^{-t})(1-2e^{-t})$, implying that $g$ is increasing on $(0,\log2)$ and decreasing on $(\log2,1)$. Moreover, $g(0)=0$ and $g(1) = 1-3/e+ e^{-2} >0$. Hence~\eqref{eq:k_less_square} follows.

Since $Q_x(t)<0$, we have $Q_x(t) > -E(B+C) \ge -2E$. Together with~\eqref{eq:Rx_Qx} and~\eqref{eq:k_less_square}, this gives
\[
R_x(t) > 2E(1-C)^2-2Ek(t) = 2E \mleft( (1-e^{-t})^2-k(t) \mright) > 0.
\]
Thus $R_x(t)>0$ for every $x\in(0,1]$ and every $t>0$. By~\eqref{eq:px_Rx}, this proves~\eqref{eq:px_positive}. Returning to~\eqref{eq:Xi_prime_at_zero}, we conclude that $\Xi(x)=0$ implies that $\Xi'(x)>0$, which concludes the proof of~\eqref{eq:zero_upward}.

\textbf{Step 4:} We already know that $\Xi$ has at least one zero in $(0,1)$. Suppose, for contradiction, that it has two distinct zeros $0<x_1<x_2<1$. By~\eqref{eq:zero_upward}, we have $\Xi'(x_1)>0$ and $ \Xi'(x_2)>0$. Hence, $\Xi$ is strictly positive immediately to the right of $x_1$ and strictly negative immediately to the left of $x_2$. We may therefore choose $x_1<a<b<x_2$ such that $\Xi(a)>0$ and $\Xi(b)<0$. Let now $ c \coloneqq \inf\{ y\in[a,b]: \Xi(y)=0 \}$. Then, $\Xi(c)=0$ and $\Xi(y)>0$ for every $y<c$ sufficiently close to $c$. On the other hand,~\eqref{eq:zero_upward} gives $\Xi'(c)>0$, which implies $\Xi(y)<0$ for every $y<c$ sufficiently close to $c$. This is a contradiction. Hence, $\Xi$ has exactly one zero, which we denote by $x^\ast$.

Since $ \lim_{x\downarrow0}\Xi(x)=-\infty$ by~\eqref{eq:Xi_minus_infty_new} and $\Xi$ has no zero before $x^\ast$, we have $\Xi(x)<0$ for all $0<x<x^\ast$. Similarly, since $\Xi(1)>0$ and there are no further zeros, we know $\Xi(x)>0$ for all $x^\ast<x\le1$. Consequently, $ x\mapsto\xi(1,x)$ is strictly decreasing on $(0,x^\ast)$ and strictly increasing on $(x^\ast,1]$. Since $\xi(1,\cdot)$ is continuous on $[0,1]$, it follows that $x^\ast\in(0,1)$ is its unique global minimiser.
\end{proof}

\section{Second-order asymptotics}\label{sec:variance_asymptotics} 
In this section, we study the second-order fluctuations of the number of deposited trapeziums. More precisely, we prove an analogue of the variance estimate of Dvoretzky and Robbins~\cite[Eq.~(1.4)]{Dvoretzky_Robbins}, using some of the same ideas in this more complicated setting. Our aim is to show that the variances of both $P_x$ and $T_x$ grow linearly in $x$, and that the error in the corresponding linear approximation decays at a super-exponential rate. Throughout this section, we use the notation
\[
    \varsigma\coloneqq \frac{a+b}{2}, \qquad \text{ and }\qquad  \delta\coloneqq \frac{a-b}{2}.
\]
Thus, $a=\varsigma +\delta$ and $0\le \delta\le \varsigma $. One useful feature of this notation is that, after the first trapezium has been deposited, the sizes of the two remaining strips always add up to
$x-\varsigma $.

We define the following random variables for notational convenience; $X_{\calP}(x)\coloneqq P_x$, $X_{\calT}(x)\coloneqq T_x$, $m_{\calP}(x)\coloneqq \E[P_x]=\mu(x)$ and $m_{\calT}(x)\coloneqq \E[T_x]=\nu(x)$. Furthermore, let
\begin{equation}\label{eq:defn_variance_affine_L}
    L(x)\coloneqq \lambda(a,b)x+\lambda(a,b)\varsigma -1, \quad \text{for all } x>0.
\end{equation}
The reason for introducing precisely this affine function is that it satisfies the exact branching identity
\begin{equation}\label{eq:L_exact_additivity}
    L(x)=1+L(u)+L(x-\varsigma -u),
\end{equation}
whenever $u$ and $x-\varsigma -u$ are the sizes of the two strips obtained after the first deposition. Indeed,
\[
\begin{aligned}
1+L(u)+L(x-\varsigma -u)
&=
1+\lambda(a,b)u+\lambda(a,b)\varsigma -1 +\lambda(a,b)(x-\varsigma -u)+\lambda(a,b)\varsigma -1\\
&=
\lambda(a,b)x+\lambda(a,b)\varsigma -1= L(x).
\end{aligned}
\]

Throughout the section, we will write
\begin{equation}\label{eq:defn_typewise_mean_errors}
    e_s(x)\coloneqq m_s(x)-L(x), \qquad \text{for all } s\in\{\calP,\calT\} \text{ and }x >0.
\end{equation}
We first show that the first-order asymptotic expansion from Theorem~\ref{thm:main_theorem} also holds separately for $\mu$ and $\nu$, and not just for the average $\eta$.

\begin{lemma}\label{lem:typewise_mean_error}
There exists a finite constant $C>0$ such that, for all $x>(3a-b)/2$,
\begin{equation}\label{eq:defn_R1_variance_section}
\max_{s\in\{\calP,\calT\}}|e_s(x)| \le C\mathcal R_1(x), \quad \text{where}\quad \mathcal R_1(x) \coloneqq
\begin{dcases}
    \mleft(\dfrac{2e}{x/a}\mright)^{x/a-3/2},
    &\text{if }a=b,\\
    \mleft(\dfrac{2e}{x/a}\mright)^{x/a-5/2},
    &\text{if }a>b.
\end{dcases}
\end{equation}
\end{lemma}

\begin{proof}
Recall that $\eta(x)=(\mu(x)+\nu(x))/2$, and hence, by Theorem~\ref{thm:main_theorem}, we have for all $x>2a(e+1)$ that
\begin{equation}\label{eq:eta_L_R1}
    |\eta(x)-L(x)|\le  K \mathcal R_1(x),
\end{equation} for some explicit constant (depending on whether $a=b$ or $a>b$) given in Theorem~\ref{thm:main_theorem}. It therefore remains to control the difference between $\mu$ and $\nu$. Set $h(x)\coloneqq \mu(x)-\nu(x)$. If $a=b$, then $\delta=0$, and the strips $\calP_x$ and $\calT_x$ coincide. Consequently, $P_x\eqd T_x$ for every $x>0$, and hence $h(x)=0$ for all $x>0$. The result then follows from~\eqref{eq:eta_L_R1} immediately.

Assume therefore that $a>b$, so that $\delta>0$. From the expectation recursions~\eqref{eq:resurc_mu} and~\eqref{eq:resurc_nu} obtained in the proof of Lemma~\ref{lem:eta_func}, we have for all $x> (3a-b)/2$, that
\begin{equation}\label{eq:h_contraction}
h(x) = \frac{1}{x-a}
\mleft( \int_0^{x-a}\mu(t)\D t - \int_0^{x-\varsigma }\mu(t)\D t + \int_\delta^{x-\varsigma }\nu(t)\D t - \int_\delta^{x-a}\nu(t)\D t \mright) = -\frac{1}{x-a} \int_{x-a}^{x-\varsigma }h(t)\D t.
\end{equation}
Since $x-\varsigma =x-a+\delta$, we may equivalently write
\begin{equation}\label{eq:h_contraction_shifted}
h(x) = -\frac{1}{x-a} \int_0^\delta h(x-a+u)\D u, \quad \text{ for all }x>x_0\coloneqq \frac{3a-b}{2}.
\end{equation} By Remark~\ref{rem:det_bound_eta}, the functions $\mu$ and $\nu$ have at most linear growth, and hence there exists a finite constant $C_0$ such that $|h(x)|\le C_0(1+x)$, for all $x \ge 0$. For $x>x_0+a$, define $k_x \coloneqq \lfloor(x-x_0)/a\rfloor$, then $x-k_xa\in[x_0,x_0+a)$. Iterating~\eqref{eq:h_contraction_shifted} $k_x$ times, and using the triangle inequality, gives
\begin{equation}
    \begin{aligned}\label{eq_h_bound_iterate}
|h(x)| &\le \int_0^\delta \cdots \int_0^\delta  \frac{ \mleft| h\mleft( x-k_xa+u_1+\cdots+u_{k_x} \mright) \mright| } { \prod_{j=1}^{k_x} \mleft( x-ja+u_1+\cdots+u_{j-1} \mright) } \D u_1\cdots\D u_{k_x}\\
&\le
\frac{\delta^{k_x}} {\prod_{j=1}^{k_x}(x-ja)} \sup_{0\le u_1,\ldots,u_{k_x}\le\delta} \mleft| h\mleft( x-k_xa+u_1+\cdots+u_{k_x} \mright) \mright|,
\end{aligned}
\end{equation} where we also used that $x-ja+u_1+\cdots+u_{j-1} \ge x-ja>0$ for all $j=1,\ldots,k_x$ and $u_1,\ldots u_j$. The argument of $h$ in the upper bound in~\eqref{eq_h_bound_iterate}, satisfies
\[
   x-k_xa+u_1+\cdots+u_{k_x} \le x-k_xa+k_x\delta \le x_0+a+\frac{\delta}{a}x,
\]
and therefore, by the linear growth of $h$, we have altogether, that
\[
|h(x)| \le C_1(1+x) \frac{\delta^{k_x}}{\prod_{j=1}^{k_x}(x-ja)}.
\]
Put $z\coloneqq x/a$. Using the Gamma-function identity, $\prod_{j=1}^{k_x}(x-ja) = a^{k_x} \Gamma(z)/\Gamma(z-k_x)$, we obtain
\[
|h(x)| \le C_1(1+x) \mleft(\frac{\delta}{a}\mright)^{k_x} \frac{\Gamma(z-k_x)}{\Gamma(z)}.
\]
Since $z-k_x$ remains in the fixed compact set $[x_0/a, x_0/a +1]$ for all $x>x_0$, the quantity $\Gamma(z-k_x)$ is uniformly bounded. Moreover, $k_x=z+\Oh(1)$. Stirling's formula therefore gives
\begin{equation}\label{eq:h_gamma_bound}
|h(x)| \le C_2(1+x)z^{1/2} \mleft(\frac{e\delta}{az}\mright)^z.
\end{equation}
Since $\delta/a \le 1/2$, the bound in~\eqref{eq:h_gamma_bound} decays strictly faster than $(2e/z)^{z-5/2}$, and hence, there exists some constant $C$, such that $ |h(x)|\le C \mathcal R_1(x)$ for all $x> (3a-b)/2$.

Finally, applying~\eqref{eq:eta_L_R1} together with the decompositions below concludes the proof: 
\[
e_{\calP}(x) = \mu(x)-L(x) = \eta(x)-L(x)+\frac{h(x)}{2}, \text{ and } e_{\calT}(x) = \nu(x)-L(x) = \eta(x)-L(x)-\frac{h(x)}{2}. \qedhere
\]
\end{proof}

We now turn to the second moments. For $s\in\{\calP,\calT\}$, define
\begin{equation}\label{eq:defn_B_s}
B_s(x) \coloneqq \E\mleft[\mleft(X_s(x)-L(x)\mright)^2\mright].
\end{equation}
Notice that $B_s$ is not exactly the variance of $X_s(x)$, since the centring $L(x)$ need not equal the exact mean. By~\eqref{eq:defn_typewise_mean_errors}, we have instead, that 
\begin{equation}\label{eq:Var_from_B}
\Var(X_s(x)) = B_s(x)-e_s(x)^2, \quad \text{ for }s \in\{\calP,\calT\}.
\end{equation}

We use throughout the following standard consequence of the recursive construction of $\mathcal{UDA}$. Conditional on the first deposited trapezium, the subsequent parking procedures on the two disjoint child strips are independent. Equivalently, one may construct the process by using independent arrival sequences on the two child strips after the first split. Thus, conditional on the first deposition position, the numbers of trapeziums eventually deposited on the two child strips are independent.

\begin{lemma}\label{lem:second_moment_recurrences}
Let $B_\calP$ and $B_\calT$ be defined as in~\eqref{eq:defn_B_s}. For all $x>(3a-b)/2$, it follows that
\begin{align}
B_{\calP}(x)
&=
\frac{1}{x-a} \mleft( \int_0^{x-a}B_{\calP}(t)\D t + \int_\delta^{x-\varsigma }B_{\calT}(t)\D t \mright) + q_{\calP}(x), \quad \text{ and } \label{eq:BP_rec}\\
B_{\calT}(x)
&=
\frac{1}{x-a} \mleft( \int_\delta^{x-a}B_{\calT}(t)\D t + \int_0^{x-\varsigma }B_{\calP}(t)\D t \mright) + q_{\calT}(x), \label{eq:BT_rec}
\end{align}
where
\begin{align}
q_{\calP}(x) 
&=
\frac{1}{x-a} \mleft( \int_0^{x-a} e_{\calP}(t)e_{\calT}(x-\varsigma -t)\D t + \int_\delta^{x-\varsigma } e_{\calT}(t)e_{\calP}(x-\varsigma -t)\D t \mright),\quad \text{ and } \label{eq:qP_def}\\
q_{\calT}(x)
&=
\frac{1}{x-a} \mleft( \int_0^{x-a-\delta} e_{\calT}(t+\delta)e_{\calT}(x-a-t)\D t + \int_0^{x-\varsigma }  e_{\calP}(t)e_{\calP}(x-\varsigma -t)\D t \mright). \label{eq:qT_def}
\end{align}
\end{lemma}

\begin{proof}
We first derive the recursion for $B_{\calP}$. Let $x>(3a-b)/2$, in which case the recursive representation~\eqref{eq:probibalistic_vers_P_x_T_x} applies, and condition on $\tau\sim\mathcal U(0,2x-2a)$.

Suppose first, that $\tau\in(0,x-a)$. Then the first deposited trapezium leaves two strips, one of type $\calP$ and size $\tau$, and one of type $\calT$ and size $x-\varsigma -\tau$. Hence $P_x = 1+P_\tau+T_{x-\varsigma -\tau}$, where, conditionally on $\tau$, the two random variables $P_\tau$ and $T_{x-\varsigma -\tau}$ on the right-hand side are independent. By~\eqref{eq:L_exact_additivity}, we have that $L(x) = 1+L(\tau)+L(x-\varsigma -\tau)$, and therefore
\[
P_x-L(x) = \big(P_\tau-L(\tau)\big) + \big(T_{x-\varsigma -\tau}-L(x-\varsigma -\tau)\big).
\]
Consequently, conditional on $\tau=t$,
\[
\E\mleft[ \mleft(P_x-L(x)\mright)^2 \,\big|\,\tau=t \mright] = B_{\calP}(t) + B_{\calT}(x-\varsigma -t) + 2e_{\calP}(t)e_{\calT}(x-\varsigma -t).
\]

On the second branch, $\tau\in[x-a,2x-2a)$, put $t=t_\tau = \tau+(3a-b)/2-x$, then $t\in[\delta,x-\varsigma ]$, and the two child strips have types $\calT,\calP$ and sizes $t$, and $x-\varsigma -t$, respectively. Thus,
\[
\E\mleft[ \mleft(P_x-L(x)\mright)^2 \,\big|\,t_\tau=t \mright] = B_{\calT}(t) + B_{\calP}(x-\varsigma -t) + 2e_{\calT}(t)e_{\calP}(x-\varsigma -t).
\]

Averaging over the two branches, i.e.\ taking expectations w.r.t. $\tau$, gives
\begin{equation}\label{eq:decomp_B_P_}
    \begin{aligned}
B_{\calP}(x)
&=
\frac{1}{2x-2a} \int_0^{x-a} B_{\calP}(t) + B_{\calT}(x-\varsigma -t) + 2e_{\calP}(t)e_{\calT}(x-\varsigma -t) \D t\\
&\quad
+ \frac{1}{2x-2a} \int_\delta^{x-\varsigma } B_{\calT}(t) + B_{\calP}(x-\varsigma -t) + 2e_{\calT}(t)e_{\calP}(x-\varsigma -t) \D t.
\end{aligned}
\end{equation}
Changing variables $u=x-\varsigma -t$ in the appropriate terms of~\eqref{eq:decomp_B_P_}, we obtain
\[
\int_0^{x-a}B_{\calT}(x-\varsigma -t)\D t = \int_\delta^{x-\varsigma }B_{\calT}(u)\D u, \quad \text{and}\quad  \int_\delta^{x-\varsigma }B_{\calP}(x-\varsigma -t)\D t = \int_0^{x-a}B_{\calP}(u)\D u.
\]
Combining the two cross-products from~\eqref{eq:decomp_B_P_} yields exactly $q_\calP(x)$. Hence~\eqref{eq:decomp_B_P_} concludes~\eqref{eq:BP_rec} and~\eqref{eq:qP_def}. The proof for $B_{\calT}$ is completely analogous to the present proof for $B_\calP$. 
\end{proof}

We next estimate the inhomogeneous $q_s$-terms appearing in Lemma~\ref{lem:second_moment_recurrences}.

\begin{lemma}\label{lem:q_toll_bound}
There exists a finite constant $C>0$ such that, for all sufficiently large $x$,
\begin{equation}\label{eq:defn_R2_variance_section}
\max_{s\in\{\calP,\calT\}}|q_s(x)| \le C\mathcal R_2(x), \quad \text{where}\quad \mathcal R_2(x) \coloneqq
\begin{dcases}
    \mleft(\dfrac{4e}{x/a}\mright)^{x/a-4},
    &\text{if }a=b,\\
    \mleft(\dfrac{4e}{x/a}\mright)^{x/a-(5+(a+b)/(2a))},
    &\text{if }a>b.
\end{dcases}
\end{equation}
\end{lemma}

\begin{proof}
We prove the estimate for $q_{\calP}$. The proof for $q_{\calT}$ is completely analogous. Let $\beta \coloneqq (3/2)\1_{\{a=b\}}+(5/2)\1_{\{a>b\}}$. Choose $x_0>0$ sufficiently large, such that Lemma~\ref{lem:typewise_mean_error} applies for all $x \ge x_0$. Since the functions $e_s$ are locally bounded, there also exists $C_0<\infty$, such that $ |e_s(y)|\le C_0$ for all $0 \le y \le x_0$ and $s \in \{\calP,\calT\}$.

Consider the first integral in~\eqref{eq:qP_def}. Choose $x_0>\delta$ sufficiently large that the estimate of Lemma~\ref{lem:typewise_mean_error} holds for all $x\ge x_0$, and assume that $x>\varsigma +2x_0$. Since $x-a=x-\varsigma -\delta$ and $x_0>\delta$, we have $x-\varsigma -x_0<x-a$. Therefore, we may decompose
\begin{align}\label{eq:qP_first_integral_decomposition}
\begin{aligned}
\int_0^{x-a} \mleft| e_{\calP}(t)e_{\calT}(x-\varsigma -t) \mright|\D t &= \int_0^{x_0} \mleft| e_{\calP}(t)e_{\calT}(x-\varsigma -t) \mright|\D t + \int_{x_0}^{x-\varsigma -x_0} \mleft| e_{\calP}(t)e_{\calT}(x-\varsigma -t) \mright|\D t \\
&\quad + \int_{x-\varsigma -x_0}^{x-a} \mleft| e_{\calP}(t)e_{\calT}(x-\varsigma -t) \mright|\D t.
\end{aligned}
\end{align}
The first and third terms in~\eqref{eq:qP_first_integral_decomposition} are the endpoint contributions. Indeed, if $t\in[0,x_0]$, then $e_{\calP}(t)$ is uniformly bounded and $x-\varsigma -t \in [x-\varsigma -x_0,x-\varsigma ]$. Thus, the argument of $e_{\calT}$ differs from $x$ by a uniformly bounded quantity. Since $\mathcal R_1$ is eventually decreasing, it follows from Lemma~\ref{lem:typewise_mean_error} that
\[
\int_0^{x_0} \mleft| e_{\calP}(t)e_{\calT}(x-\varsigma -t) \mright|\D t \le C \int_0^{x_0} \mathcal R_1(x-\varsigma -t)\D t\le C x_0\mathcal R_1(x-\varsigma -x_0).
\]
Similarly, if $t\in[x-\varsigma -x_0,x-a]$, then $x-\varsigma -t\in[\delta,x_0]$. Hence, $e_{\calT}(x-\varsigma -t)$ is uniformly bounded, whereas $t\ge x-\varsigma -x_0$. The length of this interval is $(x-a)-(x-\varsigma -x_0)  = x_0-\delta$, and consequently
\[
\int_{x-\varsigma -x_0}^{x-a} \mleft| e_{\calP}(t)e_{\calT}(x-\varsigma -t) \mright|\D t \le C(x_0-\delta) \mathcal R_1(x-\varsigma -x_0).
\]
Combining the two endpoint estimates, we obtain
\begin{equation}\label{eq:qP_endpoint_contribution}
\frac{1}{x-a} \int_{[0,x_0] \cup [x-\varsigma -x_0,x-a]} \mleft| e_{\calP}(t)e_{\calT}(x-\varsigma -t) \mright|\D t 
=
\Oh\mleft( \frac{1}{x} \mathcal R_1(x-\varsigma -x_0) \mright)
=
\Oh(\mathcal R_2(x)).
\end{equation}
To verify the final estimate of~\eqref{eq:qP_endpoint_contribution}, write $z=x/a$ and $c=(\varsigma +x_0)/a$, then $(x-\varsigma -x_0)/a=z-c$ where $c>0$ is fixed. Note that $(z/(z-c))^{z-c-\beta} =(1+c/(z-c))^{z-c-\beta} \to e^c$ as $z \to \infty$, hence by definition of $\beta$ and $\alpha= 4\1_{\{a=b\}}+ (5+(a+b)/(2a)) \1_{\{a>b\}}$, it follows that
\[
\frac{ x^{-1}\mathcal R_1(x-\varsigma -x_0)}{\mathcal R_2(x)}
=
\frac{1}{az} \mleft(\dfrac{2e}{z-c}\mright)^{z-c-\beta} \mleft(\dfrac{4e}{z}\mright)^{\alpha-z} 
= 
\Oh\mleft( 2^{-z}z^{c+\beta-\alpha-1} \mright)
=
\oh(1),
\]
concluding the endpoint contribution in~\eqref{eq:qP_endpoint_contribution}. 

It remains to estimate the contribution from the middle interval $[x_0,x-\varsigma -x_0]$. For every $t$ in this interval, we know that $t\ge x_0$ and $x-\varsigma -t\ge x_0$, so Lemma~\ref{lem:typewise_mean_error} may be applied simultaneously to both factors. Put $z=x/a$, $c=\varsigma /a$, $\ell= z-c=(x-\varsigma )/a$, $u=t/a$, and recall the definition of $\beta$. By Lemma~\ref{lem:typewise_mean_error}, there exists a constant $C>0$, such that
\begin{equation}\label{eq:q_middle_product_bound}
\mleft| e_{\calP}(t)e_{\calT}(x-\varsigma -t) \mright| \le C \mleft(\frac{2e}{u}\mright)^{u-\beta} \mleft(\frac{2e}{\ell-u}\mright)^{\ell-u-\beta}.
\end{equation}
We rewrite the product on the right-hand side as
\begin{equation}\label{eq:q_middle_product_rewrite}
\mleft(\frac{2e}{u}\mright)^{u-\beta} \mleft(\frac{2e}{\ell-u}\mright)^{\ell-u-\beta} = (2e)^\ell u^{-u}(\ell-u)^{-(\ell-u)} (2e)^{-2\beta} u^\beta(\ell-u)^\beta.
\end{equation}
The function $u\to u\log u+(\ell-u)\log(\ell-u)$ is strictly convex on $(0,\ell)$ and attains its unique minimum at $u=\ell/2$. Therefore, $u\log u+(\ell-u)\log(\ell-u) \ge \ell\log(\ell/2)$, or, equivalently, $ u^u(\ell-u)^{\ell-u} \ge (\ell/2)^\ell$. Thus, $u^{-u}(\ell-u)^{-(\ell-u)} \le (2/\ell)^\ell$. Moreover, by the arithmetic-geometric mean inequality, $u(\ell-u)\le \ell^2/4$, we also have that $u^\beta(\ell-u)^\beta \le (\ell^2/4)^\beta$. Substituting these two estimates into~\eqref{eq:q_middle_product_rewrite}, gives
\begin{equation}\label{eq:q_convolution_pointwise}
\mleft(\frac{2e}{u}\mright)^{u-\beta} \mleft(\frac{2e}{\ell-u}\mright)^{\ell-u-\beta} 
\le 
C_\beta \mleft(\frac{4e}{\ell}\mright)^\ell \ell^{2\beta},
\end{equation}
where $C_\beta>0$ depends only on $\beta$. It follows from~\eqref{eq:q_middle_product_bound} and~\eqref{eq:q_convolution_pointwise}, that
\begin{equation}\label{eq:q_middle_integrated_bound}
\frac{1}{x-a} \int_{x_0}^{x-\varsigma -x_0} \mleft| e_{\calP}(t)e_{\calT}(x-\varsigma -t) \mright|\D t 
\le 
\frac{C(x-\varsigma -2x_0)}{x-a} \mleft(\frac{4e}{\ell}\mright)^\ell \ell^{2\beta}.
\end{equation}
Since $0\le (x-\varsigma -2x_0)/(x-a) \le x/(x-a)$, the fraction on the right-hand side of~\eqref{eq:q_middle_integrated_bound} is uniformly bounded for all sufficiently large $x$. Consequently,
\begin{equation}\label{eq:q_middle_pre_R2}
\frac{1}{x-a} \int_{x_0}^{x-\varsigma -x_0} \mleft| e_{\calP}(t)e_{\calT}(x-\varsigma -t) \mright|\D t
\le
C \mleft(\frac{4e}{\ell}\mright)^\ell \ell^{2\beta}.
\end{equation}

We now compare the right-hand side of~\eqref{eq:q_middle_pre_R2} with $\mathcal R_2(x)$. Recall that $ \ell=z-c$, and $\alpha = 4\1_{\{a=b\}}+(5+(a+b)/(2a))\1_{\{a>b\}}$. Then, $\mathcal R_2(x) = (4e/z)^{z-\alpha}$.
We have
\begin{equation}\label{eq:q_R2_ratio}
\mleft(\dfrac{4e}{z-c}\mright)^{z-c} (z-c)^{2\beta} \mleft(\dfrac{4e}{z}\mright)^{\alpha-z}= (4e)^{\alpha-c} z^{c+2\beta-\alpha} \mleft( \frac{z}{z-c} \mright)^{z-c-2\beta}.
\end{equation}
Furthermore, $(z/(z-c))^{z-c-2\beta} =(1+c/(z-c))^{z-c-2\beta} \to e^c$ as $z \to \infty$. Thus, this factor remains bounded for all sufficiently large $z$.

If $a=b$, then $c=1$, $\beta=3/2$, and $\alpha=4$, and therefore $c+2\beta-\alpha =0$. If $a>b$, then $c=\varsigma /a<1$, $\beta=5/2$, and $\alpha=(5+(a+b)/(2a))=c+2\beta$, and hence $c+2\beta-\alpha=0$. In either case, $z^{c+2\beta-\alpha}$ is bounded for all sufficiently large $z$. Hence, it follows from~\eqref{eq:q_R2_ratio} that
\begin{equation}\label{eq:q_R2_comparison}
\mleft(\frac{4e}{\ell}\mright)^{\ell} \ell^{2\beta}=\mleft(\frac{4e}{z-c}\mright)^{z-c} (z-c)^{2\beta}
=
\Oh(\mathcal R_2(x)),
\end{equation} as $x \to \infty$ (i.e. $z \to \infty$). Combining~\eqref{eq:q_middle_pre_R2} and~\eqref{eq:q_R2_comparison}, we obtain
\begin{equation}\label{eq:qP_middle_final}
\frac{1}{x-a} \int_{x_0}^{x-\varsigma -x_0} \mleft| e_{\calP}(t)e_{\calT}(x-\varsigma -t) \mright|\D t
=
\Oh(\mathcal R_2(x)).
\end{equation} Together with the endpoint estimate~\eqref{eq:qP_endpoint_contribution}, this yields
\begin{equation}\label{eq:qP_first_integral_final}
\frac{1}{x-a} \int_0^{x-a} \mleft| e_{\calP}(t)e_{\calT}(x-\varsigma -t) \mright|\D t
=
\Oh(\mathcal R_2(x)).
\end{equation}

We next observe that the two integrals defining $q_{\calP}$ are in fact identical. Indeed, using the change of variables $u=x-\varsigma -t$, and the identity $x-\varsigma -\delta=x-a$, we obtain
\begin{equation*}
\int_\delta^{x-\varsigma } e_{\calT}(t)e_{\calP}(x-\varsigma -t)\D t = \int_0^{x-a} e_{\calP}(u)e_{\calT}(x-\varsigma -u)\D u.
\end{equation*} Hence,~\eqref{eq:qP_def} may be written as $q_{\calP}(x) = 2 (x-a)^{-1} \int_0^{x-a} e_{\calP}(t)e_{\calT}(x-\varsigma -t)\D t$. It follows immediately from~\eqref{eq:qP_first_integral_final} that $ q_{\calP}(x) = \Oh(\mathcal R_2(x))$, concluding the proof.
\end{proof}

We shall use the following scalar version of the Dvoretzky--Robbins argument~\cite{Dvoretzky_Robbins}. The additional term $C_0/x$ in the recursion is harmless. This is the same phenomenon observed in~\cite[Thm~1]{Dvoretzky_Robbins}, that such a term disappears from the fundamental comparison identity.

\begin{lemma}
\label{lem:scalar_DR_variance}
Let $F:[0,\infty)\to[0,\infty)$ be locally bounded and continuous for all sufficiently large arguments. Suppose that, for some $C_0\in\R$ and all sufficiently large $x$,
\begin{equation}\label{eq:scalar_DR_equation}
F(x+a) = \frac{1}{x} \mleft( \int_0^xF(t)\D t + \int_\delta^{x+\delta}F(t)\D t \mright) + \frac{C_0}{x} + p(x+a), 
\end{equation}
where $ |p(x)|\le C\mathcal R_2(x)$ for all sufficiently large $x$. Then, there exists a constant $\rho\ge0$ such that
\[
F(x) = \rho\cdot (x+\varsigma ) + \Oh(\mathcal R_2(x)).
\]
\end{lemma}

\begin{proof}
Let $H(x)\coloneqq x+\varsigma $. A direct calculation shows that $H$ satisfies the homogeneous equation
\[
H(x+a) = \frac{1}{x} \mleft( \int_0^xH(t)\D t + \int_\delta^{x+\delta}H(t)\D t \mright).
\]

Let $0<x\le y$ be sufficiently large. Multiplying~\eqref{eq:scalar_DR_equation} at $x$ by $x/y$ and subtracting the resulting identity from~\eqref{eq:scalar_DR_equation} at $y$, we obtain
\begin{equation}\label{eq:scalar_DR_fundamental}
F(y+a) = \frac{x}{y}F(x+a) + \frac{1}{y} \mleft(\int_x^yF(t)\D t + \int_{x+\delta}^{y+\delta}F(t)\D t \mright) + p(y+a) - \frac{x}{y}p(x+a).
\end{equation}
Notice that the terms involving $C_0$ cancel exactly.

For $A\in\R$, define $G_A(t)\coloneqq F(t)-A\cdot (t+\varsigma )$. Since $H(t)=t+\varsigma $ satisfies the homogeneous equation,~\eqref{eq:scalar_DR_fundamental} yields
\begin{equation}\label{eq:scalar_DR_fundamental_G}
G_A(y+a) = \frac{x}{y}G_A(x+a) + \frac{1}{y} \mleft( \int_x^yG_A(t)\D t + \int_{x+\delta}^{y+\delta}G_A(t)\D t \mright) + p(y+a) - \frac{x}{y}p(x+a).
\end{equation}

Define $I_x \coloneqq \inf_{x\le t\le x+a} F(t)/(t+\varsigma )$ and $S_x \coloneqq \sup_{x\le t\le x+a} F(t)/(t+\varsigma )$. Furthermore, let $ p_x^\ast \coloneqq \sup_{x\le t\le x+2a}|p(t)|$ and $ \varepsilon_x \coloneqq 2p_x^\ast/(x+a+\varsigma )$. Take $A=I_x$ in~\eqref{eq:scalar_DR_fundamental_G}. If $ x\le y\le x+\varsigma $, then $ [x,y]\subseteq[x,x+a]$ and, since $\varsigma +\delta=a$, it follows that $ [x+\delta,y+\delta]\subseteq[x,x+a]$. Hence, all the terms involving $G_{I_x}$ on the right-hand side of~\eqref{eq:scalar_DR_fundamental_G} are non-negative. Therefore, $G_{I_x}(y+a)\ge-2p_x^\ast$. Since, $y+a+\varsigma \ge x+a+\varsigma $, we obtain $ F(y+a)/(y+a+\varsigma ) \ge  I_x-\varepsilon_x$. Together with the definition of $I_x$, this shows that $I_z\ge I_x-\varepsilon_x$ for all $ x\le z\le x+\varsigma $. The same argument with $A=S_x$ gives $S_z\le S_x+\varepsilon_x$ for all $x\le z\le x+\varsigma $.

Define now $\Delta_x  \coloneqq \sum_{j=0}^\infty \varepsilon_{x+j\varsigma }$, and we claim that
\begin{equation}\label{eq:Delta_R2_bound}
    \Delta_x  = \Oh\mleft( \frac{\mathcal R_2(x)}{x} \mright).
\end{equation}
Indeed, by definition of $p_x^*$, since $ p(t)=\Oh(\mathcal R_2(t))$ and $\mathcal R_2$ is eventually decreasing, there exists a constant $C>0$ such that, for all sufficiently large $x$, we have $ p_x^\ast \le C\mathcal R_2(x)$. Consequently,
\begin{equation}\label{eq:epsilon_R2_bound}
    \varepsilon_x \le  C\frac{\mathcal R_2(x)}{x}.
\end{equation} We next use the fact that
\begin{equation}\label{eq:R2_gamma_ratio_zero}
    \frac{\mathcal R_2(x+\varsigma )}{\mathcal R_2(x)} \to  0, \qquad\text{as }x\to\infty.
\end{equation}
To verify this, write $z=x/a$, $d=\varsigma /a>0$, and let $\alpha = 4\1_{\{a=b\}}+(5+(a+b)/(2a))\1_{\{a>b\}}$. Then, $\mathcal R_2(x) = (4e/z)^{z-\alpha}$, and therefore
\begin{equation*}
\frac{\mathcal R_2(x+\varsigma )} {\mathcal R_2(x)}= \mleft(\dfrac{4e}{z+d}\mright)^{z+d-\alpha}  \mleft(\dfrac{4e}{z}\mright)^{\alpha-z}= (4e)^d (z+d)^{-d} \mleft( \frac{z}{z+d} \mright)^{z-\alpha}\to  0, \quad \text{as }x \to \infty,
\end{equation*}
Since $( z/(z+d))^{z-\alpha}\to e^{-d}$ and $(z+d)^{-d}\to 0$ as $x \to \infty$, we obtain~\eqref{eq:R2_gamma_ratio_zero}.

It follows that there exists $x_1>0$ such that $ \mathcal R_2(y+\varsigma ) \le \frac12\mathcal R_2(y)$ for every $y\ge x_1$. Hence, for $x\ge x_1$ and every $j\ge0$,
\begin{equation}\label{eq:R2_geometric_decay}
    \mathcal R_2(x+j\varsigma ) \le 2^{-j}\mathcal R_2(x).
\end{equation}
Using~\eqref{eq:epsilon_R2_bound} and~\eqref{eq:R2_geometric_decay}, we obtain
\begin{equation*}
\Delta_x = \sum_{j=0}^\infty \varepsilon_{x+j\varsigma } \le C \sum_{j=0}^\infty \frac{ \mathcal R_2(x+j\varsigma ) }{ x+j\varsigma } 
\le \frac{C}{x} \sum_{j=0}^\infty \mathcal R_2(x+j\varsigma ) \le \frac{C\mathcal R_2(x)}{x} \sum_{j=0}^\infty2^{-j} = \frac{2C\mathcal R_2(x)}{x},
\end{equation*} proving~\eqref{eq:Delta_R2_bound}.

Iterating the inequalities $I_z\ge I_x-\varepsilon_x$ for all $ x\le z\le x+\varsigma $ and $S_z\le S_x+\varepsilon_x$ for all $x\le z\le x+\varsigma $ therefore gives, for all sufficiently large $y\ge x$, that
\begin{equation}\label{eq:IS_propagation}
    I_y\ge I_x-\Delta_x, \qquad \text{and}\qquad  S_y\le S_x+\Delta_x.
\end{equation}

Since $F\ge0$, the quantities $I_x$ are bounded below. Moreover, fixing one sufficiently large $x_0$,~\eqref{eq:IS_propagation} shows that $S_y\le S_{x_0}+\Delta_{x_0}$ for all $y \ge x_0$, so the quantities $S_y$ are bounded above. Taking lower and upper limits in~\eqref{eq:IS_propagation}, and then letting $x\to\infty$, shows that both limits $I_\infty \coloneqq \lim_{x\to\infty}I_x$ and $S_\infty \coloneqq \lim_{x\to\infty}S_x$ exist and are finite.

We next show that these limits coincide. Since $S_y$ is bounded from above and $I_y$ is bounded from below, it follows that $F(x)=\Oh(x)$. Using~\eqref{eq:scalar_DR_fundamental}, for $x\le y\le x+a$ we therefore obtain $|F(y+a)-F(x+a)|=\Oh(1)$. Indeed, the first term produced by subtracting $F(x+a)$ is bounded by $(y-x)y^{-1}|F(x+a)|=\Oh(1)$, each integral has bounded length and an integrand of order $x$, and the terms involving $p$ tend to zero.

It follows that
\[
\sup_{x+a\le u,v\le x+2a} \mleft| \frac{F(u)}{u+\varsigma } - \frac{F(v)}{v+\varsigma } \mright| = \Oh(x^{-1}), \quad \text{ as } x \to \infty.
\]
Hence, $S_x-I_x\to 0$ as $x \to \infty$, and therefore $I_\infty=S_\infty$. Denote their common value by $\rho$. Since $F\ge0$, we have $\rho\ge0$.

From~\eqref{eq:IS_propagation}, letting $y\to\infty$ gives $I_x-\Delta_x \le \rho \le S_x+\Delta_x$ for all sufficiently large $x$. Since $F(t)/(t+\varsigma )$ is continuous for all sufficiently large $t$, there therefore exists $x_\varpi\in[\varpi,\varpi+a]$ such that $| F(x_\varpi)/(x_\varpi+\varsigma )-\rho|\le \Delta_\varpi$. By~\eqref{eq:Delta_R2_bound}, it thus follows that
\begin{equation}\label{eq:scalar_anchor_bound}
    \mleft| F(x_\varpi)-\rho\cdot (x_\varpi+\varsigma ) \mright| = \Oh(\mathcal R_2(\varpi)).
\end{equation} Set $G(x)\coloneqq F(x)-\rho\cdot (x+\varsigma )$ and $ \mathfrak D(\varpi) \coloneqq \sup_{\varpi\le t\le\varpi+a}|G(t)|$. We now derive a one-step estimate for $\mathfrak D$. Suppose first that $x_\varpi\le t\le\varpi+a$, and apply~\eqref{eq:scalar_DR_fundamental_G} with $x=x_\varpi-a$, $y=t-a$, and $A=\rho$. The first integration interval in~\eqref{eq:scalar_DR_fundamental_G} is contained in $[\varpi-a,\varpi]$, and the second is contained in $[\varpi-\varsigma ,\varpi-\varsigma +a]$. Using~\eqref{eq:scalar_anchor_bound}, we thereby obtain $ |G(t)| \le a(\varpi-a)^{-1} ( \mathfrak D(\varpi-a)  + \mathfrak D(\varpi-\varsigma ) ) + C\mathcal R_2(\varpi)$. Suppose next that $ \varpi\le t\le x_\varpi$, and apply~\eqref{eq:scalar_DR_fundamental_G} with $x=t-a$, $y=x_\varpi-a$, $A=\rho$, and solve the resulting identity for $G(t)$. Since $(x_\varpi-a)/(t-a) \le \varpi/(\varpi-a)$, which is bounded for large $\varpi$, the same argument gives $|G(t)| \le a (\varpi-a)^{-1} ( \mathfrak D(\varpi-a)  + \mathfrak D(\varpi-\varsigma ) ) + C\mathcal R_2(\varpi)$. Thus,
\begin{equation}\label{eq:scalar_DR_one_step}
    \mathfrak D(\varpi) \le \frac{a}{\varpi-a} \mleft( \mathfrak D(\varpi-a) + \mathfrak D(\varpi-\varsigma ) \mright) + C\mathcal R_2(\varpi).
\end{equation}

We finally show that~\eqref{eq:scalar_DR_one_step} implies $\mathfrak D(\varpi)=\Oh(\mathcal R_2(\varpi))$. Write $z=\varpi/a$, and recall definition of $\alpha$. Then, $ \mathcal R_2(\varpi) = (4e/z)^{z-\alpha}$, and
\[
\frac{2a}{\varpi-a} \frac{\mathcal R_2(\varpi-a)} {\mathcal R_2(\varpi)} = \frac{2}{z-1} \mleft(\dfrac{4e}{z-1}\mright)^{z-1-\alpha}\mleft(\dfrac{4e}{z}\mright)^{\alpha-z} = \frac{z}{2e(z-1)} \mleft( \frac{z}{z-1} \mright)^{z-1-\alpha} \to  \frac12,
\] as $\varpi \to \infty$. Since $\varsigma \le a$ and $\mathcal R_2$ is eventually decreasing, it follows that $\mathcal R_2(\varpi-\varsigma ) \le \mathcal R_2(\varpi-a)$. Hence, for all sufficiently large $\varpi$,
\[
\frac{a}{\varpi-a} \frac{ \mathcal R_2(\varpi-a)  + \mathcal R_2(\varpi-\varsigma )}{ \mathcal R_2(\varpi) } \le \frac34.
\]
Choosing a sufficiently large constant $M$, and applying~\eqref{eq:scalar_DR_one_step} successively on intervals of length $\varsigma $, we obtain $ \mathfrak D(\varpi)\le M\mathcal R_2(\varpi)$ for all sufficiently large $\varpi$. This proves $F(x) = \rho(x+\varsigma ) + \Oh(\mathcal R_2(x))$, and hence concludes the proof.
\end{proof}

We can now apply Lemma~\ref{lem:scalar_DR_variance} to the second moments. Define
\begin{equation}\label{eq:defn_average_difference_B}
\overline{B}(x) \coloneqq \frac{B_{\calP}(x)+B_{\calT}(x)}{2}, \qquad \text{and} \qquad  \Delta_B(x) \coloneqq B_{\calP}(x)-B_{\calT}(x).
\end{equation}

\begin{lemma}\label{lem:B_second_moment_asymptotics}
There exists a constant $\lambda_2(a,b)\ge0$, such that, as $x \to \infty$,
\[
B_{\calP}(x)
= \lambda_2(a,b)(x+\varsigma ) + \Oh(\mathcal R_2(x)), \quad \text{and}\quad  B_{\calT}(x) = \lambda_2(a,b)(x+\varsigma ) + \Oh(\mathcal R_2(x)).
\]
\end{lemma}

\begin{proof}
We first control the difference $\Delta_B$. Subtracting~\eqref{eq:BT_rec} from~\eqref{eq:BP_rec} gives
\begin{equation}\label{eq:D_B_recursion}
\Delta_B(x) = -\frac{1}{x-a} \int_{x-a}^{x-\varsigma } \Delta_B(t)\D t + r(x), \quad \text{where}\quad  r(x)\coloneqq q_{\calP}(x)-q_{\calT}(x).
\end{equation} By Lemma~\ref{lem:q_toll_bound}, $r(x)=\Oh(\mathcal R_2(x))$. If $a=b$, then $\calP_x$ and $\calT_x$ coincide, and therefore $B_{\calP}(x)=B_{\calT}(x)$ for every $x$, so $\Delta_B \equiv0$. Assume instead that $a>b$. Since the interval of integration in~\eqref{eq:D_B_recursion} has length $\delta$, and $\mathcal R_2$ is eventually decreasing, if $|\Delta_B(t)|\le M\mathcal R_2(t)$ for all earlier arguments, then $|\Delta_B(x)| \le M \delta(x-a)^{-1}\mathcal R_2(x-a) + C\mathcal R_2(x)$. Moreover,
\[
 \frac{\delta}{x-a} \frac{\mathcal R_2(x-a)} {\mathcal R_2(x)} \rightarrow \frac{\delta}{4a} \le \frac18.
\]
Thus, by choosing $M$ sufficiently large and arguing successively on intervals of length $\varsigma $, we obtain
\begin{equation}\label{eq:D_B_R2}
    \Delta_B(x)=\Oh(\mathcal R_2(x)).
\end{equation}

We next derive a scalar equation for $\overline{B}$. Adding~\eqref{eq:BP_rec} and~\eqref{eq:BT_rec}, and using $B_{\calP}=\overline{B}+\Delta_B/2$ and $B_{\calT}=\overline{B}-\Delta_B/2$, gives
\[
\overline{B}(x) = \frac{1}{x-a} \mleft( \int_\delta^{x-a}\overline{B}(t)\D t +\int_\delta^{x-\varsigma }\overline{B}(t)\D t + \int_0^\delta B_{\calP}(t)\D t \mright) + \frac{q_{\calP}(x)+q_{\calT}(x)}{2}.
\]
Replacing $x$ by $x+a$ yields
\begin{equation}\label{eq:A(x+a)_identity}
 \overline{B}(x+a)
=
\frac{1}{x} \mleft( \int_\delta^x \overline{B}(t)\D t + \int_\delta^{x+\delta} \overline{B}(t)\D t  + \int_0^\delta B_{\calP}(t)\D t \mright) + \frac{q_{\calP}(x+a)+q_{\calT}(x+a)}{2}.
\end{equation}
Since, $-\int_0^\delta \overline{B}(t)\D t + \int_0^\delta B_{\calP}(t)\D t
=
\int_0^\delta \mleft( B_{\calP}(t)-B_{\calT}(t) \mright)\D t/2= \int_0^\delta \Delta_B(t)\D t/2$, we can rewrite~\eqref{eq:A(x+a)_identity}, as
\begin{equation}\label{eq:A_scalar_recursion}
\begin{aligned}
\overline{B}(x+a)
&=
\frac{1}{x} \mleft( \int_0^x \overline{B}(t)\D t + \int_\delta^{x+\delta} \overline{B}(t)\D t \mright) + \frac{C_0}{x} + p(x+a),
\end{aligned}
\end{equation}
where $C_0 \coloneqq \int_0^\delta \Delta_B(t)\D t/2$ and $p(x) \coloneqq (q_{\calP}(x)+q_{\calT}(x))/2$. By Lemma~\ref{lem:q_toll_bound}, we thereby get $p(x)=\Oh(\mathcal R_2(x))$. Furthermore, $\overline{B}$ is non-negative and locally bounded, and the integral recursions imply that $\overline{B}$ is continuous for all sufficiently large arguments. Lemma~\ref{lem:scalar_DR_variance} therefore yields a constant $\lambda_2(a,b)\ge0$ such that
\begin{equation}\label{eq:A_second_order_asymptotic}
\overline{B}(x)  = \lambda_2(a,b)(x+\varsigma ) + \Oh(\mathcal R_2(x)).
\end{equation}
Combining~\eqref{eq:A_second_order_asymptotic} with~\eqref{eq:D_B_R2}, $B_{\calP}=\overline{B}+\Delta_B/2$, and $B_{\calT}=\overline{B}-\Delta_B/2$ proves the result.
\end{proof}

It remains to show that the constant $\lambda_2(a,b)$ is strictly positive. For this, we adapt the argument used by Dvoretzky and Robbins in the proof of~\cite[Thm~4]{Dvoretzky_Robbins}.

\begin{lemma}
\label{lem:positive_variance_lower_bound}
There exists a $c=c(a,b)>0$ such that $\Var(P_x)\ge c/x$ for all sufficiently large $x$.
\end{lemma}

\begin{proof}
Write $V_{\calP}(x)\coloneqq\Var(P_x)$ and $V_{\calT}(x)\coloneqq\Var(T_x)$. By the law of total variance, it follows that
\[
\Var(P_x) = \E\big[\Var(P_x\mid\tau)\big] + \Var\big(\E[P_x\mid\tau]\big) \ge \E\big[\Var(P_x\mid\tau)\big].
\]
Conditional on the first deposition, the parking numbers on the two child strips are independent. Repeating the same changes of variables as in the proof of Lemma~\ref{lem:second_moment_recurrences}, we therefore obtain
\begin{equation}\label{eq:variance_lower_rec_P}
V_{\calP}(x) \ge \frac{1}{x-a} \mleft( \int_0^{x-a}V_{\calP}(t)\D t + \int_\delta^{x-\varsigma }V_{\calT}(t)\D t \mright),
\end{equation} for all sufficiently large $x$.

We shall show that, for every choice of $0\le b\le a$, there exists a non-empty bounded open interval $J$ and a type $s\in\{\calP,\calT\}$ such that $\Var(X_s(t))>0$ for all $t \in J$. Assume first that $0<b<a$. Recall from the proof of Lemma~\ref{lem:eta_func} that, for $a+b<x<(3a+b)/2$, we have
\[
P_x = 2\1_{\{\tau\in(0,x-a-b)\}} + \1_{\{\tau\in(x-a-b,x-a+b)\}} + 2\1_{\{\tau\in(x-a+b,2x-2a)\}}, \quad \text{where}\quad \tau\sim\mathcal U(0,2x-2a).
\]
Since $b>0$ and $x>a+b$, both the event on which $P_x=1$ and the event on which $P_x=2$ have strictly positive probability. Hence, $V_{\calP}(x)>0$ for all $a+b<x<(3a+b)/2$. Assume now instead that $b=0$. In this case, we can show that $V_{\calT}(x)>0$ for all $3a/2<x<2a$. Fix such an $x$. On the first branch in the recursive representation of $T_x$, namely $ 0<\tau<x-3a/2$, the two child strips are of type $\calT$ and have sizes $\tau + a/2$ and $x-\tau-a$. Both sizes belong to $(a/2,a)$. For $b=0$, a trapezium strip $\calT_y$ with $a/2 < y< a$ contains exactly one deposited triangle. Consequently, $T_x=3$ on an event of strictly positive probability. On the second branch, write $t_\tau = \tau+3a/2-x$. The two child strips are of type $\calP$ and have sizes $t_\tau$ and $x-a/2-t_\tau$. Whenever $x-3a/2 < t_\tau < a$, both of these sizes are strictly smaller than $a$. Since $P_y=0$ for $y<a$ when $b=0$, neither child strip contains an additional triangle. Hence, $T_x=1$ on an event of strictly positive probability. We conclude therefore that $V_{\calT}(x)>0$ for all $3a/2<x<2a$.

Finally, assume that $a=b$. Then the model reduces, after the spatial rescaling $x\mapsto x/a$, to the classical one-dimensional parking problem. Consider $2a<x<3a$. There is a set of first deposition positions of positive Lebesgue measure for which both remaining gaps have size strictly less than $a$, and hence only one trapezium is eventually deposited. There is also a set of first deposition positions of positive Lebesgue measure for which one of the two remaining gaps has size at least $a$, in which case at least one additional trapezium is deposited. Thus $V_{\calP}(x)>0$ for all $ 2a<x<3a$.

We have therefore shown in all cases that there exists a bounded open interval $J$ and a type $s\in\{\calP,\calT\}$, such that $V_s(t)>0$ for all $t \in J$. Consequently, $A_0 \coloneqq \int_JV_s(t)\D t > 0$.

If $s=\calP$, then, for all sufficiently large $x$, the interval $J$ is contained in $[0,x-a]$. Hence~\eqref{eq:variance_lower_rec_P} gives
\[
V_{\calP}(x) \ge \frac{1}{x-a} \int_JV_{\calP}(t)\D t 
=
\frac{A_0}{x-a}>0.
\]
If $s=\calT$, then, for all sufficiently large $x$, the interval $J$ is
contained in $[\delta,x-\varsigma ]$, and therefore
\[
V_{\calP}(x) \ge \frac{1}{x-a} \int_JV_{\calT}(t)\D t
=
\frac{A_0}{x-a}>0.
\]
Thus, in either case, $V_{\calP}(x) \ge A_0/(x-a)>0$ for all sufficiently large $x$. After decreasing the constant if necessary, this yields $V_{\calP}(x)\ge c/x$ for some $c>0$ and all sufficiently large $x$.
\end{proof}

We can now establish the asymptotic behaviour of the variances, and hence prove Theorem~\ref{thm:variance_asymptotics}.

\begin{proof}[Proof of Theorem~\ref{thm:variance_asymptotics}]
By Lemma~\ref{lem:B_second_moment_asymptotics}, there exists a constant $\lambda_2(a,b)\ge0$ such that it holds that $B_{\calP}(x) = \lambda_2(a,b)(x+\varsigma ) + \Oh(\mathcal R_2(x))$ and $B_{\calT}(x) = \lambda_2(a,b)(x+\varsigma ) + \Oh(\mathcal R_2(x))$. By~\eqref{eq:Var_from_B}, we have $\Var(X_s(x)) = B_s(x)-e_s(x)^2$ for $s\in\{\calP,\calT\}$. Lemma~\ref{lem:typewise_mean_error} therefore gives $e_s(x)^2 = \Oh(\mathcal R_1(x)^2)$. A direct comparison of~\eqref{eq:defn_R1_variance_section} and~\eqref{eq:defn_R2_variance_section}, shows that $\mathcal R_1(x)^2 = \oh(\mathcal R_2(x))$. Consequently, $\Var(P_x) = \lambda_2(a,b)(x+\varsigma ) + \Oh(\mathcal R_2(x))$ and $\Var(T_x) = \lambda_2(a,b)(x+\varsigma ) + \Oh(\mathcal R_2(x))$. Since $\mathcal R_2 (x)>0$ for all $x>0$ and $\mathcal{R}_2$ is locally bounded on compact intervals, we can extend this result to hold in the sense that $ \max_{s \in \{\calP,\calT\}}| \Var(X_s(x))-\lambda_2(a,b)(x+(a+b)/2 )| \le C_{a,b} \mathcal R_2(x)$, for all $x>(3a-b)/2$.

It remains only to prove that $\lambda_2(a,b)>0$. Suppose, for contradiction, that $\lambda_2(a,b)=0$. Then, $ \Var(P_x) = \Oh(\mathcal R_2(x))$, and since $\mathcal R_2$ decays super-exponentially, it follows that $\mathcal R_2(x)=\oh(x^{-1})$. It would therefore follow that $ \Var(P_x)=\oh(x^{-1})$. This contradicts Lemma~\ref{lem:positive_variance_lower_bound}, which gives $\Var(P_x)\ge c/x$ for all sufficiently large $x$. Hence, $\lambda_2(a,b)>0$, concluding the proof. 
\end{proof}

\section*{Acknowledgements}

\thanks{
\noindent
The authors would like to thank the Isaac Newton Institute for Mathematical Sciences, Cambridge, for support and hospitality during the programme Stochastic systems for anomalous diffusion, where work on this paper was undertaken. This work was supported by EPSRC grant EP/Z000580/1.}

DKB is supported by AUFF NOVA grant AUFF-E-2022-9-39. S\v{S} is supported by the Croatian Science Foundation, project IP-2022-10-2277. This research was also funded by the European union--NextGenerationEU through the National Recovery and Resilience Plan 2021-2026 Institutional grant of University of Zagreb Faculty of Electrical Engineering and Computing (VALOR). This work was carried out within a project DIGIT.2.1.02.016 funded by the Digital, Innovation, and Green Technology Project – DIGIT Project (IBRD Loan No. 9558‑HR).

\section*{AI disclosure}

During the preparation of this manuscript, the authors used ChatGPT (OpenAI, GPT-5.6) to assist with the development and checking of arguments concerning variance asymptotics, mathematical exposition, and the refinement of selected arguments. All mathematical statements, proofs, references, and conclusions were independently verified by the authors.

\printbibliography

\end{document}